\documentclass[11pt,reqno]{amsart}

\usepackage{amssymb}
\usepackage{mathrsfs,mathtools}
\usepackage{enumerate,enumitem}
\usepackage{xspace}
\usepackage{url}

\usepackage{graphicx}
\usepackage{epstopdf,epsfig,subfigure}

\usepackage{xcolor}
\usepackage{curve2e}
\usepackage{tikz}

\usepackage{verbatim}

\usepackage{algorithm}
\usepackage{algorithmic}
\usepackage[
    bookmarks=true,         
    unicode=false,          
    pdftoolbar=true,        
    pdfmenubar=true,        
    pdffitwindow=false,     
    pdfstartview={FitH},    
    pdftitle={Stabilization of noautonomous parabolic systems},    
    pdfauthor={Author},     
    pdfsubject={Subject},   
    pdfcreator={Creator},   
    pdfproducer={Producer}, 
    pdfkeywords={keyword1, key2, key3}, 
    pdfnewwindow=true,      
    colorlinks=true,       
    linkcolor=red,          
    citecolor=blue,        
    filecolor=magenta,      
    urlcolor=cyan,           
hypertexnames=false
]{hyperref}

\usepackage[numbers,sort&compress]{natbib}

\theoremstyle{definition}
\newtheorem{theorem}{Theorem}[section]
\newtheorem{corollary}[theorem]{Corollary}

\newtheorem{lemma}[theorem]{Lemma}
\newtheorem{definition}[theorem]{Definition}

\newtheorem{remark}[theorem]{Remark}
\newtheorem{assumption}[theorem]{Assumption}

\newtheorem{example}[theorem]{Example}

\numberwithin{equation}{section}

\usepackage{geometry}
\usepackage{scrtime}

\usepackage{todonotes}
\newcommand{\todoautc}[3][]{%
   \ifthenelse{\equal{#1}{}}{\todo[size=\scriptsize]{{\bf#2} #3}{}}{\todo[color=#1,size=\scriptsize]{{\bf#2} #3}{}}%
}

\newcommand{\linspan}{\mathop{\rm span}\nolimits}

\newcommand{\rest}{\left.\kern-2\nulldelimiterspace\right|_}
\newcommand{\norm}[2]{\left|#1\right|_{#2}}
\newcommand{\dnorm}[2]{\left\|#1\right\|_{#2}}

\newcommand{\zero}{{\mathbf0}}
\newcommand{\Id}{{\mathbf1}}
\newcommand{\indf}{1}

\newcommand{\ex}{\mathrm{e}}
\newcommand{\p}{\partial}

\newcommand*{\Bigcdot}{\raisebox{-.25ex}{\scalebox{1.25}{$\cdot$}}}

\newcommand{\clA}{{\mathcal A}}
\newcommand{\clB}{{\mathcal B}}
\newcommand{\clC}{{\mathcal C}}

\newcommand{\clE}{{\mathcal E}}
\newcommand{\clF}{{\mathcal F}}

\newcommand{\clH}{{\mathcal H}}

\newcommand{\clJ}{{\mathcal J}}
\newcommand{\clK}{{\mathcal K}}
\newcommand{\clL}{{\mathcal L}}

\newcommand{\clO}{{\mathcal O}}

\newcommand{\clT}{{\mathcal T}}
\newcommand{\clU}{{\mathcal U}}

\newcommand{\clX}{{\mathcal X}}

\newcommand{\bbN}{{\mathbb N}}

\newcommand{\bbR}{{\mathbb R}}

\newcommand{\bfA}{{\mathbf A}}
\newcommand{\bfB}{{\mathbf B}}
\newcommand{\bfC}{{\mathbf C}}
\newcommand{\bfD}{{\mathbf D}}
\newcommand{\bfE}{{\mathbf E}}
\newcommand{\bfF}{{\mathbf F}}
\newcommand{\bfG}{{\mathbf G}}
\newcommand{\bfH}{{\mathbf H}}

\newcommand{\bfJ}{{\mathbf J}}
\newcommand{\bfK}{{\mathbf K}}
\newcommand{\bfL}{{\mathbf L}}
\newcommand{\bfM}{{\mathbf M}}
\newcommand{\bfN}{{\mathbf N}}

\newcommand{\bfP}{{\mathbf P}}
\newcommand{\bfQ}{{\mathbf Q}}
\newcommand{\bfR}{{\mathbf R}}
\newcommand{\bfS}{{\mathbf S}}

\newcommand{\bfU}{{\mathbf U}}
\newcommand{\bfV}{{\mathbf V}}

\newcommand{\bfX}{{\mathbf X}}
\newcommand{\bfY}{{\mathbf Y}}
\newcommand{\bfZ}{{\mathbf Z}}

\newcommand{\fkB}{{\mathfrak B}}

\newcommand{\fkL}{{\mathfrak L}}

\newcommand{\fkR}{{\mathfrak R}}

\newcommand{\fkT}{{\mathfrak T}}

\newcommand{\rmD}{{\mathrm D}}

\newcommand{\bfc}{{\mathbf c}}

\newcommand{\bfj}{{\mathbf j}}

\newcommand{\bfn}{{\mathbf n}}

\newcommand{\bfu}{{\mathbf u}}
\newcommand{\bfv}{{\mathbf v}}

\newcommand{\rmc}{{\mathrm c}}
\newcommand{\rmd}{{\mathrm d}}
\newcommand{\rme}{{\mathrm e}}
\newcommand{\rmf}{{\mathrm f}}

\newcommand{\rmm}{{\mathrm m}}

\newcommand{\fkh}{{\mathfrak h}}

\newcommand{\ovlineC}[1]{\overline C_{\left[#1\right]}}

\definecolor{DarkBlue}{rgb}{0,0.08,0.45}
\definecolor{DarkRed}{rgb}{.65,0,0}
\definecolor{DarkGreen}{rgb}{.1,.45,.1}
\definecolor{applegreen}{rgb}{0.55, 0.71, 0.0}

\newcounter{mymac@matlab}
\newcommand{\matlab}{MATLAB%
   \ifnum\value{mymac@matlab}<1%
   \textregistered%
   \setcounter{mymac@matlab}{1}%
   \fi%
  }

\newcommand{\black}{ \color{black} }

\begin{document}
\title{Stabilization of nonautonomous linear parabolic-like equations: oblique projections versus Riccati feedbacks}
\author{S\'ergio S.~Rodrigues}

\begin{abstract}
An oblique projections based feedback stabilizability result in the literature is extended to a larger class of reaction-convection terms.
A discussion is presented including a comparison between 
explicit oblique projections based feedback controls and Riccati based feedback controls.
Advantages and limitations of each type of feedback are addressed as well as their finite-elements implementation.
Results of numerical simulations are presented comparing their stabilizing performances for the case of time-periodic dynamics. It is shown that the solution of the periodic Riccati based feedback can be computed iteratively.
\end{abstract}

\thanks{
\vspace{-1em}\newline\noindent
\newline\noindent
{\sc Address}: Johann Radon Institute for Computational and Applied Mathematics,
 \"OAW, Altenbergerstrasse 69, 4040 Linz, Austria.\newline\noindent
{\sc Email}: {\small\tt sergio.rodrigues@ricam.oeaw.ac.at}
}

\maketitle

\pagestyle{myheadings} \thispagestyle{plain} \markboth{\sc S. S.
Rodrigues}{\sc stabilization of nonautonomous linear parabolic equations}


\section{Introduction}
We consider controlled scalar linear parabolic equations as
\begin{subequations}\label{sys-y-parab}
\begin{align}
 &\tfrac{\p}{\p t} y +(-\nu\Delta+\Id) y+ay +b\cdot\nabla y=\textstyle\sum\limits_{j=1}^{M_0}u_j(t)\indf_{\omega_j},\\
  &\fkB y\rest{\p\Omega}=0,\quad
y(0)=y_0.
       \end{align}
 \end{subequations}
The state~$y$ is assumed to be defined in a bounded connected open spatial
subset~$\Omega\in\bbR^d$, with~$d\in\bbN_+\coloneqq\{1,2,\dots\}$ a positive integer. For simplicity, we assume that the domain~$\Omega$ is
 either
smooth or a convex polygon. The temporal interval is the semiline~$\bbR_+\coloneqq (0,+\infty) $. Hence, the state
is a function $y=y(x,t)$, defined for~$(x,t)\in \Omega\times\bbR_+$. The operator~$\fkB$ sets the
conditions
on the boundary~$\p\Omega$ of~$\Omega$,
\begin{align}
  \fkB &=\Id,&&\quad\mbox{for Dirichlet boundary conditions},\notag\\
  \fkB &=\bfn\cdot\nabla=\tfrac{\p}{\p\bfn},&&\quad\mbox{for Neumann boundary conditions,}\notag
\end{align}
where~$\bfn=\bfn(\bar x)$ stands for the outward unit normal vector to~$\p\Omega$, at~$\bar x\in\p\Omega$.
The functions $a=a(x,t)$ and $b=b(x,t)$  are assumed to satisfy
\begin{align}
 &a\in L^\infty(\Omega\times\bbR_+),\qquad b\in L^\infty(\Omega\times\bbR_+)^d.\label{assum.abf.parab}
\end{align}
 The vector function~$u\in L^2(\bbR_+,\bbR^{M_0})$, where~$M_0\in\bbN_+$, is a control input at our disposal, and our actuators are the indicator functions of given open subsets~$\omega_i$,
\begin{equation}\label{indf-w}
\indf_{\omega_i}(x)\coloneqq\begin{cases}
 1,&\mbox{ if }x\in\omega_i,\\
 0,&\mbox{ if }x\in\Omega\setminus\omega_i,
 \end{cases}\qquad \omega_i\subseteq\Omega.
\end{equation}
We assume that the family of actuators is linearly independent,
\begin{equation}\label{setUM}
  U_{M_0}=\linspan\{\indf_{\omega_j}\mid 1\le j\le M_0\},\qquad\clU_{M_0}\coloneqq\linspan U_{M_0},\qquad \dim\clU_{M_0}=M_0.
\end{equation}

In order to shorten the notation we define the spaces
\begin{align}
H^2_\fkB(\Omega)&\coloneqq\{h\in  H^2(\Omega)\mid \fkB h\rest{\p\Omega}=0\},
\mbox{ for }\fkB\in\{\Id,\tfrac{\p}{\p\bfn}\},\notag
\intertext{and}
V_\Id(\Omega)&\coloneqq \{h\in  H^1(\Omega)\mid h\rest{\p\Omega}=0\},\qquad
V_{\frac{\p}{\p\bfn}}(\Omega)\coloneqq H^1(\Omega),\notag
\end{align}
and set the spaces
\begin{equation}\label{parab.Hspaces}
 H\coloneqq L^2(\Omega),\quad V\coloneqq V_{\fkB},\quad\mbox{and}\quad \rmD(A)\coloneqq H^2_\fkB(\Omega),
\end{equation}
and the operators~$A\in\clL(V,V')$  and~$A_{\rm rc}=A_{\rm rc}(t)\in\clL(V,H)$, with
\begin{equation}\label{parab.oper}
\langle Ay,z\rangle_{V',V}\coloneqq\nu(\nabla y,\nabla z)_{(H)^d}+(y,z)_{H},\qquad
\langle A_{\rm rc} y,z\rangle_{V',V}\coloneqq(ay +b\cdot\nabla y,z)_{H}.
\end{equation}

We discuss aspects related to the computation of
linear stabilizing feedback input controls in the form~$u(t)=\clK(t)y(t)\in\bbR^{M_0}$ depending on the state~$y(t)$, at time~$t\ge0$. More precisely, we shall compare explicitly given oblique projection based feedbacks with the classical Riccati based feedbacks. Since the general class of parabolic-like equations considered in~\cite{KunRod19-cocv} does not include~\eqref{sys-y-parab}, we shall first extend the theoretical result in~\cite{KunRod19-cocv}, on stabilizability by means of an oblique projections based feedback, to a more general class of abstract parabolic-like equations,
\begin{align}\label{sys-y-intro-abst}
 \dot y +Ay+A_{\rm rc}y =Bu,\qquad y(0)= y_0,
\end{align}
with the control operator~$B=B_\Phi$ as
\begin{equation}\label{B_Phi}
Bu\coloneqq \textstyle\sum\limits_{j=1}^{M_0}u_j(t)\Phi_j,
\end{equation} 
where~$A$ and~$A_{\rm rc}$ will play, respectively, the roles of~$-\nu\Delta+\Id$  and~$a\Id +b\cdot\nabla$, and where the functions~$\Phi_j$ will play the role of actuators.
Again, we  assume that the family~$\{\Phi_j\mid 1\le j\le {M_0}\}$ of actuators is linearly independent.

So, we shall be looking for time-dependent feedback control operators~$\clK(t)\in\clL(V,\bbR^{M_0})$, giving us the control input~$u(t)=\clK(t)y(t)$ such that the (norm of the) solution of the system
\begin{align}\label{sys-y-intro-K}
 \dot y +Ay+A_{\rm rc}y =B\clK y,\qquad y(0)= y_0,
\end{align}
satisfies, for 
suitable constants~$ \varrho\ge1$ and~$\mu>0$, the inequality
\begin{align}\label{goal-intro}
 \norm{y(t)}{H}\le  \varrho\ex^{-\mu(t-s)}\norm{y(s)}{H},\quad\mbox{for all}\quad t\ge s\ge0,
 \quad\mbox{and all}\quad y_0\in H.
\end{align}
The details shall be given in Theorem~\ref{T:stabOP}.

The discussion on this manuscript is focused on nonautonomous systems. Comparing
Riccati based feedbacks to oblique projections based feedbacks, the former require the computation of the solution of a Riccati equation, while the latter require the computation of a suitable oblique projection~$P_{\clU_{M_0}}^{Y}$ onto the linear span~$\clU_{M_0}$ of the actuators along along a suitable auxiliary closed subspace~$Y$ of~$H$.
\begin{definition}
Let~$X$ and ~$Y$ be closed subspaces of a Hilbert space~$\clH$. We write $\clH=X\oplus Y$ if~$\clH=X+Y$ and~$X\bigcap Y=\{0\}$. If~$\clH=X\oplus Y$,
the (oblique) projection~$P_X^Y\in\clL(\clH,X)\subseteq\clL(\clH)$ in~$\clH$ onto~$X$ along~$Y$ is defined as~$P_X^Yh\coloneqq h_X$, where~$h_X$ is defined by the relations~$(h_X,h_Y)\in X\times Y$ and~$h=h_X+h_Y$.
\end{definition}

The computation of the Riccati feedback in the entire time interval~$[0,+\infty)$ is likely not possible for general nonautonomous systems, hence we shall restrict the comparison to nonautonomous time-periodic dynamics. This brings us to another theoretical contribution of this manuscript where we shall show that the solution of the periodic Riccati equation can be found by an iterative process. The details shall be given in Theorem~\ref{T:convRicPer}.

We focus on general rather
than on specific numerical aspects of each feedback.
We shall see that oblique projection feedbacks are an  interesting alternative to Riccati feedbacks
when there are no restrictions
on the number and location of actuators and on the total energy spent during the stabilization process.

As a motivation, note that the free dynamics of system~\eqref{sys-y-intro-K} (i.e., with~$\clK=\zero$)  can be unstable,  indeed the norm~$\norm{y(t)}{H}$ of its solution may
diverge exponentially to~$+\infty$
as~$t\to+\infty$, for some pairs~$(A,A_{\rm rc})$.
Therefore, we need to look for an input feedback control operator~$\clK$
 in order to achieve stability.
The reason to consider only a finite number~$M_0$ of actuators is motivated mainly by the fact that in real world applications we will likely have only a finite number of actuators at our disposal.

The diffusion-like operator~$A$ is assumed to be independent of time. As we have said, we focus on nonautonomous systems, where~$A+A_{\rm rc}=A+A_{\rm rc}(t)$ is
allowed to be time-dependent.
In the autonomous case, where~$A_{\rm rc}$ is
time-independent,  the spectral properties of the operator~$A+A_{\rm rc}$  can play a crucial role in the
derivation of stabilizability results,~\cite{BarbuTri04}. Such spectral properties are not
an appropriate tool to deal with the nonautonomous case, as shown by the examples in~\cite{Wu74}.
For general nonautonomous systems, in~\cite{BarRodShi11} the spectral arguments are replaced by 
suitable {\em truncated} observability inequalities for the adjoint linear system. This
approach makes direct
use of the exact null controllability of the system (by means of infinite-dimensional controls). More recently,
a different approach is proposed in~\cite{KunRod19-cocv}, using suitable oblique projections
in the Hilbert space~$H$. 
Prior to the theoretical results in these works,  works have been done towards
the development of numerical methods motivated by stabilization of nonautonomous systems;
see~\cite{Kornev06}. 

The problem of stabilization (to zero) of nonautonomous systems appears, for example,
when we want to stabilize the system to a time-dependent trajectory; see~\cite{BarRodShi11}. Similarly,
the problem of stabilization (to zero) of autonomous systems appears, for example,
when we want to stabilize the system to a time-independent trajectory (steady state, equilibrium);
see~\cite{BarbuTri04}.
We refer the reader to~\cite{Lunardi91,BadMitRamRay20}, for works focused on
the stabilization to {\em time-periodic} trajectories (or, the stabilization to zero of time-periodic dynamical systems), a case where
 an interesting argument
 shows that the spectral properties
 of the so called Poincar\'e mapping,
can be used to investigate the stabilizability of the system, and we can use arguments inspired
in the ones used in the autonomous case.

One reason the stabilization to general time-dependent trajectories is important
is that
neither steady states nor time-periodic solutions will exist
if our free dynamical system is subject to nonperiodic time-dependent external forces.
However,
when steady states do exist, then it is natural
to consider their steady
behavior as a desired one, which makes them natural targeted
solutions. This is a reason many works are dedicated to the 
stabilization to these particular trajectories.
We refer the reader to~\cite{BarbuTri04,BarbuLasTri06,Barbu11,RaymThev10,BadTakah11,KrsticMagnVazq08,KrsticMagnVazq09,Lefter09,Barbu_TAC13,BreitenKunisch14,Munteanu12,Munteanu3D12,Raymond19,NgomSeneLeRoux15,Morris11}.

For real world applications, the knowledge of the existence of a stabilizing feedback
input control operator~$\clK$ is not enough,
it is equally important to know how to compute and implement such operators.
The most popular of stabilizing feedbacks is the classical Riccati based one, which allows us to minimize
a classical quadratic cost functional, representing the total energy spent during the stabilization process. 
A considerable number of works have been dedicated to the investigation of
theoretical and numerical aspects of such
feedback, we refer the reader to
\cite{BennerBujaKursSaak20,Heiland16,CurtainPritchard76,UngureanuDragan12,BanKunisch84,BanschBennerSaakWech15,BennerLaubMehrmann97,Benner06,KunkelMehrmann90,Zabczyk75}
and references therein.

The numerical computation of Riccati feedbacks consists in solving a suitable
nonlinear matrix Riccati equation.
Finding such solution is an interesting nontrivial  numerical task. The solution is often found through a
Newton-like iteration, and one particular difficulty relies on the choice of an initial guess
for starting such iteration. We shall recall/propose a strategy to deal with such problem.

The computation of Riccati feedbacks becomes more 
expensive (e.g., time consuming) as the size of the matrix increases,
and can become unfeasible for
accurate finite-elements approximations of parabolic equations.
Ways to circumvent this fact  can be either the use of an appropriate model reduction, or
to compute it in a coarser mesh, or simply to look for alternative  feedbacks.
Here we shall consider the last two approaches (the first one of which can also be seen as a simple model reduction). An alternative to Riccati, including the case of general nonautonomous systems, is the explicit feedback  introduced in~\cite{KunRod19-cocv},
which involves a suitable oblique projection operator, and whose numerical implementation requires essentially
the discretization of such projection.

Though we focus on Riccati and oblique projections feedbacks,
we would like to briefly mention other feedbacks. Namely, the feedback
in~\cite{Barbu_TAC13}, also presented as an alternative
to Riccati, in the context of stabilization of {\em autonomous} systems
by means of {\em boundary} controls (see~\cite{HalanayMureaSafta13} for related simulations); the
feedback in~\cite{AzouaniTiti14,LunasinTiti17},  directly exploiting
the existence of suitable determining parameters (e.g., nodes and Fourier modes)
for parabolic-like equations, see also~\cite{CockburnJonesTiti97};  the backstepping approach
in~\cite{KrsticMagnVazq08} for boundary controls, in
\cite{WangWoittennek13,WoittennekWangKnueppel14} for controls on transmission conditions,
and in~\cite{TsubakinoKrsticHara12} for internal controls with a particular shape/profile.

\bigskip
\noindent
{\em Contents.}
The rest of the paper is organized as follows. In section~\ref{S:Stabil} we present the theory involved in the construction of an explicit stabilizing feedback based on oblique projections, prove the main theoretical results, and recall the classical Riccati feedback. In section~\ref{S:NumerImpl} we  discuss a finite-elements
numerical implementation for both Riccati  and oblique projections feedbacks.
Numerical simulations are presented in section~\ref{S:LocAct}, and concluding remarks
are gathered in sections~\ref{S:remarks} and~\ref{S:conclusion}. 

\section{Stabilizability of nonautonomous parabolic-like equations}\label{S:Stabil}
The abstract  form~\eqref{sys-y-intro-K} for~\eqref{sys-y-parab} (with a feedback control) lies in a class of controlled linear parabolic-like systems as
\begin{align}\label{sys-y-F-abst}
 \dot y +Ay+A_{\rm rc}y =\clF y,\qquad y(0)= y_0\in H,
\end{align}
under general assumptions on the operators~$A$ and~$A_{\rm rc}$, where we have written~$\clF=B\clK$. Note that since~$B\in\clL(\bbR^{M_0},\clU_{M_0})$ is an isomorphism, looking for the input control feedback operator~$\clK\in\clL(V,\bbR^{M_0})$ is equivalent to looking for the feedback operator~$\clF\in\clL(V,\clU_{M_0})$. Here~$\clU_{M_0}$ is the space spanned by the family~$U_{M_0}$ of linearly independent actuators (cf.~\eqref{setUM} and~\eqref{B_Phi})
\begin{equation}\label{setUM-Phi}
  U_{M_0}=\linspan\{\Phi_j\mid 1\le j\le {M_0}\}\subset H,\quad\clU_{M_0}\coloneqq\linspan U_{M_0},\quad \dim\clU_{M_0}={M_0}.
\end{equation}

For a given subset~$S\subset \clH$ of a separable real Hilbert space~$\clH$, we denote the orthogonal complement of~$S$ by~
\[
S^{\perp\clH}\coloneqq\{h\in\clH\mid (h,s)_\clH=0\mbox{ for all }s\in S\}.
\]
For simplicity, in the case~$\clH=H$ is our pivot space, we denote~$S^\perp\coloneqq S^{\perp H}$.

\subsection{Assumptions}
Hereafter the evolution of~\eqref{sys-y-F-abst} is considered in a pivot Hilbert space~$H$,  $H=H'$. 
All Hilbert spaces are assumed real and separable.

\begin{assumption}\label{A:A0sp}
$V\subset H$ is a Hilbert space, $A\in\clL(V,V')$ is symmetric, and $(y,z)\mapsto\langle Ay,z\rangle_{V',V}$ is a
 complete scalar product on~$V.$
\end{assumption}

Hereafter, $V$ is endowed with the scalar product~$(y,z)_V\coloneqq\langle Ay,z\rangle_{V',V}$,
which again makes~$V$ a Hilbert space.
Necessarily, $A\colon V\to V'$ is an isometry.
\begin{assumption}\label{A:A0cdc}
The inclusion $V\subseteq H$ is dense, continuous, and compact.
\end{assumption}

Necessarily, we have that
\[
 \langle y,z\rangle_{V',V}=(y,z)_{H},\quad\mbox{for all }(y,z)\in H\times V,
\]
and also that the operator $A$ is densely defined in~$H$, with domain $\rmD(A)$ satisfying
\[
\rmD(A)\xhookrightarrow{\rm d,\,c} V\xhookrightarrow{\rm d,\,c} H\xhookrightarrow{\rm d,\,c} V'\xhookrightarrow{\rm d,\,c}\rmD(A)'.
\]
Further,~$A$ has a compact inverse~$A^{-1}\colon H\to H$, and we can find a nondecreasing
system of (repeated accordingly to their multiplicity) eigenvalues $(\alpha_n)_{n\in\bbN_+}$ and a  basis of
eigenfunctions $(e_n)_{n\in\bbN_+}$ as
\begin{equation}\label{eigfeigv}
0<\alpha_1\le\alpha_2\le\dots\le\alpha_n\to+\infty, \quad Ae_n=\alpha_n e_n.
\end{equation}

 We can define, for every $\zeta\in\bbR$, the fractional powers~$A^\zeta$, of $A$, by
 \[
  y=\sum_{n=1}^{+\infty}y_ne_n,\quad A^\zeta y=A^\zeta \sum_{n=1}^{+\infty}y_ne_n\coloneqq\sum_{n=1}^{+\infty}\alpha_n^\zeta y_n e_n,
 \]
 and the corresponding domains~$\rmD(A^{|\zeta|})\coloneqq\{y\in H\mid A^{|\zeta|} y\in H\}$, and
 $\rmD(A^{-|\zeta|})\coloneqq \rmD(A^{|\zeta|})'$.
 We have that~$\rmD(A^{\zeta})\xhookrightarrow{\rm d,\,c}\rmD(A^{\zeta_1})$, for all $\zeta>\zeta_1$,
 and we  see that~$\rmD(A^{0})=H$, $\rmD(A^{1})=\rmD(A)$, $\rmD(A^{\frac{1}{2}})=V$.

\begin{assumption}\label{A:A1}
For almost every~$t>0$ we have~$A_{\rm rc}(t)\in\clL(H, V')+\clL(V, H)$,
and we have a uniform bound, that is, $\norm{A_{\rm rc}}{L^\infty(\bbR_0,\clL(H,V')+\clL(V, H))}\eqqcolon C_{\rm rc}<+\infty.$
\end{assumption}

Below, it is convenient to consider the number of actuators as~$M_0=\sigma(M)$, as a term of a subsequence of positive integers.

\begin{assumption}\label{A:poincare}
There exists a sequence~$(U_{\sigma(M)},E_{\sigma(M)})_{M\in\bbN_+}$, where for each~$M$,
\[
{U_{\sigma(M)}}=\{\Phi_{M,j}\mid 1\le j\le {\sigma(M)}\}\subset  H
\] is a set of actuators and
\[
{E_{\sigma(M)}}=\{e_{M,j}\mid 1\le j\le \sigma(M)\}\subset V
\]
is a set of auxiliary eigenfunctions, satisfying the following:
\begin{enumerate}[noitemsep,topsep=5pt,parsep=5pt,partopsep=0pt,leftmargin=3em]%
\renewcommand{\theenumi}{\ref{A:poincare}({\roman{enumi}})} %
\renewcommand{\labelenumi}{({\roman{enumi}})}%
\item\label{A:poincaresigma}
$\sigma\colon\bbN_+\to\bbN_+$
is a strictly increasing function,
\item\label{A:poincareDS}
$H=\clU_{\sigma(M)}\oplus\clE_{\sigma(M)}^\perp$, for all~$M\in\bbN_+$, with~$ \clU_{\sigma(M)}\coloneqq\linspan{U_{\sigma(M)}}$, $\clE_{\sigma(M)}\coloneqq\linspan{E_{\sigma(M)}}$, and  $\dim\,{\clU_{\sigma(M)}}=  \dim\,{\clE_{\sigma(M)}}={\sigma(M)}$,
\item\label{A:poincareOP}
we have that~$\sup\limits_{M\in\bbN_+}\norm{P_{\clU_{\sigma(M)}}^{\clE_{\sigma(M)}^\perp}}{\clL(H)}\eqqcolon C_P<+\infty$,
\item\label{A:poincarexi}
defining, for each~$M\in\bbN_+$,
the  Poincar\'e-like constants
\[
\overline{\xi}_M\coloneqq\inf_{\varTheta\in (V\bigcap\clE_{\sigma(M)}^\perp)\setminus\{0\}}
\tfrac{\norm{\varTheta}{V}^2}{\norm{\varTheta}{H}^2}\quad\mbox{and}\quad
\underline{\xi}_M\coloneqq\sup_{\theta\in (V\bigcap\clE_{\sigma(M)})\setminus\{0\}}
\tfrac{\norm{\theta}{V}^2}{\norm{\theta}{H}^2},
\]
we have that
$
\lim\limits_{M\to+\infty}\overline{\xi}_M=+\infty\quad\mbox{and}\quad\sup\limits_{M\in\bbN_+}\overline{\xi}_M^{-1}\underline{\xi}_M\eqqcolon C_\clE<+\infty.
$
\end{enumerate}
\end{assumption}

\begin{remark}
Assumption~\ref{A:A1} is weaker than its analogous in~\cite[Assum.~2.3]{KunRod19-cocv} where it was considered~$A_{\rm rc}(t)\in\clL(H, V')$. Note that, if ~$b(\Bigcdot,t)\in L^\infty(\Omega)$ (cf.~\eqref{assum.abf.parab}) with~$\nabla\cdot b$ not regular enough, then~$y\mapsto b(\Bigcdot,t)\cdot\nabla y$ (cf.~\eqref{sys-y-parab}) is in~$\clL(V, H)\setminus\clL(H, V')$.
\end{remark}

\subsection{Oblique projection stabilizing feedbacks}
Here, we construct an explicit stabilizing feedback. We show that the explicit feedback proposed in~\cite{KunRod19-cocv} can be used for more general reaction-convection terms as in Assumption~\ref{A:A1}.
\begin{lemma}\label{L:xiAV'}
For~$\varTheta\in \clE_{\sigma(M)}^\perp\setminus\{0\}$ and~$\theta\in\clE_{\sigma(M)}\setminus\{0\}$, we have the relations $\tfrac{\norm{\varTheta}{H}^2}{\norm{\varTheta}{V'}^2}\ge\overline{\xi}_M$ and $
\tfrac{\norm{\theta}{H}^2}{\norm{\theta}{V'}^2}\le\underline{\xi}_M.$
\end{lemma}
\begin{proof}
Note that
$A^{-\frac12}\clE_{\sigma(M)}^\perp\subseteq V\bigcap\clE_{\sigma(M)}^\perp$ and~$A^{-\frac12}\clE_{\sigma(M)}\subseteq V\bigcap\clE_{\sigma(M)}$. The definitions of~$\overline{\xi}_M$ and~$\underline{\xi}_M$ give us~$\tfrac{\norm{\varTheta}{H}^2}{\norm{\varTheta}{V'}^2}=\tfrac{\norm{A^{-\frac12}\varTheta}{V}^2}{\norm{A^{-\frac12}\varTheta}{H}^2}\ge\overline{\xi}_M$ and~$\tfrac{\norm{\theta}{H}^2}{\norm{\theta}{V'}^2}=\tfrac{\norm{A^{-\frac12}\theta}{V}^2}{\norm{A^{-\frac12}\theta}{H}^2}\le\underline{\xi}_M$.
\end{proof}

\begin{lemma}\label{L:extOP}
Let~$H=\clU\oplus \clE^{\perp}$, with~$\clU=\clU_{\sigma(M)}$ and~$\clU=\clE_{\sigma(M)}$ being finite-dimensional spaces as in Assumption~\ref{A:poincare}. Then, $V'=\clU\oplus \clE^{\perp V'}$ and the oblique projection~$P_{\clU}^{\clE^{\perp V'}}\in\clL(V')$ is an extension of~$P_{\clU}^{\clE^{\perp}}\in\clL(H)$. Furthermore,  we have
\begin{align}\notag
&P_{\clE}^{\clE^{\perp V'}}\!\!=P_{\clE}^{\clE^{\perp V'}}\!P_{\clU}^{\clE^{\perp V'}},\quad P_{\clU}^{\clE^{\perp V'}}\!\!=P_{\clU}^{\clE^{\perp V'}}\!P_{\clE}^{\clE^{\perp V'}},
\quad\mbox{and}\quad P_{\clE^{\perp V'}}^{\clU}=P_{\clE^{\perp V'}}^{\clE}+P_{\clE^{\perp V'}}^{\clU}P_{\clE}
^{\clE^{\perp V'}}\!.
\end{align}
\end{lemma}
\begin{proof}
The first statements are shown in~\cite[Lems.~3.2 and 3.3]{KunRod19-cocv}. We show  the last identities (stated in~\cite[Lem.~3.4]{KunRod19-cocv} without proof).
For an arbitrary~$f\in V'$,
\begin{align}
&P_{\clE}^{\clE^{\perp V'}}f=P_{\clE}^{\clE^{\perp V'}}\left(P_{\clU}^{\clE^{\perp V'}}f+P_{\clE^{\perp V'}}^{\clU}f\right)=P_{\clE}^{\clE^{\perp V'}}P_{\clU}^{\clE^{\perp V'}}f,\notag\\
&P_{\clU}^{\clE^{\perp V'}}f=P_{\clU}^{\clE^{\perp V'}}\left(P_{\clE}^{\clE^{\perp V'}}f+P_{\clE^{\perp V'}}^{\clE}f\right)=P_{\clU}^{\clE^{\perp V'}}P_{\clE}^{\clE^{\perp V'}}f,\notag\\
&P_{\clE^{\perp V'}}^{\clU}f=P_{\clE^{\perp V'}}^{\clU}\left(
P_{\clE^{\perp V'}}^{\clE}f+P_{\clE}^{\clE^{\perp V'}}f\right)=P_{\clE^{\perp V'}}^{\clE}f+P_{\clE^{\perp V'}}^{\clU} P_{\clE}^{\clE^{\perp V'}}f,\notag
\end{align}
which finishes the proof.
\end{proof}

The next result extends the result given in~\cite{KunRod19-cocv} for~$A_{\rm rc}\in\clL(H,V')$  to the case~$A_{\rm rc}\in\clL(V,H)+\clL(H,V')$. It also considers the case where~$\clE_{\sigma(M)}$ is not necessarily spanned by the first eigenfunctions of~$A$ as in~\cite{KunRod19-cocv}, that is, we present the details for the claim in~\cite[Rem.~3.9]{KunRod19-cocv}. See also~\cite[sect.~4.8]{KunRod19-cocv} for examples where the appropriately chosen~$\clE_{\sigma(M)}$s are not spanned by the first eigenfunctions of~$A$.

Hereafter~$\ovlineC{a_1,a_2,\dots,a_k}$, $k\in\bbN_+$, stands for a constant that increases with each of its nonnegative arguments~$a_i$, $1\le i\le k$.
\begin{theorem}\label{T:stabOP}
Let Assumptions~\ref{A:A0sp}--\ref{A:poincare} hold true and let~$\mu>0$ and~$\gamma>1$. Then, there exists $M_*=\ovlineC{C_\clE,C_{\rm rc},C_P,\gamma\mu}$ such that for all~$M\ge M_*$ and all~$\lambda\ge\mu$ the weak solution~$y$ of system~\eqref{sys-y-F-abst}, with 
\begin{align}\label{FeedKunRod}
 \clF=\clF^{\rm obli}\coloneqq P_{\clU_{\sigma(M)}}^{\clE_{\sigma(M)}^{\perp V'}}
 \Bigl(A+A_{\rm rc}-\lambda\Id\Bigr),
 \end{align}
satisfies  inequality~\eqref{goal-intro} with~$\varrho=\ovlineC{C_\clE,C_{\rm rc},C_P,\underline{\xi}_M,\overline{\xi}_M^{-1}\lambda^2,\left(\max\{\gamma\mu,\lambda\}-\mu\right)^{-1}}$.
\end{theorem}
\begin{proof}
Firstly, for the orthogonal component~$z\coloneqq P_{\clE_{\sigma(M)}}^{\clE_{\sigma(M)}^\perp}y=P_{\clE_{\sigma(M)}}^{\clE_{\sigma(M)}^{\perp V'}}y$, we find
\begin{equation}\label{dyn-fin}
\dot z=-\lambda z,
\end{equation}
thus~$z$ is stable. Let us fix~$M$ and denote, for simplicity, $\clU\coloneqq\clU_{\sigma(M)}$ and $\clE\coloneqq\clE_{\sigma(M)}$. The complementary component~$Z\coloneqq P_{\clE^\perp}^{\clE}y=P_{\clE^{\perp V'}}^{\clE}y$ satisfies
\begin{align}
\dot Z&=-P_{\clE^{\perp V'}}^{\clE}P_{\clE^{\perp V'}}^{\clU}\left(Ay+A_{\rm rc}y\right)-\lambda P_{\clE^{\perp V'}}^{\clE}P_{\clU}^{\clE^{\perp V'}}y\notag\\
&=-AZ-P_{\clE^{\perp V'}}^{\clU}A_{\rm rc}Z-p,\notag
\\
\mbox{with}\qquad p&\coloneqq P_{\clE^{\perp V'}}^{\clU}(Az+A_{\rm rc}z)+\lambda P_{\clE^{\perp V'}}^{\clE}P_{\clU}^{\clE^{\perp V'}}z\notag
\end{align}
and, multiplying (testing) the dynamics equation with~$2Z$ leads us to\black
\begin{align}\label{dtZ1}
\tfrac{\rmd}{\rmd t}\norm{Z}{H}^2&=-2\norm{Z}{V}^2-2\left\langle P_{\clE^{\perp V'}}^{\clU}A_{\rm rc}Z+p,Z\right\rangle_{V',V}.
\end{align}

Let us fix an arbitrary $(A_{\rm rc1},A_{\rm rc2})\in\clL(V,H)\times\clL(H,V')$ so that~$A_{\rm rc}=A_{\rm rc1}+A_{\rm rc2}$.
For the reaction-convection term, with~$(w,Z)\in V\times(V\bigcap \clE_M^\perp)$ we obtain
\begin{subequations}\label{react-est1}
\begin{align}
&-\left\langle P_{\clE^{\perp V'}}^{\clU}A_{\rm rc1}w,Z\right\rangle_{V',V}
\le\norm{P_{\clE^{\perp}}^{\clU}}{\clL(H)}\norm{A_{\rm rc1}}{\clL(V,H)}\norm{w}{V}\norm{Z}{H},\\
&-\left\langle P_{\clE^{\perp V'}}^{\clU}A_{\rm rc2}w,Z\right\rangle_{V',V}
=-\left\langle P_{\clE^{\perp V'}}^{\clE}A_{\rm rc2}w,Z\right\rangle_{V',V}-\left\langle P_{\clE^{\perp V'}}^{\clU}P_{\clE}^{\clE^{\perp V'}}A_{\rm rc2}w,Z\right\rangle_{V',V}\notag\\
&\hspace{2em}\le\norm{A_{\rm rc2}}{\clL(H,V')}\norm{w}{H}\norm{Z}{V}+\norm{P_{\clE^{\perp V'}}^{\clU}P_{\clE}^{\clE^{\perp V'}}}{\clL(V')}\norm{A_{\rm rc2}}{\clL(H,V')}\norm{w}{H}\norm{Z}{V},\\
&\norm{P_{\clE^{\perp V'}}^{\clU}P_{\clE}^{\clE^{\perp V'}}}{\clL(V')}=
\norm{P_{\clE^\perp}^{\clE}P_{\clE^{\perp}}^{\clU}P_{\clE}^{\clE^{\perp}}P_{\clE}^{\clE^{\perp V'}}}{\clL(V')}
\notag\\
&\hspace{2em}
\le\norm{P_{\clE^{\perp}}^{\clE}}{\clL(H,V')}\norm{P_{\clE^{\perp}}^{\clU}}{\clL(H)}\norm{P_{\clE}^{\clE^{\perp}}}{\clL(V',H)}\le(\overline\xi_{M})^{-\frac12}\underline{\xi}_M^{\frac12}\norm{P_{\clE^{\perp}}^{\clU}}{\clL(H)}\notag\\
&\hspace{2em}
\le C_\clE^{\frac12} (1+C_P),
\end{align}
\end{subequations}
with~$C_\clE$ and~$C_P$ as in Assumption~\ref{A:poincare}, and where we have used Lemma~\ref{L:xiAV'}. Hence, from~\eqref{react-est1} with $w=Z$ and the Young inequality it follows that
\begin{align}\notag
&-2\left\langle P_{\clE^{\perp V'}}^{\clU}A_{\rm rc}Z,Z\right\rangle_{V',V}
\le\tfrac12\norm{Z}{V}^2+\ovlineC{C_\clE, C_P}(\norm{A_{\rm rc1}}{\clL(V,H)}^2+\norm{A_{\rm rc2}}{\clL(H,V')}^2)\norm{Z}{H}^2,
\end{align}
and taking the infimum over the pair~$(A_{\rm rc1},A_{\rm rc2})$ it follows that
\begin{align}\notag
-2\left\langle P_{\clE^{\perp V'}}^{\clU}A_{\rm rc}Z,Z\right\rangle_{V',V}
&\le\tfrac12\norm{Z}{V}^2+\ovlineC{C_\clE, C_P}\norm{A_{\rm rc}}{\clL(V,H)+\clL(H,V')}^2\norm{Z}{H}^2\\
&\le\tfrac12\norm{Z}{V}^2+\ovlineC{C_\clE, C_P,C_{\rm rc}}\norm{Z}{H}^2,\label{react-est2}
\end{align}
with~$C_{\rm rc}$ as in Assumption~\ref{A:A1}.
Therefore, from~\eqref{dtZ1} and~\eqref{react-est2} we find that
\begin{align}\label{dtZ2}
\tfrac{\rmd}{\rmd t}\norm{Z}{H}^2&\le-\tfrac32\norm{Z}{V}^2+\ovlineC{C_\clE, C_P,C_{\rm rc}}\norm{Z}{H}^2-2\left\langle p,Z\right\rangle_{V',V}.
\end{align}

Next, we observe that
\begin{subequations}\label{est-pZ}
\begin{align}
&-2\left\langle p,Z\right\rangle_{V',V}=-2\left\langle P_{\clE^{\perp V'}}^{\clU}(Az+A_{\rm rc}z)+\lambda P_{\clE^{\perp V'}}^{\clE}P_{\clU}^{\clE^{\perp V'}}z,Z\right\rangle_{V',V},\\
&-2\left\langle P_{\clE^{\perp V'}}^{\clU}Az,Z\right\rangle_{V',V}=-2\left\langle P_{\clE^{\perp}}^{\clU}Az,Z\right\rangle_{V',V}\notag\\
&\hspace{4em}\le2\overline{\xi}_M^{-\frac12}(1+C_P)\underline{\xi}_M\norm{z}{H}\norm{Z}{V},\\
&-2\left\langle \lambda P_{\clE^{\perp V'}}^{\clE}P_{\clU}^{\clE^{\perp V'}}z,Z\right\rangle_{V',V}=-2\left\langle \lambda P_{\clE^{\perp}}^{\clE}P_{\clU}^{\clE^{\perp}}z,Z\right\rangle_{V',V}\notag\\
&\hspace{4em}\le2\lambda\overline{\xi}_M^{-\frac12}C_P\norm{z}{H}\norm{Z}{V}
\intertext{and, using~\eqref{react-est1} with~$w=z$,}
&-2\left\langle P_{\clE^{\perp V'}}^{\clU}A_{\rm rc}z,Z\right\rangle_{V',V}\notag\\
&\hspace{4em}\le2\ovlineC{C_P}\norm{A_{\rm rc1}}{\clL(V,H)}\norm{z}{V}\norm{Z}{H}+
\ovlineC{C_P,C_\clE}\norm{A_{\rm rc2}}{\clL(H,V')}\norm{z}{H}\norm{Z}{V}\notag\\
&\hspace{4em}\le\ovlineC{C_P,C_\clE}(\norm{A_{\rm rc1}}{\clL(V,H)}+\norm{A_{\rm rc2}}{\clL(H,V')})\norm{z}{H}\norm{Z}{V}.
\end{align}
\end{subequations}
Thus, from~\eqref{est-pZ} and the Young inequality it follows that
\begin{align}
-2\left\langle p,Z\right\rangle_{V',V}
\le\tfrac12\norm{Z}{V}^2+\ovlineC{C_P,C_\clE,\underline{\xi}_M,\lambda\overline{\xi}_M^{-\frac12}}\norm{z}{H}^2,\notag
\end{align}
which together with~\eqref{dtZ2} give us
\begin{align}
\tfrac{\rmd}{\rmd t}\norm{Z}{H}^2&\le-\norm{Z}{V}^2+\ovlineC{C_\clE, C_P,C_{\rm rc}}\norm{Z}{H}^2+\ovlineC{C_\clE, C_P,C_{\rm rc},\underline{\xi}_M,\overline{\xi}_M^{-1}\lambda^2}\norm{z}{H}^2\label{dtZ3}\\
&\le-(\overline{\xi}_M-\ovlineC{C_\clE, C_P,C_{\rm rc}})\norm{Z}{H}^2+\ovlineC{C_\clE, C_P,C_{\rm rc},\underline{\xi}_M,\overline{\xi}_M^{-1}\lambda^2}\norm{z}{H}^2.\notag
\end{align}

Now, due to Assumption~\ref{A:poincare}, we can choose $M_*\in\bbN_+$ such that
\[\tfrac12(\overline{\xi}_M-\ovlineC{C_\clE, C_P,C_{\rm rc}})\eqqcolon\mu_*>\gamma\mu,\quad\mbox{for all}\quad M\ge M_*,\]
which gives us
\begin{align}\notag
\tfrac{\rmd}{\rmd t}\norm{Z}{H}^2
&\le-2\gamma\mu\norm{Z}{H}^2+\ovlineC{C_\clE, C_P,C_{\rm rc},\underline{\xi}_M,\overline{\xi}_M^{-1}\lambda^2}\norm{z}{H}^2,\quad\mbox{for all}\quad M\ge M_*.
\end{align}
Setting also~$\lambda\ge\mu$, by Duhamel formula and recalling~\eqref{dyn-fin}, it follows that
\begin{align}
&\norm{Z(t)}{H}^2
\le\rme^{-2\gamma\mu(t-s)}\norm{Z(s)}{H}^2+\int_s^t\rme^{-2\gamma\mu(t-\tau)}\ovlineC{C_\clE, C_P,C_{\rm rc},\underline{\xi}_M,\overline{\xi}_M^{-1}\lambda^2}\norm{z(\tau)}{H}^2\,\rmd\tau\notag\\
&\hspace{2em}\le\rme^{-2\gamma\mu(t-s)}\norm{Z(s)}{H}^2+\ovlineC{C_\clE, C_P,C_{\rm rc},\underline{\xi}_M,\overline{\xi}_M^{-1}\lambda^2}\norm{z(s)}{H}^2\int_s^t\rme^{-2\gamma\mu(t-\tau)}\rme^{-2\lambda(\tau-s)}\,\rmd\tau.\notag
\end{align}
Using~\cite[Prop.~3.2]{AzmiRod20} to estimate the integral term, we find 
\begin{align}
\norm{Z(t)}{H}^2
&\le\rme^{-2\gamma\mu(t-s)}\norm{Z(s)}{H}^2+\ovlineC{C_\clE, C_P,C_{\rm rc},\underline{\xi}_M,\overline{\xi}_M^{-1}\lambda^2,\frac{1}{\max\{\gamma\mu,\lambda\}-\mu}}\norm{z(s)}{H}^2\rme^{-2\mu(t-s)},\notag
\end{align}
which allows us to conclude that
\begin{align}
\norm{y(t)}{H}^2&=\norm{z(t)}{H}^2+\norm{Z(t)}{H}^2\le \varrho\rme^{-2\mu(t-s)}\norm{y(s)}{H}^2,\quad\mbox{for all}\quad M\ge M_*,\quad\lambda\ge\mu,\notag
\end{align}
with~$\varrho=\ovlineC{C_\clE, C_P,C_{\rm rc},\underline{\xi}_M,\overline{\xi}_M^{-1}\lambda^2,\frac{1}{\max\{\gamma\mu,\lambda\}-\mu}}$.
That is, we have stability with exponential rate~$\mu>0$, for large enough~$M$ and~$\lambda$.
\end{proof}
Assumptions~\ref{A:A0sp}--\ref{A:A1} are satisfied for systems~\eqref{sys-y-parab} with the spaces and operators in~\eqref{parab.Hspaces} and~\eqref{parab.oper}, with~\eqref{assum.abf.parab}.
Examples of sequences~$(\clU_M,\clE_M)_{M\in\bbN_+}$ satisfying Assumption~\ref{A:poincaresigma}--\ref{A:poincareOP} are given in~\cite[Thms.~2.1 and~2.3]{RodSturm20} for equations evolving in one-dimensional spatial domain~$(0,L)$, namely, for a given~$r\in(0,1)$ we can take $\clU_M$ as the span of the actuators~$\indf_{\omega_{j}^M}$ whose supports are the intervals
\begin{equation}\label{loc-actOP1D}
\omega_j^M=\left(\tfrac{(2j-1)L}{2M}-\tfrac{rL}{2M},\tfrac{(2j-1)L}{2M}+\tfrac{rL}{2M}\right)\subset(0,L),\qquad 1\le j\le M,
\end{equation}
and we can take $\clE_M$ as the span of the first eigenfuntions of the Laplacian (for both Dirichlet and Neumann boundary conditions).
Examples for equations evolving in higher-dimensional rectangular domains~$\Omega\subset\bbR^d$ are given in~\cite[sect.~4.8.1]{KunRod19-cocv}, by  taking Cartesian products of those one-dimensional $\clU_M$ and~$\clE_M$, which correspond to take~$M_\sigma=M^d$ actuators (cf. Assumption~\ref{A:poincare}).
Finally, in~\cite[sect.~2.2]{Rod20-eect}, it is shown that those Cartesian products also satisfy Assumption~\ref{A:poincarexi}.
\begin{remark}
If~$A_{\rm rc}\in\clL(V,H)\subset \clL(V,H)+\clL(H,V')$ we do not need to assume the uniform bound~$C_\clE$ for~$\overline{\xi}_M^{-1}\underline{\xi}_M$ in Assumption~\ref{A:poincarexi}. Further, following the proof above we would obtain~$
\norm{y(t)}{H}^2\le\ovlineC{C_P,C_{\rm rc},\frac{1}{\max\{\gamma\mu,\lambda\}-\mu},\underline{\xi}_M,\overline{\xi}_M^{-1}\lambda^2}\rme^{-2\mu(t-s)}\norm{y(s)}{H}^2$, for all $M\ge M_*$ and $\lambda\ge\mu.$ Reaction-convection terms as~$A_{\rm rc}\in\clL(V,H)$ have been considered in~\cite{Rod20-eect} for stabilization of strong solutions of semilinear equations with initial states~$y_0\in V$. Here we consider stabilization of  weak solutions of linear systems with initial states in a larger space;~$y_0\in H\supset V$.
\end{remark}

\begin{corollary}\label{C:stabOP}
Let Assumptions~\ref{A:A0sp}--\ref{A:poincare} hold true and let~$M_*\in \bbN_*$ and~$\lambda\ge\mu$ be as in Theorem~\ref{T:stabOP}. Then for any~$\overline\mu <\mu$ we have that
\begin{align}
&\norm{\rme^{(\Bigcdot-s)\overline\mu }y}{L^2((s,+\infty),H)}^2\le\ovlineC{C_\clE, C_P,C_{\rm rc},\underline{\xi}_M,\overline{\xi}_M^{-1}\lambda^2,\frac{1}{\max\{\gamma\mu,\lambda\}-\mu},\frac1{\mu-\overline\mu }}\norm{y(s)}{H}^2,\notag\\
&\norm{\rme^{(\Bigcdot-s)\overline\mu }\clF^{\rm obli}y}{L^2((s,+\infty),H)}^2
\le\ovlineC{C_\clE, C_P,C_{\rm rc},\underline{\xi}_M,\overline{\xi}_M^{-1}\lambda^2,\frac{1}{\max\{\gamma\mu,\lambda\}-\mu},\frac1{\mu-\overline\mu },\lambda,\alpha_1^{-1},\overline\mu}\norm{y(s)}{H}^2,\notag
\end{align}
for all~$s\ge0$.
\end{corollary}
\begin{proof}
For the state, with~$\bbR_{+s}\coloneqq(s,+\infty)$ for~$s\ge0$, we find that
\begin{align}
\norm{\rme^{(t-s)\overline\mu }y}{L^2(\bbR_{+s},H)}^2&\le\varrho\norm{y(s)}{H}^2\norm{\rme^{-(\mu-\overline\mu )(t-s)}}{L^2(\bbR_{+s},\bbR)}^2\le\varrho\tfrac1{2(\mu-\overline\mu )}\norm{y(s)}{H}^2\label{expy-bound}
\end{align}
and, for the feedback control,  denoting again $\clU\coloneqq\clU_{\sigma(M)}$ and $\clE\coloneqq\clE_{\sigma(M)}$, we find
\begin{subequations}\label{normFeed1}
\begin{align}
& \norm{P_{\clU}^{\clE^{\perp V'}}
 \Bigl(A+A_{\rm rc}-\lambda\Id\Bigr)y}{H}\le  \norm{\Id\rest{\clU}}{\clL(H,V')}\norm{P_{\clU}^{\clE^{\perp}}}{\clL(H)}
\norm{P_{\clE}^{\clE^{\perp V'}}
 \Bigl(A+A_{\rm rc}-\lambda\Id\Bigr)y}{H}\notag\\
&\hspace{4em}\le  \alpha_1^{-\frac12}C_P
\norm{P_{\clE}^{\clE^{\perp V'}}
 \Bigl(A+A_{\rm rc}-\lambda\Id\Bigr)y}{H}
\intertext{and, with~$z=P_{\clE}^{\clE^{\perp}}y$,}
&\norm{P_{\clE}^{\clE^{\perp V'}}
 \Bigl(A-\lambda\Id\Bigr)y}{H}\le(\underline{\xi}_M+\lambda)\norm{z}{H}\le (\underline{\xi}_M+\lambda)\alpha_1^{-\frac12}\norm{z}{V}
\intertext{and, for an arbitrary $(A_{\rm rc1},A_{\rm rc2})\in\clL(V,H)\times\clL(H,V')$ so that~$A_{\rm rc}=A_{\rm rc1}+A_{\rm rc2}$,}
&\norm{P_{\clE}^{\clE^{\perp V'}}
 A_{\rm rc}y}{H}\le \norm{A_{\rm rc1}y}{H}+\underline{\xi}_M^\frac12\norm{A_{\rm rc2}y}{V'}
\le\norm{A_{\rm rc1}}{\clL(V,H)}\norm{y}{V}+\underline{\xi}_M^\frac12\norm{A_{\rm rc2}}{\clL(H,V')}\norm{y}{H}\notag\\
&\hspace*{4em}\le \bigl(\,\norm{A_{\rm rc1}}{\clL(V,H)}+\underline{\xi}_M^\frac12\norm{A_{\rm rc2}}{\clL(H,V')}\norm{\Id\rest{V}}{\clL(V,H)}\,\bigr)\norm{y}{V}\notag\\
&\hspace*{4em}\le \bigl(1+\underline{\xi}_M^\frac12\alpha_1^{-\frac12}\,\bigr)2^\frac12\norm{(A_{\rm rc1},A_{\rm rc2})}{\clL(V,H)\times\clL(H,V')}\norm{y}{V},
\intertext{which leads us to}
&\norm{P_{\clE}^{\clE^{\perp V'}}
 A_{\rm rc}y}{H}\le (1+\underline{\xi}_M^\frac12\alpha_1^{-\frac12})2^\frac12\norm{A_{\rm rc}}{\clL(V,H)+\clL(H,V')}\norm{y}{V}.
 \end{align}
\end{subequations}
Therefore, from~\eqref{normFeed1}, it follows that
\begin{align}\notag
&\norm{\clF^{\rm obli}y}{H}\le D_0\norm{y}{V}\quad\mbox{with}\quad D_0=\ovlineC{C_P,C_{\rm rc},\lambda,\underline{\xi}_M,\alpha_1^{-1}},
\end{align}
and, denoting
\begin{equation}
 Z=P_{\clE^\perp}^{\clE}y=y-z,\quad\mbox{and}\quad \varphi(t)\coloneqq(t-s)\overline\mu ,\label{varphi-exp}
\end{equation}
we obtain
\begin{align}
\norm{\rme^{\varphi}\clF^{\rm obli}y}{L^2(\bbR_{+s},H)}^2&\le D_0^2\left(\norm{\rme^{\varphi}z}{L^2(\bbR_{+s},V)}^2+\norm{\rme^{\varphi}Z}{L^2(\bbR_{+s},V)}^2\right)\notag\\
&\le D_0^2\left(\underline\xi_M\norm{\rme^{\varphi}z}{L^2(\bbR_{+s},H)}^2+\norm{\rme^{\varphi}Z}{L^2(\bbR_{+s},V)}^2\right)\label{Fobl-norm1}
\end{align}
with~$\underline\xi_M$ as in Assumption~\ref{A:poincarexi}.
Next, we observe that
\begin{align}\notag
&\tfrac{\rmd}{\rmd t}\norm{\rme^{\varphi}Z}{H}^2
=2\overline\mu \rme^{2\varphi}\norm{Z}{H}^2+\rme^{2\varphi}\tfrac{\rmd}{\rmd t}\norm{Z}{H}^2
\end{align}
and, by~\eqref{dtZ3}, we find
\begin{align}
&\rme^{2\varphi}\tfrac{\rmd}{\rmd t}\norm{Z}{H}^2\le-\rme^{2\varphi}\norm{Z}{V}^2+D_1\rme^{2\varphi}\norm{Z}{H}^2+D_2\rme^{2\varphi}\norm{z}{H}^2\notag
\intertext{with}
&D_1=\ovlineC{C_\clE, C_P,C_{\rm rc}}\quad\mbox{and}\quad D_2=\ovlineC{C_\clE, C_P,C_{\rm rc},\underline{\xi}_M,\overline{\xi}_M^{-1}\lambda^2}.\notag
\end{align}
Thus, 
\begin{align}\notag
&\tfrac{\rmd}{\rmd t}\norm{\rme^{\varphi}Z}{H}^2+\rme^{2\varphi}\norm{Z}{V}^2
\le2\overline\mu \rme^{2\varphi}\norm{Z}{H}^2+D_1\rme^{2\varphi}\norm{Z}{H}^2+D_2\rme^{2\varphi}\norm{z}{H}^2
\end{align}
and, time integration over~$I_T\coloneqq (s,s+T)$, for arbitrary~$T>0$, gives us
\begin{align}
&-\norm{Z(s)}{H}^2+\norm{\rme^{\varphi}Z}{L^2(I_T,V)}^2\notag\\
&\hspace{3em}\le 2\overline\mu \norm{\rme^{\varphi}Z}{L^2(I_T,H)}^2+D_1\norm{\rme^{\varphi}Z}{L^2(I_T,H)}^2+D_2\norm{\rme^{\varphi}z}{L^2(I_T,H)}^2\notag\\
&\hspace{3em}\le 2\overline\mu \norm{\rme^{\varphi}Z}{L^2(\bbR_{+s},H)}^2+D_1\norm{\rme^{\varphi}Z}{L^2(\bbR_{+s},H)}^2+D_2\norm{\rme^{\varphi}z}{L^2(\bbR_{+s},H)}^2.\notag
\end{align}
Hence, since~$T$ is arbitrary,
\begin{align}
&\norm{\rme^{\varphi}Z}{L^2(\bbR_{+s},V)}^2\le (2\overline\mu +D_1)\norm{\rme^{\varphi}Z}{L^2(\bbR_{+s},H)}^2+D_2\norm{\rme^{\varphi}z}{L^2(\bbR_{+s},H)}^2+\norm{Z(s)}{H}^2\notag
\end{align}
and, by recalling~\eqref{Fobl-norm1}, we find
\begin{align}
\norm{\rme^{\varphi}\clF^{\rm obli}y}{L^2(\bbR_{+s},H)}^2&\le D_0^2(\underline{\xi}_M+D_2)\norm{\rme^{\varphi}z}{L^2(\bbR_{+s},H)}^2\notag\\
&\quad +D_0^2(2\overline\mu +D_1)\norm{\rme^{\varphi}Z}{L^2(\bbR_{+s},H)}^2+D_0^2\norm{Z(s)}{H}^2.\notag
\end{align}

Therefore,  recalling now~\eqref{expy-bound},
\begin{align}
\norm{\rme^{\varphi}\clF^{\rm obli}y}{L^2(\bbR_{+s},H)}^2&\le D_0^2(\underline{\xi}_M+D_2)\varrho\tfrac1{2(\mu-\overline\mu )}\norm{y(s)}{H}^2\notag\\
&\quad +D_0^2(2\overline\mu +D_1)\varrho\tfrac1{2(\mu-\overline\mu )}\norm{y(s)}{H}^2+D_0^2\norm{Z(s)}{H}^2,\notag
\end{align}
and, by setting
\[
D_3\coloneqq D_0^2\max\{(\underline{\xi}_M+D_2)\varrho\tfrac1{2(\mu-\overline\mu )},1+(2\overline\mu +D_1)\varrho\tfrac1{2(\mu-\overline\mu )}\}
\]
and recalling~\eqref{varphi-exp}, we arrive at
\begin{align}
\norm{\rme^{(\Bigcdot-s)\overline\mu }\clF^{\rm obli}y}{L^2(\bbR_{+s},H)}^2&\le D_3\norm{y(s)}{H}^2,\notag
\end{align}
which ends the proof.
\end{proof}

\begin{corollary}\label{C:stabOP-coo}
Let Assumptions~\ref{A:A0sp}--\ref{A:poincare} hold true and let~$M_*\in \bbN_*$ and~$\lambda\ge\mu$ be as in Theorem~\ref{T:stabOP}. Then the stabilizing input control~$u=B^{-1}\clF^{\rm obli}y$ is bounded:
\[
\norm{\rme^{(\Bigcdot-s)\overline\mu }u}{L^2((s,+\infty),\bbR^{M_\sigma})}^2\le D\norm{B^{-1}}{\clL(H,\bbR^{M_\sigma})}^2\norm{y(s)}{H}^2,
\]
for all~$s\ge0$, with a constant~$D=\ovlineC{C_\clE, C_P,C_{\rm rc},\underline{\xi}_M,\overline{\xi}_M^{-1}\lambda^2,\frac{1}{\max\{\gamma\mu,\lambda\}-\mu},\frac1{\mu-\overline\mu },\lambda,\alpha_1^{-1},\overline\mu}.$
\end{corollary}
\begin{proof}
Straightforward, from Corollary~\ref{C:stabOP}.
\end{proof}

\subsection{Optimal control and the classical Riccati feedback}\label{sS:optim-ricc}
 From Theorem~\ref{T:stabOP} and its Corollaries~\ref{C:stabOP} and~\ref{C:stabOP-coo} it follows that we can find a set~$U_{M_0}$, with~$M_0=\sigma(M)\in\bbN_+$, of actuators and a control input such that, for an arbitrary initial state~$y_0\in H$, the spent energy
\begin{equation}\label{LQcost-mu}
\clJ^{\overline\mu ,\beta}(y_0;y,u)\coloneqq \tfrac12\norm{\rme^{\overline\mu\Bigcdot }y}{L^2(\bbR_+,H)}^2+ \tfrac12\beta\norm{\rme^{\overline\mu\Bigcdot }u}{L^2(\bbR_+,\bbR^{M_0})}^2
\end{equation}
is bounded and satisfies
\begin{equation}\label{ol-fcc}
 \clJ^{\overline\mu ,\beta}(y_0;y,u)<C\norm{y_0}{H}^2,
\end{equation}
where~$\beta>0$, $\overline\mu >0$ and the feedback control input~$u(t)$ is given by
\begin{equation}\label{u-obli}
u(t)=\clK^{\rm obli}(t)y(t)\coloneqq B^{-1}\clF^{\rm obli}(t)y(t),
\end{equation}
and the control operator~$B\in\clL(\bbR^{M_0},\clU_M)\subset\clL(\bbR^{M_0},H)$ is the isomorphism as in~\eqref{B_Phi}.
In applications, it is (or, may be) important to minimize the spent energy. In such case 
we look for the pair~$(\widetilde y,\widetilde u)$ minimizing the cost functional~$\clJ^{\overline\mu ,\beta}$,
\begin{align}
 &\clJ^{\overline\mu ,\beta}(y_0;\widetilde y,\widetilde u)=\min\left\{\clJ^{\overline\mu ,\beta}(y_0;y,u)\mid (y,u)\in L^2(\bbR_+,H)\times L^2(\bbR_+,\bbR^{M_0})\right\}\notag
\intertext{subject to the constraints}
 &\dot y +Ay+A_{\rm rc}y -Bu=0,\qquad y(0)-y_0=0.\label{dyn-yu-barmu0}
\end{align}

Note that if we define~$y_{\overline\mu}(t)\coloneqq\rme^{\overline\mu  t}y(t)$ and~$u_{\overline\mu}(t)\coloneqq\rme^{\overline\mu  t}u(t)$, then  $(y,u)$ solves~\eqref{dyn-yu-barmu0}
 if, and only if, $(y_{\overline\mu},u_{\overline\mu})$ solves
 \[
 \dot y_{\overline\mu} +Ay_{\overline\mu}+(A_{\rm rc}-{\overline\mu}\Id)y -Bu_{\overline\mu}=0,\qquad y_{\overline\mu}(0)-y_0=0.
 \]
 
 Hence the optimal control problem above is equivalent to look for~$(\widehat y,\widehat u)$ solving
\begin{subequations}\label{optim-pair}
\begin{align}
 &\clJ^{\beta}(y_0;\widehat y,\widehat  u)=\min\left\{\clJ^{\beta}(y_0;y,u)\mid (y,u)\in L^2(\bbR_+,H)\times L^2(\bbR_+,\bbR^{M_0})\right\}
\intertext{subject to the constraints}
 &\dot y +Ay+(A_{\rm rc}-{\overline\mu}\Id)y -Bu=0,\qquad y(0)-y_0=0,\label{optim-pair-dyn}
 \intertext{where}
 &\clJ^{\beta}(y_0;y,u)\coloneqq \tfrac12\norm{y}{L^2(\bbR_+,H)}^2+ \tfrac12\beta\norm{u}{L^2(\bbR_+,\bbR^{M_0})}^2.\label{LQcost}
\end{align}
\end{subequations} 
 Indeed, we will have~$(\widehat y(t),\widehat u(t))=(\rme^{\overline\mu t}\widetilde y(t),\rme^{\overline\mu t}\widetilde u(t))$.

We show now that, for solutions of~\eqref{optim-pair-dyn},  the boundedness of~$\clJ^\beta(y_0;y,u)$ follows from that of
$\tfrac12\norm{P_{\clE^\rmf_{M_{1}}}y}{L^2(\bbR_+,L^2)}^2+\tfrac12\beta\norm{u}{L^2(\bbR_+,\bbR^{M_0})}^2,
$
for~$M_{1}$ large enough and where~$P_{\clE^\rmf_{M_{1}}}$ is the orthogonal projection onto the linear span 
\[
\clE^\rmf_{M_{1}}\coloneqq\linspan\{e_i\mid 1\le i\le M_1\}
\]
of the first eigenfunctions of the diffusion-like operator~$A$; see~\eqref{eigfeigv}.
 
 \begin{theorem}\label{T:detectM1}
Let a solution~$(y,u)$ of system~\eqref{optim-pair-dyn}, with~$B\in\clL(\bbR^{M_0},H)$ satisfy
\begin{equation}\label{cost-M1}
\clJ^\beta_{M_1}(y_0;y,u)\coloneqq\tfrac12\norm{P_{\clE^\rmf_{M_1}}y}{L^2(\bbR_+,H)}^2+\tfrac12\beta\norm{u}{L^2(\bbR_+,\bbR^{M_0})}^2\le C_J\norm{y_0}{H}^2
\end{equation}
for a constant~$C_J>0$ independent of~$y_0$. If~$M_1$ is large enough we also have an estimate~$\clJ^\beta(y_0;y,u)=\tfrac12\norm{y}{L^2(\bbR_+,L^2)}^2+\tfrac12\beta\norm{u}{L^2(\bbR_+,\bbR^{M_0})}^2\le\widehat C_J\norm{y_0}{H}^2$ for a suitable constant~$\widehat C_J>0$ independent of~$y_0$.
\end{theorem}

\begin{proof}
With~$q\coloneqq P_{\clE^\rmf_{M_1}}y$ and ~$Q\coloneqq y- P_{\clE^\rmf_{M_1}}y$, we find
\begin{align}
 &\dot Q +A Q+(\Id-P_{\clE^\rmf_{M_1}})(A_{\rm rc}-{\overline\mu}\Id)  Q 
 =(\Id-P_{\clE^\rmf_{M_1}})F \notag
\end{align}
with~$F\coloneqq Bu-(A_{\rm rc}-{\overline\mu}\Id)  q$ and
\begin{align}
\tfrac{\rmd}{\rmd t}\norm{Q}{H}^2 +2\norm{Q}{V}^2 &=-2\langle (\Id-P_{\clE^\rmf_{M_1}})A_{\rm rc}  Q,Q\rangle_{V',V} +2\overline\mu\norm{Q}{H}^2+2\langle (\Id-P_{\clE^\rmf_{M_1}})F,Q\rangle_{V',V} \notag\\
&=-2\langle A_{\rm rc}  Q,Q\rangle_{V',V}+2\overline\mu\norm{Q}{H}^2 +2\langle F,Q\rangle_{V',V},\notag
\end{align}
and using~\cite[Lem.~3.1]{Rod21-sicon}  and the Young inequality
\begin{align}
 \tfrac{\rmd}{\rmd t}\norm{Q}{H}^2 +2\norm{Q}{V}^2  &\le \tfrac12\norm{Q}{V}^2+8C_{\rm rc}\norm{Q}{H}^2 +2\overline\mu\norm{Q}{H}^2+2\norm{F}{V'}\norm{Q}{V} \notag\\
  & \le\norm{Q}{V}^2+(8C_{\rm rc}^2+2\overline\mu)\norm{Q}{H}^2+2\norm{F}{V'}^2.\notag
\end{align}
We can see that~$\norm{F}{V'}\le C_B\norm{u}{\bbR^{M_0}}+D_0\norm{q}{H}$ for suitable positive constants~$C_B$ and~$D_0=\ovlineC{C_{\rm rc},\overline\mu,M_1}$. Note also that, 
proceeding as in~\eqref{react-est1}, with~$A_{\rm rc}=A_{\rm rc1}+A_{\rm rc2}$ and~$(A_{\rm rc1},A_{\rm rc2})\in \clL(V,H)\times\clL(H,V')$, we find that, for an arbitrary~$h\in V$,
\begin{align}
\norm{\left\langle A_{\rm rc}q,h\right\rangle_{V',V}}{\bbR}
&\le\norm{A_{\rm rc1}}{\clL(V,H)}\norm{q}{V}\norm{h}{H}
+\norm{A_{\rm rc2}}{\clL(H,V')}\norm{q}{H}\norm{h}{V}\notag\\
&\le(\alpha_{M_1}\alpha_{1}^{-1}\norm{A_{\rm rc1}}{\clL(V,H)}+\norm{A_{\rm rc2}}{\clL(H,V')})\norm{q}{H}\norm{h}{V}\notag\\
&\le(1+\alpha_{M_1}\alpha_{1}^{-1})2^\frac12\norm{A_{\rm rc}}{\clL(V,H)+\clL(H,V')}\norm{q}{H}\norm{h}{V}.
\notag
\end{align}

Therefore, we can arrive at
\begin{align}
\tfrac{\rmd}{\rmd t}\norm{Q}{H}^2  +\tfrac12\norm{Q}{V}^2
  & \le-\tfrac12\norm{Q}{V}^2+(8C_{\rm rc}^2+2\overline\mu)\norm{Q}{H}^2+4(C_B^2\norm{u}{\bbR^{M_0}}^2+D_0^2\norm{q}{H}^2).\notag
\end{align}
Now, since~$Q\in (\clE^\rmf_{M_1})^\perp$, we find
\begin{align}
\tfrac{\rmd}{\rmd t}\norm{Q}{H}^2  +\tfrac12\norm{Q}{V}^2
  & \le-(\tfrac12\alpha_{M_1+1}-8C_{\rm rc}^2-2\overline\mu)\norm{Q}{H}^2+4(C_B^2\norm{u}{\bbR^{M_0}}^2+D_0^2\norm{q}{H}^2).\notag
\end{align}
Hence, if~$M_1$ is large enough such that~$\alpha_{M_1+1}\ge 16C_{\rm rc}^2-4\overline\mu$, we obtain
\begin{align}
\tfrac{\rmd}{\rmd t}\norm{Q}{H}^2  +\tfrac12\norm{Q}{V}^2
  & \le4(C_B^2\norm{u}{\bbR^{M_0}}^2+D_0^2\norm{q}{H}^2), \notag
\end{align}
and time integration gives us, for all~$t\ge  0$,
\begin{align}
\norm{Q(t)}{H}^2  +\tfrac12\norm{Q}{L^2((0,t),V)}^2
  & \le\norm{Q(0)}{H}^2+4(C_B^2\norm{u}{L^2((0,t),\bbR^{M_0})}^2+D_0^2\norm{q}{L^2((0,t),H)}^2), \notag\\
  & \le\norm{Q(0)}{H}^2+4(C_B^2+D_0^2)C_J\norm{y_0}{H}^2, \notag
 \end{align}
which implies
\begin{align}
\norm{Q}{L^\infty(\bbR_+,H)}^2  +\tfrac12\norm{Q}{L^2(\bbR_+,V)}^2
  & \le(1+4(C_B^2+D_0^2)C_J)\norm{y_0}{H}^2.\notag
\end{align}
In particular,
\begin{align}
&\norm{y}{L^2(\bbR_+,H)}^2+\norm{u}{L^2(\bbR_+,\bbR^{M_0})}^2
=\norm{q}{L^2(\bbR_+,H)}^2+\norm{u}{L^2(\bbR_+,\bbR^{M_0})}^2+\norm{Q}{L^2(\bbR_+,H)}^2\notag\\
&\hspace{3em}\le C_J\norm{y_0}{H}^2+\norm{\Id}{\clL(V,H)}^2\norm{Q}{L^2(\bbR_+,V)}^2\notag\\
&\hspace{3em}\le(C_J+2+8(C_B^2+D_0^2)C_J)\norm{y_0}{H}^2,\notag
\end{align}
and the result follows with~$\widehat C_J\coloneqq2+(1+8C_B^2+8D_0^2)C_J$.
\end{proof}

Considering~\eqref{cost-M1} instead of~\eqref{LQcost} can make numerical computations of the  optimal feedback  operator easier/faster (at least, in the autonomous case). Thus, we shall look for the optimal pair~$(\widehat y,\widehat u)$ solving
problems as
\begin{subequations}\label{optim-pairC}
\begin{align}
 &\clJ^{\beta}_{M_1}(y_0;\widehat y,\widehat u)=\min\left\{\clJ^{\beta}_{M_1}(y_0;y,u)\mid (y,u)\in L^2(\bbR_{+},H)\times L^2(\bbR_{+},\bbR^{M_0})\right\}
\label{optim-min}
\intertext{subject to the constraints}
 &\dot y +Ay+(A_{\rm rc}-\overline\mu\Id)y -Bu=0,\qquad y(0)-y_0=0.\label{sys-zBu}
\end{align}
\end{subequations} 
Following the arguments as in~\cite{BarRodShi11},  as a consequence of the Karush--Kuhn--Tucker conditions and the Dynamic Programming Principle, it turns out that the optimal control function~$\widehat u$ is given by
\begin{equation}\label{u-ricc}
 \widehat u(t)=\clK^{\rm ricc}(t)\widehat y(t)\coloneqq- \beta^{-1}B^*\Pi(t)\widehat y(t),
\end{equation}
where $B^*\in\clL(H,\bbR^{M_\sigma})$ ``is'' the adjoint of~$B$ and $\Pi\succeq0$ gives us the ``cost to go'' as
\begin{equation}\label{Pi0-optimcost}
\tfrac12(\Pi(t)\widehat y(t),\widehat y(t))_H=\tfrac12\norm{P_{\clE^\rmf_{M_1}}\widehat y}{L^2((t,+\infty),H)}^2+\tfrac12\beta\norm{\widehat u}{L^2((t,+\infty),\bbR^{M_0})}^2.
\end{equation}
Furthermore, $\Pi\succeq0$ solves the operator  differential
 Riccati equation
\begin{align}\label{Riccati}
&\dot\Pi+X^*\Pi+\Pi X-\beta^{-1}\Pi BB^*\Pi+P_{\clE^\rmf_{M_1}}=0,\qquad t\ge0,
\intertext{where}
&X=X(t)=-A-A_{\rm rc}(t)+{\overline\mu}\Id. \label{X}
\end{align}

\begin{remark}
Due to the identification~$H=H'$ that we have made, it follows that~$B^*\in\clL(H',(\bbR^{M_0})')=\clL(H,(\bbR^{M_0})')$ for given~$B\in\clL(\bbR^{M_0},H)$. Thus, the product~$BB^*z$ makes sense if, and only if, we also identify~$(\bbR^{M_0})'=\bbR^{M_0}$. More precisely, let us identify~$\bbR^{M_0}$ with column vectors~$\bbR^{M_0}=\bbR^{M_0\times 1}$, then~$(\bbR^{M_0})'=\bbR^{1\times M_0}$ is the space of row vectors. In this case
~$BB^*z$ can (and should) be understood as~$B(B^*z)^\top$. This also shows that we can identify~$(\bbR^{M_0})'=\bbR^{M_0}$ without entering in contradiction with the prior identification~$H=H'$. Recall that, in general we cannot consider two arbitrarily given Hilbert spaces simultaneously as pivot spaces.
\end{remark} 
 
\subsection{Finding the periodic optimal control iteratively}\label{sS:iteratPer-ricc} 
 Note that the initial condition~$\Pi(0)$ is not given in~\eqref{Riccati}. In fact~\eqref{Riccati}
 is to be solved backwards in time, in the unbounded time interval~$[0,+\infty)$. Hence,  in practice,
 the computation (of an approximation) of~$\Pi$ is unfeasible
 for general~$A_{\rm rc}\in L^\infty((0,+\infty),\clL(V,H)+\clL(H,V'))$. This is why, hereafter, we will focus on the case where~$A_{\rm rc}$ is time-periodic, say with period~$\varpi>0$,
 \[
 A_{\rm rc}(t)=A(t+\varpi)\quad\mbox{for all}\quad t\ge0.
 \]
In this case, we can restrict the computations to a finite time interval~$[\tau,\tau+\varpi]$, for  fixed~$\tau\ge0$. We follow a strategy analogous to the one proposed in~\cite{KroRod15} plus one additional iterative step for periodicity: 
\begin{enumerate}[noitemsep,topsep=5pt,parsep=5pt,partopsep=0pt,leftmargin=3em]%
	\renewcommand{\theenumi}{({\rm Ric}-{\roman{enumi}})} %
\renewcommand{\labelenumi}{({\rm Ric}-{\roman{enumi}})}%
\item\label{RicPer:st1}
firstly we choose~$\tau\ge0$ so that~$A_{\rm rc}(\tau+\varpi)\in\clL(V,H)+\clL(H,V')$ is well defined, at time~$t=\tau+\varpi$, and look for the nonnegative definite ($\succeq0$, for short) solution~$\Pi_{\tau+\varpi}$ of the
algebraic operator Riccati equation
\begin{subequations}\label{Riccati-algTbk}
 \begin{align}\label{Riccati-alg}
  &X^*(\tau+\varpi)\Pi_{\tau+\varpi}+\Pi_{\tau+\varpi}  X(\tau+\varpi)-\beta^{-1}\Pi_{\tau+\varpi} BB^{*}\Pi_{\tau+\varpi}+\clC^{*}\clC=0,\\
 &\Pi_{\tau+\varpi}\succeq0,\qquad\clC\in\clL(H)\quad\mbox{with}\quad\clC^{*}\clC=P_{\clE^\rmf_{M_1}};\notag
\end{align}
satisfying
\begin{equation}\notag%
\tfrac12(\Pi_{\tau+\varpi}w,w)_H=\min\tfrac12\norm{P_{\clE^\rmf_{M_1}} y}{L^2((\tau+\varpi,+\infty),H)}^2+\tfrac12\beta\norm{u}{L^2((\tau+\varpi,+\infty),\bbR^{M_0})}^2,
\end{equation}
for~$(y,u)$ subject to the autonomous dynamics
\[
\dot y +Ay+(A_{\rm rc}(\tau+\varpi)-\overline\mu\Id)y -Bu=0,\qquad y(\tau+\varpi)-w=0,\qquad t>\tau+\varpi.
\]
\item\label{RicPer:st2} then, we use~$\Pi^1(\tau+\varpi)=\Pi_{\tau+\varpi}$ as final time condition and solve the differential operator  Riccati equation
backwards in time,
\begin{align}
 &\dot\Pi^1+X^*\Pi^1+\Pi^1  X-\beta^{-1}\Pi ^1 BB^{*}\Pi^1+\clC^{*}\clC=0,\quad \Pi^1(\tau+\varpi)=\Pi_{\tau+\varpi},\label{Riccati-Tbk}\\
 &\Pi^1(t)\succeq0\quad\mbox{for all}\quad t\in[\tau,\tau+\varpi];\notag
\end{align}
\end{subequations}
\item\label{RicPer:st3} finally, if ~$\norm{\Pi^1(\tau)-\Pi^1(\tau+\varpi)}{\clL(H)}$  we accept~$\Pi^1$ as periodic solution. Otherwise, we repeat step~\ref{RicPer:st2},  solving~\eqref{Riccati-Tbk} with final condition~$\Pi^n(\tau+\varpi)=\Pi^{n-1}(\tau)$ until we find a solution with~$\Pi^n(\tau)\approx\Pi^n(\tau+\varpi)$, $n\ge2$.
\end{enumerate}

\begin{remark}
We can take~$\clC=\clC^{*}=\clC^{*}\clC=P_{\clE^\rmf_{M_1}}$ in~\eqref{Riccati-alg}. We just write it as the product~$\clC^{*}\clC$ to have a more canonical form for the Riccati equation.
\end{remark}

The next result concerns the convergence of iterates of solutions of~\eqref{Riccati-Tbk}.
\begin{theorem}\label{T:convRicPer}
Assume that~$A_{\rm rc}$ satisfies Assumption~\ref{A:A1}, is time periodic with period~$\varpi>0$, and is well defined at~$t=\tau+\varpi\ge0$ with~$A_{\rm rc}(\tau+\varpi)\in\clL(H,V')+\clL(V,H)$. Then, the sequence~$\Pi^n(\tau)$ as in step~\ref{RicPer:st3} above, with~$\Pi^1(\tau+\varpi)=\Pi_{\tau+\varpi}$, concerning solutions of the differential Riccati equation~\eqref{Riccati-Tbk} converges, in the weak operator topology, to the operator~$\Pi_{\rm p}(\tau)$ given by the evaluation at initial time~$\tau$ of the periodic solution
\begin{align}
 &\dot\Pi_{\rm p}+X^*\Pi_{\rm p}+\Pi_{\rm p}  X-\beta^{-1}\Pi_{\rm p} BB^{*}\Pi_{\rm p}+\clC^{*}\clC=0,\quad \Pi_{\rm p}(\tau+\varpi)=\Pi_{\rm p}(\tau),\label{Riccati-per}\\
 &\Pi_{\rm p}(t)\succeq0\quad\mbox{for all}\quad t\in[\tau,\tau+\varpi],\notag
\end{align}
giving us the optimal cost to go (cf.~\eqref{Pi0-optimcost}) 
\[
\tfrac12(\Pi_{\rm p}(\tau)w,w)_H=\min\tfrac12\norm{P_{\clE^\rmf_{M_1}} y}{L^2((\tau,+\infty),H)}^2+\tfrac12\beta\norm{u}{L^2((\tau,+\infty),\bbR^{M_0})}^2,
\]
for~$(y,u)$ subject to the nonautonomous time-periodic dynamics
\[
\dot y +Ay+(A_{\rm rc}-\overline\mu\Id)y -Bu=0,\qquad y(\tau)-w=0,\qquad t\ge\tau.
\]
Moreover, for all~$(w_1,w_2)\in H\times H$, we have that~$(\Pi^{n}(\tau)w_1,w_2)_H$ converges exponentially to~$(\Pi_{\rm p}(\tau)w_1,w_2)_H$.
\end{theorem}

\begin{proof}
We denote the optimal solution of the periodic dynamics by~$(\widehat y,\widehat u)$, thus
\begin{equation}\label{cost-tau}
\tfrac12(\Pi_{\rm p}(\tau)w,w)_H=\tfrac12\norm{P_{\clE^\rmf_{M_1}}\widehat y}{L^2((\tau,+\infty),H)}^2+\tfrac12\beta\norm{\widehat u}{L^2((\tau,+\infty),\bbR^{M_0})}^2.
\end{equation}

Let us now consider the analog system where we 
take a time independent~$A_{n,\rm rc}$ for time~$t\ge\tau+n\varpi$, namely,
\[
 \dot z_n +Az_n +(A_{n,\rm rc}-\overline\mu\Id)z_n  -Bv_n =0,\qquad z_n (\tau)=w,
\]
with
\[
A_{n,\rm rc}(t)\coloneqq\begin{cases}
A_{\rm rc}(t),&\quad\mbox{for}\quad t\in(\tau,\tau+n\varpi),\\
A_{\rm rc}(\tau+n\varpi),&\quad\mbox{for}\quad t\in[\tau+n\varpi,+\infty).
\end{cases}
\]
Analogously, for the corresponding optimal cost and optimal pair~$(\widehat z_n ,\widehat v_n )$, we find
\begin{equation}\label{cost-tau-n}
\tfrac12(\Pi^{n}(\tau)w,w)_H=\tfrac12\norm{P_{\clE^\rmf_{M_1}}\widehat z_n }{L^2(\bbR_{\tau+},H)}^2+\tfrac12\beta\norm{\widehat v_n }{L^2(\bbR_{\tau+},\bbR^{M_0})}^2
\end{equation}
where~$\Pi^{n}$ is the solution of the corresponding Riccati equation. By the dynamical programming principle, and the time-periodicity it follows that for~$t\ge\tau+n\varpi$, we have that~$\Pi^{n}(t)=\Pi^{n}(\tau+n\varpi)$ where~$\Pi^{n}(\tau+n\varpi)$ solves the algebraic equation in~\eqref{Riccati-alg}, hence
\[
\Pi^{n}(t)=\Pi_{\tau+\varpi}=\Pi^1(\tau+\varpi)\quad\mbox{for all}\quad t\ge\tau+n\varpi,
\]
with~$\Pi_{\tau+\varpi}$ as in~\eqref{Riccati-alg} and~\eqref{Riccati-Tbk}.

The  optimal costs are bounded as follows.
\begin{subequations}\label{CJ-Datko}
\begin{align}
\tfrac12(\Pi_{\rm p}(\tau)w,w)_H
&\le C_J\norm{w}{H}^2,\\
\tfrac12(\Pi^n(\tau)w,w)_H
&\le C_J\norm{w}{H}^2,\quad \mbox{for all}\quad n\in\bbN_+,
\end{align}
\end{subequations}
for suitable positive constant
$C_J$. Note that, by optimality, the Riccati feedback gives us a cost smaller that the one obtained with the explicit oblique projection feedback, hence by Theorem~\ref{T:stabOP} and Corollary~\ref{C:stabOP-coo} the constant~$C_J$
can be taken depending on the upper bound~$C_{\rm rc}$ for the norm of~$A_{\rm rc}$ as in Assumption~\ref{A:A1}, thus independent of~$n$.
Let us now denote the interval
\[
I_n\coloneqq(\tau,\tau+n\varpi)
\]
and the truncated cost functional
\begin{align}
\clJ_{\tau}^n(w;y,u)&\coloneqq\tfrac12\norm{P_{\clE^\rmf_{M_1}}y}{L^2((I_n,H)}^2+\tfrac12\beta\norm{u}{L^2(I_n,\bbR^{M_0})}^2.\notag
\end{align}

By optimality and the dynamic programming principle we also find that
\begin{align}
\tfrac12(\Pi_{\rm p}(\tau)w,w)_H&\le\clJ_{\tau}^n(w;\widehat z_n,\widehat v_n)
+\tfrac12(\Pi_{\rm p}(\tau)\widehat z_n(\tau+n\varpi),\widehat z_n(\tau+n\varpi))_H\notag\\
&=\tfrac12(\Pi^n(\tau)w,w)_H-\tfrac12(\Pi_{\tau+\varpi}\widehat z_n(\tau+n\varpi),\widehat z_n(\tau+n\varpi))_H\notag\\
&\quad+\tfrac12(\Pi_{\rm p}(\tau)\widehat z_n(\tau+n\varpi),\widehat z_n(\tau+n\varpi))_H,\notag
\\
\tfrac12(\Pi^n(\tau)w,w)_H&\le\clJ_{\tau}^n(w;\widehat y,\widehat u)
+\tfrac12(\Pi_{\tau+\varpi}\widehat y(\tau+n\varpi),\widehat y(\tau+n\varpi))_H\notag\\
&=\tfrac12(\Pi_{\rm p}(\tau)w,w)_H-\tfrac12(\Pi_{\rm p}(\tau)\widehat y(\tau+n\varpi),\widehat y(\tau+n\varpi))_H\notag\\
&\quad+\tfrac12(\Pi_{\tau+\varpi}\widehat y(\tau+n\varpi),\widehat y(\tau+n\varpi))_H,\notag
\end{align}
which give us
\begin{align}
&(\Pi_{\rm p}(\tau)w,w)_H-(\Pi^n(\tau)w,w)_H\notag\\
&\hspace{2em}\le(\Pi_{\rm p}(\tau)\widehat z_n(\tau+n\varpi),\widehat z_n(\tau+n\varpi))_H-(\Pi_{\tau+\varpi}\widehat z_n(\tau+n\varpi),\widehat z_n(\tau+n\varpi))_H\notag\\
&\hspace{2em}\le 4C_J\norm{\widehat z_n(\tau+n\varpi)}{H}^2,\notag
\\
&(\Pi^n(\tau)w,w)_H-(\Pi_{\rm p}(\tau)w,w)_H\notag\\
&\hspace{2em}\le\tfrac12(\Pi_{\tau+\varpi}\widehat y(\tau+n\varpi),\widehat y(\tau+n\varpi))_H
-\tfrac12(\Pi_{\rm p}(\tau)\widehat y(\tau+n\varpi),\widehat y(\tau+n\varpi))_H\notag\\
&\hspace{2em}\le 4C_J\norm{\widehat y(\tau+n\varpi)}{H}^2.\notag
\end{align}

Since the optimal control is given in feedback form with a bounded feedback operator, the resulting dynamical system gives us a $C(0,e)$ evolution process (cf.~\cite[Def.~1]{Datko72}). Hence, by Datko Theorem~\cite[Thm.~1]{Datko72} (see also~\cite[Thm.~2.2]{Gibson79}) we conclude that the optimal solutions converge exponentially to zero, that is,
\begin{align}
\norm{\widehat z_n(t)}{H}^2&\le D\rme^{-\epsilon (t-\tau)}\norm{\widehat z_n(\tau)}{H}^2=D\rme^{-\epsilon (t-\tau)}\norm{w}{H}^2,\quad\mbox{for all}\quad t\ge\tau,\notag\\
\norm{\widehat y(t)}{H}^2&\le D\rme^{-\epsilon (t-\tau)}\norm{\widehat y(\tau)}{H}^2=D\rme^{-\epsilon (t-\tau)}\norm{w}{H}^2,\quad\mbox{for all}\quad t\ge\tau,\notag
\end{align}
with~$D$ and~$\epsilon$ independent of~$n$. We refer the reader, in particular, to the arguments in the proof of~\cite[Thm.~1]{Datko72} where the exponential stability is derived from an inequality as~\eqref{CJ-Datko}, namely,~\cite[Equ.~(7)]{Datko72}. Therefore, we have
\begin{align}
(\Pi_{\rm p}(\tau)w,w)_H-(\Pi^n(\tau)w,w)_H
&\le 4C_JD\rme^{-\epsilon n\varpi}\norm{w}{H}^2,\notag
\\
(\Pi^n(\tau)w,w)_H-(\Pi_{\rm p}(\tau)w,w)_H
&\le 4C_JD\rme^{-\epsilon n\varpi}\norm{w}{H}^2,\notag
\end{align}
thus~$(\Pi^n(\tau)w,w)_H$ converges to~$(\Pi_{\rm p}(\tau)w,w)_H$,
exponentially with rate~$\epsilon \varpi$,
\begin{align}
\norm{(\Pi^n(\tau)w,w)_H-(\Pi_{\rm p}(\tau)w,w)_H}{\bbR}
&\le 4C_JD\rme^{-\epsilon \varpi n}\norm{w}{H}^2,\quad \mbox{for all}\quad w\in H.\notag
\end{align}

For an arbitrary pair~$(w_1,w_2)\in H\times H$, using the symmetry and linearity of~$\Pi_{\rm p}$ and~$\Pi^{n}$, and the triangle inequality, we obtain
\begin{align}
&\norm{2(\Pi_{\rm p}w_1,w_2)_H-2(\Pi^{n}w_1,w_2)_H}{\bbR}\notag\\
&\hspace{1.5em}\le \norm{(\Pi_{\rm p}(w_1+w_2),(w_1+w_2))_H-(\Pi^{n}(w_1+w_2),(w_1+w_2))_H}{\bbR}\notag\\
&\hspace{2.5em}+\norm{-(\Pi_{\rm p}w_1,w_1)_H+(\Pi^{n}w_1,w_1)_H}{\bbR}+\norm{-(\Pi_{\rm p}w_2,w_2)_H+(\Pi^{n}w_2,w_2)_H}{\bbR}\notag\\
&\hspace{1.5em}\le4C_JD\rme^{-\epsilon \varpi n}\left(\norm{w_1+w_2}{H}^2+\norm{w_1}{H}^2+\norm{w_2}{H}^2\right)\le12C_JD\rme^{-\epsilon \varpi n}\left(\norm{w_1}{H}^2+\norm{w_2}{H}^2\right)\!,\notag
\end{align}
which implies that, for all~$(w_1,w_2)\in H\times H$, the scalar product
$(\Pi^{n}w_1,w_2)_H$ converges exponentially to~$(\Pi_{\rm p}w_1,w_2)_H$.
In particular, $\Pi^{n}$ converges to~$\Pi_{\rm p}$ in the weak operator topology.
\end{proof}

 For results concerning the existence and uniqueness of solutions for general equations in the form~\eqref{Riccati-Tbk} we refer the reader to~\cite{CurtainPritchard76} and references therein.

\subsection{Homotopy step for algebraic Riccati equations. Stabilizability and detectability}\label{sS:homotopy}
In the process of solving~\eqref{Riccati-alg}, through a Newton iteration, we shall need to provide a stabilizing initial guess~$\Pi_G$ so that~$X-BB^*\Pi_G$ is stable. That is, essentially we need a
stabilizing feedback operator. In general, finding~$\Pi_G$ is nontrivial, we shall overcome this issue by considering the family of equations
\begin{subequations}\label{Riccati-alg-s}
\begin{align}
&\clX_s^*\Pi+\Pi\clX_s-\beta^{-1}\Pi BB^{*}\Pi+\clC^{*}\clC=0,\qquad s\in[0,1],
\intertext{with}
&\clX_s\coloneqq -A-s(A_{\rm rc}(\tau+\varpi)-\overline\mu\Id).
\end{align}
\end{subequations}
Recalling~\eqref{eigfeigv}, for $s=0$ the operator~$\clX_0=-A$ is stable  and an initial stabilizing feedback is easier to find, for example, ~$\clX_0-BB^{*}\Pi_G$ is stable with~$\Pi_G=\zero$.

Then we shall consider a discrete homotopy with $N+1$ steps connecting~$\clX_0=-A$ to~$\clX_1=X(\tau+\varpi)$, where we shall use the solution of the Riccati equation for $s=(k-1)\frac{1}{N}$ as initial guess to solve the equation for~$s=k\frac{1}{N}$, $1\le k\le N$. Note that we are essentially replacing the reaction-convection term~$A_{\rm rc}(\tau+\varpi)$ by~$sA_{\rm rc}(\tau+\varpi)$ and asking for a smaller stability rate~$s\overline\mu\le\overline\mu$. Recall that we know that the number of actuators and in Theorem~\ref{T:stabOP}, to guarantee a stability rate~$\overline\mu$, depends on an upper bound~$C_{\rm rc}$ as in Assumption~\ref{A:A1}, since this bound is smaller for $0\le s<1$, we have that there exists a set of actuators that stabilize the system with rate~$\overline\mu$ for all~$s\in[0,1]$.
Analogously, we can see that the natural number~$M_1$ in Theorem~\ref{T:detectM1} also depends on the upper bound~$C_{\rm rc}$ in Assumption~\ref{A:A1}, thus there exists such an~$M_1$ for which the same theorem holds for all~$s\in[0,1]$.
Therefore, we can follow the arguments in section~\ref{sS:optim-ricc} to guarantee the existence of a nonnegative definite solution for~\eqref{Riccati-alg-s} for each~$s\in[0,1]$.

At this point we would like to recall that, in general,   the existence of a nonnegative definite solution for general Riccati equations in the form~\eqref{Riccati-alg-s} is related to concepts of stabilizability and detectability, which we recall now, for the sake of completeness.
Let us be given Hilbert spaces~$V$ and~$H=H'$, satisfying Assumption~\ref{A:A0cdc},
and an operator~$\fkL\in\clL(V,V')$. We assume that, as expected for linear parabolic-like systems, weak solutions do exist for the autonomous linear system
\begin{equation}\label{sys.fkL}
\dot y=\fkL y,\qquad y(0)=y_0\in H,
\end{equation}
and satisfy
\[
y\in W_{\rm loc}(\bbR_+,V,V')\coloneqq\{y\in L^2_{\rm loc}(\bbR_+,V)\mid \dot y\in L^2_{\rm loc}(\bbR_+,V')\}\subset C(\overline\bbR_+,H).
\]
\begin{definition}
The operator~$\fkL\in\clL(V,V')$ is said exponentially stable, if there are constants~$\varrho\ge1$ and~$\mu>0$ such that every weak solution of~\eqref{sys.fkL}
satisfies~\eqref{goal-intro}.
\end{definition}

\begin{definition}\label{D:stab-pair}
The pair~$(\clX_s,B)\in\clL(V,V')\times\clL(\bbR^{M_0},H)$ is said stabilizable, if there exists~$K\in \clL(H,\bbR^{M_0})$ so that~$\clX_s+B K$ 
is exponentially stable.
\end{definition}

\begin{definition}\label{D:detect-pair}
The pair~$(\clX_s, \clC)\in\clL(V,V')\times\clL(H)$ is said detectable, if
there exists~$L\in \clL(H)$ so that~$\clX_s+L\clC$ 
is exponentially stable.
\end{definition}

Observe that the detectability of~$(\clX_s, \clC)$, as in Definition~\ref{D:detect-pair}, implies that if~$u=u(t)\in\bbR^{M_0}$ is a control function so that, for the weak solution of
\[
\dot y=\clX_s y+ Bu,\qquad z(0)=z_0\in H,
\]
we have that (cf.~\eqref{cost-M1})
 \begin{equation}\label{costC}
\tfrac12\norm{\clC y}{L^2(\bbR_+,H)}^2+\tfrac12\beta\norm{u}{L^2(\bbR_+,\bbR^{M_0})}^2
\end{equation}
 is bounded, then also (cf.~\eqref{LQcost})
 \begin{equation}\label{costC-gen}
\tfrac12\norm{y}{L^2(\bbR_+,H)}^2+\tfrac12\beta\norm{u}{L^2(\bbR_+,\bbR^{M_0})}^2
\end{equation}
is bounded.
Indeed, from
\begin{align}\notag
\dot y=\clX_s y+B u=\clX_s y+L \clC y -L \clC y+B u,
\end{align}
using the exponential stability of~$\clX_s+L \clC $ and looking at
\[
p\coloneqq-L \clC y +B u\in H,\qquad\norm{p}{L^2(\bbR_+, H)}^2<+\infty,
\]
as a perturbation, by Duhamel (variation of constants) formula   we can see that~\eqref{costC-gen} is bounded.
Finally, we recall that from the boundedness of~\eqref{costC-gen}, with the Riccati  control~$u=-\beta^{-1}B^{*}\Pi y$ minimizing~\eqref{costC}, we can derive that~$\clX_s-\beta^{-1}BB^{*}\Pi$ is exponentially stable, due to Datko results~\cite[Lem.~1 and Thm.~1]{Datko72}.

\subsection{On the computation of the control input}
Recalling~\eqref{u-ricc}, the control input for the Riccati feedback is given by~$u=-\beta^{-1}B^*\Pi y$, while for the oblique projection feedback it is given by~$u=B^{-1}\clF^{\rm obli} y$; see~\eqref{u-obli}. The oblique projection~$P_{\clU_{\sigma(M)}}^{\clE_{\sigma(M)}^{\perp V'}}$ in~\eqref{FeedKunRod} is an extension of~$P_{\clU_{\sigma(M)}}^{\clE_{\sigma(M)}^{\perp}}$ which can be computed as
$P_{\clU_{\sigma(M)}}^{\clE_{\sigma(M)}^{\perp}}z= B[(\clE_{\sigma(M)},\clU_{\sigma(M)})_H]^{-1} [(\clE_{\sigma(M)},z)_H]$,
where~$[(\clE_{\sigma(M)},z)_H]\in\bbR^{\sigma(M)\times 1}$ is the vector the~$i$th row of which contains the scalar product~$(e_{M,i},z)_H$, $1\le i\le M_\sigma$, involving the $i$th eigenfunction in the set spanning~$\clE_{\sigma(M)}$ and~$[(\clE_{\sigma(M)},\clU_{\sigma(M)})_H]$ is the matrix the entry of which in the $i$th row and $j$th column is given by $(e_{M,i},\Phi_{M,j})_H$ where~$\Phi_j$ is the~$j$th actuator in the set spanning~$\clU_{\sigma(M)}$ (cf.~\cite[Lem.~2.8]{KunRod19-cocv}).

Therefore, in order to compare the optimal cost associated with Riccati feedback to the larger cost associated with the explicit oblique projection feedback, we can compute the corresponding control input vectors as
\begin{subequations}
\begin{align}
&u^{\rm ricc}=-\beta^{-1}B^*\Pi y^{\rm ricc};\label{uy-ricc}\\
&u^{\rm obli}=[(\clE_{{\sigma(M)}},\clU_{{\sigma(M)}})_H]^{-1} [(\clE_{{\sigma(M)}},Ay^{\rm obli}+A_{\rm rc}y^{\rm obli}-\lambda y^{\rm obli})_H].\label{uy-obli}
\end{align} 
\end{subequations}

\subsection{Stabilizability: Riccati versus oblique projection}
Assumption~\ref{A:poincare} is required for stabilizability with oblique projection feedbacks as~\eqref{FeedKunRod}. For scalar parabolic equations, such assumption is satisfied for suitable locations of the actuators; recall~\eqref{loc-actOP1D}. Riccati based stabilizing feedbacks can be found also  for other locations where the explicit oblique projection feedback may be not stabilizing, namely, for actuators located in an apriori given subdomain~$\clO\subset\Omega$; see~\cite{BreKunRod17,PhanRod18-mcss}.
Therefore, one advantage of Riccati feedback is that it may succeed to stabilize the system when the explicit feedback fails. A second advantage is that it gives us the solution minimizing a classical energy functional.

On the other hand some advantages of the explicit feedback are that it is less expensive to compute, and it can be computed online in real time, while Riccati has to be computed offline. The computation cost  of the explicit feedback is essentially the same for autonomous and nonautonomous systems, while  Riccati is more expensive for nonautonomous systems (involving the solution of a differential equation) than for autonomous systems (involving the solution of an algebraic equation). Furthermore, Riccati  is  impossible to solve in the entire time interval $[0,+\infty)$ for general nonautonomous systems. 

For the particular case of nonautonomous time-periodic dynamics, it is possible to solve the Riccati equation, because its solution is also time-periodic with the same time-period, thus we can look for the periodic solution in a finite time interval with length equal to the time-period. Computing this solution is still an expensive numerical task, but if we succeed, then  the resulting feedback will likely stabilize the system when the explicit one fails.

Here we should also mention that in practical applications we could be interested in feedback operators which are able to squeeze the norm of the solution after a certain  time horizon~$T$, say, with~$T$ large enough. In this case Riccati feedbacks can also be useful (for general nonautonomous systems), as proposed in~\cite{KroRod15}, with an appropriate guess/operator~$\Pi_T$ at final time~$t=T$ for the differential Riccati equation. Here, since the time interval of interest is~$(0,T)$ we can, for example, assume that the dynamics is autonomous for time~$t\ge T$ as done in~\cite[sect.~5.3.2]{KroRod15} and use the corresponding solution of the algebraic Riccati equation for~$\Pi_T$.

\section{On the numerical implementation}\label{S:NumerImpl}
We consider linear parabolic equations as~\eqref{sys-y-parab}.
For simplicity, we restrict the exposition to the case of homogeneous Neumann boundary conditions,
\begin{subequations}\label{sys-y-parab-clK}
\begin{align}
 &\tfrac{\p}{\p t} y +(-\nu\Delta+\Id) y+ay +b\cdot\nabla y
 =B\clK y,\\
  &\tfrac{\p y}{\p\bfn}\rest{\p\Omega}=0,\qquad
y(0)=y_0,
 \intertext{with a linear continuous input feedback  control operator~$\clK\colon V=H^1(\Omega)\to \bbR^{M_0}$ and a linear isomorphism~$B\in\clL(\bbR^{M_0},\clU_M)$ as control operator, with}
 &\clU_{M_0}=\linspan\{\Phi_i,\mid 1\le i\le M_0\}\subset H=L^2(\Omega).
 \end{align}
\end{subequations}
The procedure presented hereafter can be used for homogeneous Dirichlet boundary conditions as well, by taking the appropriate matrices after spatial discretization.

\subsection{Discretization of the dynamical system}
As spatial discretization we consider piecewise linear finite-elements (based on
the classical hat functions), followed by a temporal discretization based on a Crank--Nicolson/Adams--Bashforth scheme. 
Briefly, for equations as~\eqref{sys-y-parab-clK},
let~$\bfS$ and~$\bfM$ be the stiffness and mass matrices and denote~$\bfS_{\nu}=\nu\bfS+\bfM$.
 Let~$\bfG_{x_i}$ be the discretizations of the directional derivatives~$\frac{\p}{\p x_i}$ and let~$\bfD_{\overline v}$ be the diagonal matrix, the entries of which
are those of the vector~$\overline v\in\bbR^{N\times 1}$, ~$(\bfD_{\overline v})_{(n,n)}=\overline v_{(n,1)}$. 
After spatial discretization we obtain 
\begin{align}\label{sys-y-D1}
 &\bfM\dot{\overline y} =-\bfS_\nu \overline  y-\tfrac{\bfM \bfD_{\overline a}
 +\bfD_{\overline a}\bfM}{2}\overline y
 -{\textstyle\sum\limits_{i=1}^d}(\bfD_{\overline b_i}\bfG_{x_i})\overline  y
 +\bfM\bfU\bfK\overline y,\qquad
\overline y(0)=\overline{y_0},
\end{align}
where~$\overline y(t)\in\bbR^{N\times 1}$ is the vector of values of the state at the spatial mesh (triangulation) points at time~$t\ge0$, and
\begin{equation}\label{bfU}
\bfU\coloneqq[\overline\Phi_{1}\dots \overline\Phi_{M_0}]\in\bbR^{N\times M_0}
\end{equation}
is the matrix the columns of which contain the finite-elements vectors corresponding to the actuators. Finally, $\bfu(t)=\bfK(t)\overline y(t)\in\bbR^{M_0\times 1}$ is the computed input feedback control~$u(t)=\clK(t) y(t)$ the computation of which shall be addressed in more detail
in sections~\ref{sS:comput-u-ricc} and~\ref{sS:comput-u-obli}.

Let us denote our finite dimensional finite-elements  space by
\begin{equation}\label{feHN}
 H_N=\linspan\{\fkh_n\mid 1\le n\le N\}\subset V\subset H,
\end{equation}
which is spanned by the hat functions~$\fkh_n$,
associated with the triangulation of the spatial domain~$\Omega$. Essentially, we
look for an approximation of the state~$y(\Bigcdot,t)$ as
\[
y(\Bigcdot,t) \approx\textstyle\sum\limits_{n=1}^N y(p_n,t)\fkh_n,
\quad \overline  y(t)\in\bbR^{N\times 1},\quad\overline  y_{(n,1)}(t)\coloneqq y(p_n,t),
\]
where~$\fkh_i$ is the hat function satisfying~$\fkh_i(p_i)=1$ and~$\fkh_j(p_j)=0$ for $j\ne i$, where the~$p_n$s, $1\le n\le N$, are the points in the mesh.
Denoting
\[
 \bfL^0\coloneqq\tfrac{\bfM \bfD_{\overline a}+\bfD_{\overline a}\bfM}{2},
 \qquad
 \bfL^1\coloneqq
 {\textstyle\sum\limits_{i=1}^d}(\bfD_{\overline b_i}\bfG_{x_i}),\qquad\ \bfF\coloneqq\bfM\bfU\bfK,
\]
after subsequent temporal
discretization, for a fixed time step~$k>0$ we find
\begin{align}
 \tfrac{\bfM {\overline y}_{j+1}-\bfM {\overline y}_{j}}{k}
 &=-\tfrac{\bfS_\nu  {\overline y}_{j+1}+\bfS_\nu  {\overline y}_{j}}{2}
  -\tfrac{\bfL^0_{j+1} {\overline y}_{j+1}+\bfL^0_j {\overline y}_{j}}{2}
  -\tfrac{\bfL^1_{j+1} {\overline y}_{j+1}+\bfL^1_j {\overline y}_{j}}{2}
  +\tfrac{\bfF_{j+1} {\overline y}_{j+1}+\bfF_j {\overline y}_{j}}{2}\notag
\end{align}
where the subscript integer~$j$ stands for evaluation at time~$t_j\coloneqq (j-1)k$,
\[
 {\overline y}_{j}={\overline y}(t_j),\quad\bfL^m_{j}=\bfL^m(t_j),\quad\bfF_{j}=\bfF(t_j),
 \qquad m\in\{0,1\},\quad j\ge1.
\]
Therefore, we arrive at
\begin{align}
 &\quad(2\bfM +k\bfS_\nu +k\bfL^0_{j+1}){\overline y}_{j+1}\notag\\
 &=(2\bfM -k\bfS_\nu -k\bfL^0_{j}){\overline y}_{j}
  -k(\bfL^1_{j+1} {\overline y}_{j+1}+\bfL^1_j {\overline y}_{j})
  +k(\bfF_{j+1} {\overline y}_{j+1}+\bfF_j {\overline y}_{j}).\notag
\end{align}
Next, we use a linear extrapolation for the unknown
terms in the right hand side, that is, we take~$ f(t_{j})+\bigl(f(t_{j})-f(t_{j-1})\bigr)$ as an approximation of~$f(t_{j+1})$,
which leads us to the implicit-explicit (IMEX)  scheme 
\begin{align}
 &(2\bfM +k\bfS_\nu +k\bfL^0_{j+1}){\overline y}_{j+1}=h_j,\quad\mbox{with}\label{sys-y-D2}\\
 &h_j\coloneqq(2\bfM -k\bfS_\nu -k\bfL^0_j){\overline y}_{j}
  -k(3\bfL^1_{j} {\overline y}_{j}-\bfL^1_{j-1} {\overline y}_{j-1})
  +k(3\bfF_{j} {\overline y}_{j}-\bfF_{j-1} {\overline y}_{j-1}),\notag
\end{align}
which we can solve to obtain~${\overline y}_{j+1}$, provided we know~$({\overline y}_{j-1},{\overline y}_{j})$.
An analogous IMEX time discretization is considered in~\cite{AscherRuuthWetton95}
for convection-diffusion equations, in~\cite{EthierBourgault08}
for the FitzHugh--Nagumo system, in~\cite{HeSun07} and~\cite[sect.~19]{MarionTemam98} for the
Navier--Stokes system, and in~\cite{ZhangJinHuangFu18} for the Burgers equation. 

Note that~${\overline y}_{1}=\overline{y_0}$ is given, at initial
time~$t=t_1=0$, however,
to start the solver/algorithm, in order to obtain~${\overline y}_{2}$ at time~$t=k$, we need the
``ghost'' state~${\overline y}_{0}$ ``at time~$t=t_0=-k$''. We have set/chosen~${\overline y}_{0}=\overline{y_0}$.

\begin{remark}
It is desirable that the matrix~$\bfA_{j+1}\coloneqq2\bfM +k\bfS_\nu +k\bfL^0_{j+1}$ ``to be inverted'' is sparse, symmetric and positive definite. Note/recall that both~$\bfM$ and~$\bfS_\nu$ are sparse, symmetric, and positive definite. Further, the reaction matrix~$\bfL^0_{j+1}$ is sparse and symmetric. Hence~$\bfA_{j+1}\coloneqq2\bfM +k\bfS_\nu +k\bfL^0_{j+1}$ has the desired properties, for small time-step~$k$. On the other hand,
 the feedback matrix~$\bfF_{j+1}$ may be not  a sparse matrix (as, in general, for the Riccati based feedback) and the convection matrix~$\bfL^1_{j+1}$
is not symmetric; these are the reasons why we do not include neither~$\bfF_{j+1}$ nor~$\bfL^1_{j+1}$ in the matrix~$\bfA_{j+1}$.
\end{remark}

\subsection{Solving the algebraic Riccati equation}\label{sS:computRic-alg}
We have seen that we need to solve equations as in~\eqref{Riccati-algTbk}
in order to compute~$\Pi=\Pi(t)$ (defined for time~$t\in[\tau,\tau+\varpi]$), from which we can  construct the 
feedback input operator~$\clK=\clK^{\rm ricc}=-\beta^{-1}B^*\Pi$ making system~\eqref{sys-y-parab-clK} exponentially stable.
This section is dedicated to the
 computation of a finite-elements approximation of equations as~\eqref{Riccati-alg},
  \begin{subequations}\label{Riccati-alg*}
 \begin{align}
 &\fkT_{\clT}(\Pi)=0,\qquad\Pi\succeq0,\quad\mbox{with}\quad
 \clT\coloneqq (\clA,\clB,\clC),\quad\mbox{and}\quad\\
 &\fkT_{\clT}(\Pi)
 \coloneqq \clA^*\Pi+\Pi  \clA-\Pi \clB \clB^{*}\Pi+\clC^{*}\clC.
\end{align}
 \end{subequations}
 
To solve~\eqref{Riccati-alg*} we shall  use a Newton method, as in the 
software/routines available in~\cite{BennerRicSolver}, see~\cite{Benner06}.
As we have mentioned in section~\ref{sS:homotopy},
a crucial point now concerns the choice of the initial guess~$\Pi_G$ to start the Newton iteration.
 and finding such a ``guess'' is a nontrivial task, see the
discussion in~\cite[after Eq.~(1.4)]{KesavanRaymond09},
in~\cite[sect.~3, Rem.~2]{BennerLiPenzl08}, and in~\cite[sect.~5.2]{BennerSaak13-gamm}.

To circumvent this issue we consider the homotopy as in~\eqref{Riccati-alg-s} and
proceed as we illustrate in Algorithm~\ref{Alg:ARE}, where we connect the Riccati data triples
 \[
(-A,\beta^{-\frac12}B,\clC)\quad\mbox{and}\quad  (-A-A_{\rm rc}(\tau+\varpi)+\overline\mu\Id,\beta^{-\frac12}B,\clC).
 \]
 Observe that~$\beta^{-1}BB^*=\beta^{-\frac12}B(\beta^{-\frac12}B)^{*}$.
 
\begin{algorithm}[ht]
 \caption{Homotopy for algebraic Riccati equation~\eqref{Riccati-alg}}
\begin{algorithmic}[1]\label{Alg:ARE}
\REQUIRE{Riccati data~$(A,A_{\rm rc},B,\clC,\beta,\overline\mu)$ and homotopy step~$\delta_s\in(0,1]$.}
\ENSURE{$\Pi_{\tau+\varpi}$, with~$\fkT_{\clT}(\Pi_{\tau+\varpi})=0$, with~$\clT=(-A-A_{\rm rc}+\overline\mu\Id,\beta^{-\frac12}B,\clC)$.}
\STATE  Set~$i=0$;
\STATE  Set~$\Pi^{\rm old}=\zero$;
\WHILE{$i\delta_s\le1$}
\STATE Set~$\widehat\clT=(-A-i\delta_s(A_{\rm rc}-\overline\mu\Id),\beta^{-\frac12}B,\clC)$;
\STATE Solve~$\fkT_{\widehat\clT}(\Pi_{\tau+\varpi})=0$, with initial guess~$\Pi^{\rm old}$;
\STATE  Set~$\Pi_T^{\rm old}=\Pi_{\tau+\varpi}$;
\STATE Shift~$i\to i+1$;
 \ENDWHILE
 \IF{$(i-1)\delta_s<1$}
 \STATE Set~$\widehat\clT=(-A-(A_{\rm rc}-\overline\mu\Id),\beta^{-\frac12}B,\clC)$;
\STATE Solve~$\fkT_{\widehat\clT}(\Pi_{\tau+\varpi})=0$, with initial guess~$\Pi^{\rm old}$.
 \ENDIF
 \end{algorithmic} 
\end{algorithm} 

\begin{remark}\label{R:Ric_initGuess} To see why a stabilizing initial guess is 
important for solving algebraic Riccati equations as~\eqref{Riccati-alg} we can observe the following.
After discretization, we will solve a matrix equation as
\begin{equation}\label{Riccati-alg-mat}
 \bfX^\top\mathbf\Pi+\mathbf\Pi  \bfX-\mathbf\Pi\bfB\bfB^\top\mathbf\Pi+\bfC^\top\bfC=0,\qquad\mathbf\Pi\succeq0,
\end{equation}
and look for~$\mathbf\Pi=\mathbf\Pi^\top$ through a Newton--Kleinman iteration,
see~\cite[Equ.~(18)]{BennerSaak13-gamm},
\begin{align}
&\mathbf\Pi_0=\bfG=\bfG^\top,\quad\bfG\succeq0\notag\\
&\bfX_i= \bfX-\bfB\bfB^\top\mathbf\Pi_i,\quad \bfX_i^\top\mathbf\Pi_{i+1}+\mathbf\Pi_{i+1}\bfX_i+\mathbf\Pi_i\bfB\bfB^\top\mathbf\Pi_i+\bfC^\top\bfC=0,\label{Lyap-iter}
\end{align}
where we look for a nonnegative definite solution~$\mathbf\Pi_{i+1}$ for the Lyapunov equation in~\eqref{Lyap-iter}.
Let us now assume for simplicity that~$B=\Id$, that~$\bfX=\bfX^\top$, and that $\bfX$ has an eigenvalue~$\zeta>0$, $\bfX\bfv=\zeta\bfv$, with~$\bfv\ne0$. Then, if~$0\le\varepsilon<\zeta$ the initial guess~$\mathbf\Pi_0=\bfG=\varepsilon\Id$ is not stabilizing and not appropriate. Indeed, it follows that~$\zeta-\varepsilon>0$ is an eigenvalue of~$\bfX_0=\bfX-\varepsilon\Id$, hence~$\bfX_0=\bfX_0^\top$ is not stable.

Suppose that, with~$i=0$, there exists a nonnegative definite solution~$\mathbf\Pi_1$ for~\eqref{Lyap-iter}, then~$\bfH^\delta\coloneqq-\mathbf\Pi_1-\delta\Id$ is negative definite for all~$\delta>0$, and solves
\[
\bfX_0^\top\bfH^\delta+\bfH^\delta\bfX_0=-2\delta\bfX_0+\varepsilon^2\Id+\bfC^\top\bfC.
\]
Now, we can choose~$\delta>0$ small enough such that
 $-2\delta\bfX_0+\varepsilon^2\Id+\bfC^\top\bfC\succ0$ is positive definite, and by
the result in~\cite[sect.~13.1, Thm.~1(b)]{LancasterTismenetsky85}, we must have that~$\bfX_0$ is stable, which is a contradiction. Therefore, for the initial guess~$\mathbf\Pi_0=\varepsilon\Id$,
there will be no nonnegative definite solution~$\mathbf\Pi_1$ for the first iteration in~\eqref{Lyap-iter}.
\end{remark}

\begin{remark}
In Algorithm~\ref{Alg:ARE}, we propose to find the symmetric positive definite solution of the algebraic equation by solving a sequence of algebraic Riccati equations starting by solving an algebraic Riccati equation for which  finding a stabilizing initial guess is easier, namely, the zero feedback. Note that to compute our feedback control input~$u=\clK^{\rm ricc} y$, we need only the product~$\clK^{\rm ricc}=-B^*\Pi$. One approach to find~$\clK^{\rm ricc}$ directly is to use a  Chandrasekhar iteration as in~\cite[sect.~2]{BanksIto91}. We refer also the reader to the partial stabilization Bernoulli equation based approach in~\cite{Benner11-ch3}.
Computing the solution of the algebraic Riccati equation has, however, the advantage to give us a way to compute an approximation of the optimal cost as~$\tfrac12\overline y_0^\top\mathbf\Pi\overline y_0\approx\tfrac12(\Pi y_0,y_0)_H=\clJ^\beta_{M_1}(y_0;\widehat y,\widehat u)$, hence without solving (say, in a large time interval) the corresponding autonomous feedback control dynamical system issued from the initial state~$y(0)=y_0$. 
\end{remark}

 \subsection{Solving the time-periodic differential Riccati equation}\label{sS:computRic-diff}
 Once we have computed (e.g., with Algorithm~\ref{Alg:ARE}) a solution~$\Pi_{\tau+\varpi}$ for~\eqref{Riccati-alg},
 we can then solve the differential equation~\eqref{Riccati-Tbk},  backwards in time,
 \begin{align}
 &\dot\Pi+ X^*\Pi+\Pi  X-\beta^{-1}\Pi B  B^{*}\Pi+\clC^{*}\clC=0,\quad \Pi(\tau+\varpi)=\Pi_{\tau+\varpi},\label{Riccati-Tbk*}\\
 &\Pi(t)\succeq0,\quad\mbox{for all}\quad t\in[\tau,\tau+\varpi].\notag
\end{align}

Recall that for autonomous systems, where~$X$ is independent of time,
we have that the solution of the differential Riccati
equation~\eqref{Riccati}
is in fact time-independent and coincides with the solution of the algebraic Riccati equation~\eqref{Riccati-alg}. For time-periodic~$X(t)$ with
period~$\varpi>0$, $X(t+\varpi)=X(t)$ for all~$t\ge0$,
then the optimal feedback, which solves of the differential Riccati
equation~\eqref{Riccati}, is also periodic in time with the same
period, $\Pi(t+\varpi)=\Pi(t)$ for all~$t\ge0$.
Let us denote
\begin{align}
 \fkT_{\clT}^{\rm per}(\Pi)
 &\coloneqq\left(
      \fkT_{\clT}^{{\rm per},1}(\Pi),
      \fkT_{\clT}^{{\rm per},2}(\Pi)\right),\qquad\clT\coloneqq (\clA, \clB,\clC),\notag
\intertext{with}
 \fkT_{\clT}^{{\rm per},1}(\Pi)&\coloneqq
      \dot\Pi+  \clA^*\Pi+\Pi \clA-\beta^{-1}\Pi B B^{*}\Pi+\clC^{*}\clC,\qquad t\in[\tau,\tau+\varpi]\notag\\
  \fkT_{\clT}^{{\rm per},2}(\Pi)&\coloneqq\norm{\Pi(\tau)-\Pi(\tau+\varpi)}{\clL(H)}.\notag
 \end{align}
 
 Note that~$\Pi$ solves the periodic Riccati equation if
\begin{equation}
 \fkT_{\clT}^{\rm per}(\Pi)=(\zero,0).\notag
\end{equation}
 
To compute the periodic Riccati solution we followed Algorithm~\ref{Alg:PRE}, which is motivated by Theorem~\ref{T:convRicPer}.  In particular, note that we start at time~$t=\tau+\varpi$ with the matrix~$\mathbf\Pi_{R}=\mathbf\Pi_{\tau+\varpi}$  solving the algebraic Riccati equation~$\fkT_{(\bfX(\tau+\varpi), \bfB,\bfC)}(\mathbf\Pi_{R})=0$.

\begin{algorithm}[ht]
 \caption{Solution for $\varpi$-periodic Riccati feedback~\eqref{Riccati-Tbk}}
\begin{algorithmic}[1]\label{Alg:PRE}
\REQUIRE{$(\clT,\varpi,\varepsilon, \overline n)$, where~$\varpi>0$,
$\varepsilon>0$, $\overline n\in\bbN$, and~$\clT(t)=(X(t),\beta^{-\frac12}B,\clC)$,
with~$X(t)=X(t+\varpi)$ for all~$t\in\bbR$.}
\ENSURE{$\Pi$ solving~\eqref{Riccati-Tbk} with $\Pi(\tau)=\Pi(\tau+\varpi)$.}
\STATE Set $n=0$;
\STATE Use Algorithm~\ref{Alg:ARE} to solve~$\fkT_{(X(\tau+\varpi), \beta^{-\frac12}B,\clC)}(\Pi_R)=0$;
\STATE Solve~$\fkT_{\clT}^{{\rm per},1}(\Pi)=0$, backwards for~$t\in[\tau,\tau+\varpi]$,
with~$\Pi(\tau+\varpi)=\Pi_{R}$;
\STATE Set $P_L=\Pi(\tau)$;
\STATE Set $e=\norm{\Pi_L-\Pi_R}{\clL(H)}$;
\WHILE {$e>\varepsilon$ and $n<\overline n$}
\STATE Set $P_R=P_L$;
\STATE Solve~$\fkT_{\clT}^{{\rm per},1}(\Pi)=0$, backwards for~$t\in[\tau,\tau+\varpi]$,
with~$\Pi(\tau+\varpi)=\Pi_R$;
\STATE Set $P_L=\Pi(\tau)$;
\STATE Set $e=\norm{P_L-P_R}{\clL(H)}$;
\STATE Shift $n=n+1$;
\ENDWHILE
\end{algorithmic} 
\end{algorithm} 

For further works on matrix periodic Riccati equations,
we refer to~\cite{Varga08,GusevJohaKagsShirVarga10}.

\subsection{Spatial discretization of the Riccati equations} \label{sS:discRiccati-space}
 We look for a  symmetric positive semidefinite matrix~$\mathbf\Pi\in\bbR^{N\times N}=\clL(\bbR^{N})\sim\clL(H_N)$, representing the symmetric positive definite linear
 continuous operator~$\Pi\in\clL(H)$ in our piecewise linear finite-elements space~$H_N$,
\[
 \overline z^\perp\bfR \overline y=0\quad\mbox{with}\quad
 (\Pi y,z)_H=\overline z^\perp\mathbf\Pi \overline y,\quad\mbox{for all}\quad (y, z)\in H_N\times H_N,
\]
where~$\bfR$ is as
\[
 \overline z^\perp\bfR \overline y=\langle(\dot\Pi+ X^*\Pi+\Pi  X-\beta^{-1}\Pi B B^{*}\Pi+\clC^{*}\clC) y,z\rangle_{V',V}.
\]
We can see that algebraic computations lead us to the semi-discrete equation
\begin{align}\notag
\bfR=\dot{\mathbf\Pi}+\bfX^\top \mathbf\Pi+\mathbf\Pi \bfX
-\mathbf\Pi\overline\bfB \bfB^\top \mathbf\Pi
+\bfC^\top\bfC=0,\quad t>0,
\end{align}
with~$\bfX=-\bfM^{-1}(\bfS_\nu +\bfL^0+\bfL^1)+\overline\mu\Id$ and where~$\bfB$ and~$\bfC$ satisfy
\begin{subequations}\label{bfBbfC}
\begin{align}
  (\clC^{*}\clC y,z)_H&=\overline z^\top\bfC^\top\bfC \overline y,\\
   \beta^{-1}(B  B^{*}y,z)_H&=\overline z^\top \bfM\bfB\bfB^\top\bfM \overline y,
\end{align}
\end{subequations}
for all~$(y, z)\in H_N\times H_N$. Indeed, the above equation can be obtained by the following observation. 
If $\bfP\in\bbR^{N\times N}$ and~$P\colon H\to H$ satisfy
\[
\overline z^\top\bfP\overline y=(P y,z)_H,\quad\mbox{for all}\quad (y, z)\in H_N\times H_N,
\]
then we can write
\[
\overline z^\top\bfP\overline y=(P y,z)_H
=\overline z^\top\bfM\overline{P}
\overline y,\quad\mbox{with}\quad \overline{P}\coloneqq\bfM^{-1}\bfP.
\]

For a given~$y\in H_N$, we define the vector~$\overline{Py}\in\bbR^{N\times 1}$ as
\[
\overline{Py}\coloneqq\overline{P}\overline y.
\]
Note that~$\overline{Py}$ is the unique vector~$\overline{w}$ satisfying
\[\langle P y, z\rangle_{X',X}\eqqcolon\overline z^\top\bfM \overline{w}\quad\mbox{for all}\quad
z\in H_N.
\]

\begin{example}
If~$P=\Id_H$, we have that~$\bfP=\bfM$ is the mass matrix and~$\overline P=\Id_{\bbR^N}$.
\end{example}

For the composition~$\clC^{*}\clC$ we consider the two cases~$\clC^{*}\clC\in\left\{\Id, P_{\clE^{\rm f}_{M_1}}\right\}$.

For the case~$\clC^{*}\clC=\Id$ we can write
\[
(\clC^{*}\clC y,z)_H=\overline z^\top\bfM\overline y.
\]

For the case~$\clC^{*}\clC=P_{\clE^{\rm f}_{M_1}}$ we can find (cf.~\cite[Lem.~2.8]{KunRod19-cocv})
\begin{equation}\label{orthProj-disc}
\overline{P_{\clE^{\rm f}_{M_1}}}=\overline{\clC^{*}\clC}=\bfE^{\rm f} \bfV^{-1}(\bfE^{\rm f})^\top\bfM,
\end{equation}
 where~$\bfE^{\rm f}\in\bbR^{N\times M_1}$ is the matrix whose columns contain the (vectors corresponding to the) $M_1$~eigenfunctions, $\overline e_i\in\bbR^{N\times 1}$, $1\le i\le M_1$, spanning~$\clE^{\rm f}_{M_1}$
and where~$\bfV\in\bbR^{M_1\times M_1}$ is the matrix
with entries~${\bfV_{M_1}}_{(i,j)}=\overline  e_i^\top\bfM\overline e_j$. Thus, 
we find
\[
(\clC^{*}\clC y,z)_H=\overline z^\top\bfM\overline{\clC^{*}\clC}\overline y=\overline z^\top\bfM\bfE^{\rm f} \bfV^{-1}(\bfE^{\rm f})^\top\bfM\overline y.
\]

Resuming we can take
\begin{subequations}\label{C-disc}
\begin{align}
&\bfC=\bfM_\bfc,&&\quad\mbox{if}\quad\clC^{*}\clC=\Id;\\
&\bfC=(\bfV^{-1})_{\bfc}(\bfE^{\rm f})^\top\bfM,&&\quad\mbox{if}\quad\clC^{*}\clC=P_{\clE^{\rm f}_{M_1}};
\end{align}
\end{subequations}
where, the subscript~$\bfc$ stands for the Cholesky factor of a given symmetric positive definite matrix~$Z$, satisfying
\begin{equation}\label{chol}
Z=Z_\bfc^\top Z_\bfc.
\end{equation}
 
Next for the term involving the control operator,  with~$\clB=\beta^{-\frac{1}{2}}B$ we find
\begin{align}
(\Pi \clB  \clB^{*}\Pi y,z)_H&=\overline z^\top\bfM\overline{\Pi\clB  \clB^{*}\Pi y}=\overline z^\top\bfM\overline{\Pi}\,\overline{\clB  \clB^{*}\Pi y}
=\overline z^\top\mathbf\Pi\overline{\clB\clB^{*} }\,\overline{\Pi y}\notag\\
&=\overline z^\top\mathbf\Pi\overline{\clB\clB^{*} }\,\bfM^{-1}\mathbf\Pi\overline y.\notag
\end{align}
Hence, if we find~$\bfB$ as in~\eqref{bfBbfC}, we can write~$\overline{\clB\clB^{*} }=\bfM^{-1} (\bfM\bfB\bfB^\top\bfM)$ and
\begin{align}
(\Pi \clB  \clB^{*}\Pi y,z)_H&\approx\overline z^\top\mathbf\Pi\bfB\bfB^\top\mathbf\Pi\overline y.\notag
\end{align}
It remains to find a suitable operator~$\bfB$ as in~\eqref{bfBbfC}. For this purpose, we note that
\[
 (Bu,z)_H= ({\textstyle\sum_{j=1}^{M_0}}u_j\Phi_j,z)_H =(u,v^z)_{\bbR^{M_0}},
\]
where~$v^z=(v_1,\dots,v_{M_0})$ is the vector with coordinates~$v^z_j\coloneqq(\Phi_j,z)_H$. Hence
\[
B^*z=v^z\quad\mbox{and}\quad \beta^{-1}(B  B^{*}y,z)_H=(\beta^{-\frac12}B^*y, \beta^{-\frac12}B^{*}z)_{\bbR^{M_0}},
\]
which leads us to
\[
 \beta^{-1}(B  B^{*}y,z)_H\approx (\bfB^\top\bfM \overline z)^\top \bfB^\top\bfM \overline y=\overline z^\top\bfM\bfB\bfB^\top\bfM \overline y,
\]
where~$\bfB=\beta^{-\frac12}\bfU\in\bbR^{N\times M_0}$
and~$\bfU$ is the matrix the columns of which contain the finite-elements vectors corresponding to the actuators, as in~\eqref{bfU}.

\subsection{Solving the differential Riccati equations. Time discretization} \label{sS:discRiccati-time}
Here we restrict ourselves to the case case ~$\clC^*\clC=\Id$, hence~$\bfC^\top\bfC=\bfM$; see~\eqref{C-disc}. To solve~\eqref{Riccati-algTbk} we shall first solve the algebraic matrix  equation
\begin{subequations}\label{DiscRicc}
\begin{align}
&\bfX^\top({\tau+\varpi}) \mathbf\Pi+\mathbf\Pi \bfX({\tau+\varpi})
-\mathbf\Pi\bfB \bfB^\top \mathbf\Pi
+\bfC^\top\bfC=0,\label{DiscRicc-alg}
\intertext{and then solve, backwars in time, the differential  matrix  equation}
&\dot{\mathbf\Pi}+\bfX^\top \mathbf\Pi+\mathbf\Pi \bfX
-\mathbf\Pi\bfB \bfB^\top \mathbf\Pi
+\bfC^\top\bfC=0,,\qquad\mathbf\Pi(\tau+\varpi)=\mathbf\Pi_{\tau+\varpi},\label{DiscRicc-diff}
\intertext{for time~$t\in[\tau,\tau+\varpi]$, with}
&\bfX(t)=-\bfM^{-1}(\bfS_\nu +\bfL^0(t)+\bfL^1(t))+\overline\mu\Id,\quad 
\bfB\in\bbR^{N\times M_0},\quad
\bfC\in\bbR^{N\times N},
\intertext{where}
&\bfB=\beta^{-\frac12}\bfU \quad\mbox{and}\quad\bfC=\bfM_\rmc,\quad\mbox{with}\quad\bfU =[\overline\Phi_{1}\;\dots\;\overline\Phi_{M_0}].
\end{align}
\end{subequations}
 
 To solve such differential equation, we set a positive time step
\begin{equation}\label{k_ric}
k_{\rm ric}\le\varpi.
\end{equation}
Then, we set the integer~$\lfloor\frac{\varpi}{k_{\rm ric}}\rfloor\ge1$, that is, the integer floor of~$\frac{\varpi}{k_{\rm ric}}$ defined as
\[
\lfloor s\rfloor\in\bbN,\qquad \lfloor s\rfloor\le s<\lfloor s\rfloor+1,\qquad\mbox{for}\quad s\in[0,+\infty),
\]
and a new time step~$\overline k_{\rm ric}$ as
\begin{equation}\label{bark_ric}
k_{\rm ric}\le\overline k_{\rm ric}\coloneqq \frac{\varpi}{\lfloor\frac{\varpi}{k_{\rm ric}}\rfloor}\le\varpi,
\end{equation}
 obtaining the temporal mesh
\begin{equation}\label{tmesh-ricc}
\tau=\overline t_1<\overline t_2<\dots <\overline t_{\#\overline t-1}<\overline t_{\#\overline t}=\tau+\varpi,\qquad \overline t_r=\tau+(r-1)\overline k_{\rm ric},
\end{equation}
where~$\#\overline t=\lfloor\frac{\varpi}{k_{\rm ric}}\rfloor+1$ is the number of elements in~$\overline t$.

 Inspired in a Crank--Nicolson scheme, we set
\begin{align}
 &\dot{\mathbf\Pi}(\tfrac{\overline t_r+\overline t_{r+1}}{2})\approx \tfrac{\mathbf\Pi^{r+1}-\mathbf\Pi^{r}}{\overline k_{\rm ric}},\qquad
 {\fkR(\mathbf\Pi)}(\tfrac{\overline t_r+\overline t_{r+1}}{2})\approx \tfrac{{\fkR(\mathbf\Pi)}(\overline t_{r+1}) +{\fkR(\mathbf\Pi)}(\overline t_r)}{2},\notag
 \intertext{with~$1\le r< \#\overline t-1$,}
 &\mathbf\Pi^r\coloneqq\mathbf\Pi(\overline t_r),\quad \mathbf\Pi^{r+1}\coloneqq\mathbf\Pi(\overline t_{r+1}),\quad\mbox{and}\notag\\
 &\fkR(\mathbf\Pi)(t)\coloneqq\bfX(t)^\top \mathbf\Pi(t)+\mathbf\Pi(t) \bfX(t)
-\mathbf\Pi(t)\bfB \bfB^\top \mathbf\Pi(t)
+\bfC^\top\bfC,\notag
\end{align}
 we find that~$2\tfrac{\mathbf\Pi^{r+1}-\mathbf\Pi^{r}}{\overline k_{\rm ric}}
 +\fkR(\mathbf\Pi)(t_{r+1})+\fkR(\mathbf\Pi)(t_{r})=0$,
 that is,
 \begin{subequations}\label{Ric-discj}
  \begin{align}
&\bfY^\top_r \mathbf\Pi^{r}+\mathbf\Pi^{r} \bfY_r
-\mathbf\Pi^{r}\bfB \bfB^\top \mathbf\Pi^{r}+\overline\bfC^\top_{r+1}\overline\bfC_{r+1}=0\label{Ric-discj-sch}
\intertext{with}
&\bfY(t)\coloneqq \bfX(t)-\tfrac{1}{\overline k_{\rm ric}}\Id=-\bfM^{-1}\left(\bfS_\nu 
+\bfL^0(t)+\bfL^1(t)+(\tfrac{1}{\overline k_{\rm ric}}-\overline\mu)\bfM\right),\notag\\
&\bfY_r\coloneqq\bfY(\overline t_r),\quad\mbox{and}\\
&\overline\bfC_{r+1}^\top\overline\bfC_{r+1}\coloneqq \bfQ_{r+1}\coloneqq
\fkR(\mathbf\Pi)(\overline t_{r+1})+2\bfC^\top\bfC+\tfrac2{\overline k_{\rm ric}}\mathbf\Pi^{r+1}.\label{Ric-discj-Q}
\end{align}

Since~$\bfC^\top\bfC=\bfM$ is positive definite, then~$\mathbf\Pi^{r+1}$ (defining the optimal cost to go) is positive definite. Hence, for small enough~$k_{\rm ric}$ the matrix~$\bfQ_{r+1}$ is symmetric and positive definite as well. Then we can choose~$\overline\bfC_{r+1}=\bfQ_{\rmc}$ as the Cholesky
factor of~$\bfQ_{r+1}=\bfQ_{\rmc}^\top\bfQ_{\rmc}$.
Finally, to find~$\mathbf\Pi^{r}$ we solve~\eqref{Ric-discj-sch}, where as starting point
for the Newton iteration
we will choose the natural guess
\begin{equation}
 \mathbf\Pi^r_0\coloneqq \mathbf\Pi^{r+1}.
\end{equation}
 \end{subequations}

In general, it may be hard to tell whether the chosen time step~$k_{\rm ric}$ will be 
small enough at each discrete time step (so that~$\bfQ_{r+1}$ is  positive definite). Thus, we used the scheme above with a dynamic time-step as in Algorithm~\ref{Alg:DRE}, leading us to a (possibly) nonuniform time mesh~$\overline t^{\rm m}$,
\begin{subequations}\label{F_Ric-disc}
\begin{equation}
\tau=\overline t^{\rm m}_1<\dots<\overline t^{\rm m}_{r-1}<\overline t^{\rm m}_{r}<\dots<\overline t^{\rm m}_{\#\overline t^{\rm m}}=\tau+\varpi
\end{equation}
and to the discretization~$\bfF(\overline t^{\rm m}_r)$, at time~$\overline t^{\rm m}_r$, ~$1\le r\le\#\overline t^{\rm m}$, of the input Riccati
operator~$-\beta^{-1}B  B^*\Pi(\overline t^{\rm m}_r)$, we need in our scheme~\eqref{sys-y-D2}, as
\begin{equation}
\bfF(\overline t^{\rm m}_r)=\bfM\bfU\bfK(\overline t^{\rm m}_r),\quad\mbox{with}\quad\bfK(\overline t^{\rm m}_r)\coloneqq-\beta^{-1}\bfU^\top\mathbf\Pi(\overline t^{\rm m}_r)\in \bbR^{ M_0\times N}.
\end{equation}
\end{subequations}

\begin{algorithm}[ht]
 \caption{Solution for  differential Riccati feedback~\eqref{Ric-discj}}
\begin{algorithmic}[1]\label{Alg:DRE}
\REQUIRE{$(\bfX(t),\bfB,\bfC,\mathbf\Pi_{\tau+\varpi},\tau,\varpi,k_{\rm ric})$, $\tau\le t\le\tau+\varpi$, $\tau>0$,  $0<k_{\rm ric}\le\varpi$,
 and positive definite product~$\bfC^\top\bfC$.}
\ENSURE{a positive definite $\mathbf\Pi(\overline t^{\rm m})$ solving~\eqref{Ric-discj} with $\Pi(\tau+\varpi)=\mathbf\Pi_{\tau+\varpi}$, in a temporal mesh $\overline t^{\rm m}$ of~$[\tau,\tau+\varpi]$.}
\STATE Set $n=0$ and $\overline k_{\rm ric}$ as in~\eqref{bark_ric};
\STATE Set $\overline t^{\rm m}=[\tau+\varpi]\in\bbR^{1\times 1}$ and $T=\tau+\varpi$;
\STATE Set $\mathbf\Pi(\tau+\varpi)=\mathbf\Pi_{\tau+\varpi}$ and ~$\mathbf\Pi_{\rm old}=\mathbf\Pi_{\tau+\varpi}$;
\WHILE {$T>\tau$}
\STATE Set $k=\overline k_{\rm ric}$ and $p=0$;
\WHILE {$p=0$}
\STATE Set $\bfQ=
\fkR(\mathbf\Pi_{\rm old})+2\bfC^\top\bfC+\tfrac2{k}\mathbf\Pi_{\rm old}$;
\IF{$\bfQ$ is positive definite}
\STATE Set $p=1$ and $\overline \bfC=\bfQ_{\bfc}$;
\ELSE
\STATE Set $k=\min\{\tfrac12k, T-\tau\}$ and $T=T-k$;\label{Alg:DRE:squeze-kric}
\ENDIF
\ENDWHILE
\STATE Set~$\bfY\coloneqq \bfX(T)-\tfrac{1}{k}\Id$;
\STATE Solve~$\bfY^\top \mathbf\Pi_{\rm new}+\mathbf\Pi_{\rm new} \bfY
-\mathbf\Pi_{\rm new}\bfB \bfB^\top \mathbf\Pi_{\rm new}+\overline\bfC^\top\overline\bfC=0$;
\STATE Concatenate~$\overline t^{\rm m}=[T\;\; \overline t^{\rm m}]$ and set $\mathbf\Pi(T)=\mathbf\Pi_{\rm new}$;
\STATE Set~$\mathbf\Pi_{\rm old}=\mathbf\Pi_{\rm new}$;
\ENDWHILE
\end{algorithmic} 
\end{algorithm}

\subsection{Computation of the Riccati input control} \label{sS:comput-u-ricc}
Suppose that we have computed the Riccati solution~$\Pi$ for a given spatial triangulation/mesh and for a given temporal discretization. Now, we show how we can compute the input control coordinates~$u$ for simulations of the evolution of our controlled system performed in refinements of such triangulation and for a possibly different temporal discretization.

\subsubsection{Using the Riccati feedback in finer spatial discretizations}
The computation time increases with the number of degrees of freedom.
In order to speed the computations up, an option could be to compute the input Riccati operator~\eqref{F_Ric-disc}
in a given initial spatial mesh, and then use it to construct a corresponding feedback for
refinements of that mesh. If the initial mesh is not too coarse, such construction will
give us a stabilizing feedback for
the refined meshes as well, as illustrated in simulations presented hereafter. 

Let~$\widehat p$ be the set of~$\widehat N$ points of a given  mesh, which is refined to obtain a new mesh
with
points~$p=\widehat p\cup \breve p$, where~$\breve p$ is a set of additional~$\breve N$ points.
Let
\[
\widehat\bfF(\overline t^{\rm m}_r)=\widehat \bfM\widehat \bfU\widehat \bfK(\overline t^{\rm m}_r)\in\bbR^{ N\times  N}
\]
 be as in~\eqref{F_Ric-disc} for the coarse initial mesh.
We assume that the points~$\widehat p$ of the coarse mesh
correspond to the first coordinates in the refined mesh, and that the order of the points~$\widehat p$ is unchanged.
Note that the operator~$\widehat\bfK(\overline t^{\rm m}_r)$ gives us the actuator tuning parameters
$u_i=u_i(t)\in\bbR$, $1\le i\le M_0$, for the feedback control~$Bu$ in the coarse mesh; see~\eqref{B_Phi}.
Now, we simply propose to use these parameters in refined meshes, with~$N=\widehat N+\breve N$ points, by using the discrete feedback input control
\begin{equation}\label{discu-ricc0}
\bfu^{\rm ricc}(\overline t^{\rm m}_r)\coloneqq \widehat \bfK(\overline t^{\rm m}_r)\Xi_{\widehat N}^{N}\overline y^{\rm ricc}(\overline t^{\rm m}_r)\in\bbR^{ M_0\times  1},\qquad\mbox{for~\eqref{uy-ricc}},
\end{equation}
where~$\Xi_{\widehat N}^{N}$ is the matrix projection/mapping collecting the coordinates of~$\overline y$ corresponding to the points in the coarse mesh as follows,
\[
\Xi_{\widehat N}^{N}\coloneqq\begin{bmatrix}\Id_{\widehat N\times \widehat N}& & \zero_{\widehat N\times \breve N}\end{bmatrix}\in\bbR^{\widehat N\times N},\qquad\Xi_{\widehat N}^{N}\colon\overline y
=\begin{bmatrix}\overline y(1,1)\\ \overline y(2,1)\\ \vdots\\ \overline y(N,1)\end{bmatrix}
\mapsto\begin{bmatrix}\overline y(1,1)\\ \overline y(2,1)\\ \vdots\\ \overline y(\widehat N,1)\end{bmatrix}.
\]

\subsubsection{Changing the temporal discretization}
Assume we have computed the Riccati feedback input operator~$\bfK(\overline t^{\rm m}_r)$ for a given
temporal discretization, 
\begin{equation}\notag
0\le\tau=\overline t_1^\rmm<\dots<\overline t_{r}^\rmm<\dots<\overline t_{\#\overline t^\rmm}^\rmm=\tau+\varpi,\qquad \#\overline t^\rmm\ge2,
\end{equation}
of the time interval~$[\tau,\tau+\varpi]$ (cf.~Algorithm~\ref{Alg:DRE} and~\eqref{F_Ric-disc}). We  can still perform
simulations for a different temporal discretization of~$[0,T]$, $T>0$. For this purpose, we proceed as follows. 
Let~$t_j^\rmm\in[0,T]$ be a discrete time in a temporal mesh~$t^\rmm$ as
\[
0= t_1^\rmm< t_2^\rmm<\dots< t_{\#t^\rmm-1}^\rmm< t_{\#t^\rmm}^\rmm=T.
\]
Then at time~$ t_j^\rmm$, $1\le j\le \#t^\rmm$, we use a convex combination (linear interpolation) based on the Riccati temporal mesh as follows,
\begin{subequations}\label{discu-ricc}
\begin{align}
&\bfu^{\rm ricc}(t_j^\rmm)= \left((1-\theta)\widehat\bfK(\overline t^{\rm m}_{r_j})+\theta\widehat\bfK(\overline t^{\rm m}_{r_j+1})\right)\Xi_{\widehat N}^{N}\overline y^{\rm ricc}(t_j^\rmm)\in\bbR^{ M_0\times  1},\\
 &\mbox{for~\eqref{uy-ricc}}.\hspace{1em}\mbox{With}\quad
 \theta=\tfrac{ t_j^\rmm -\varpi\bigl\lfloor\tfrac{ t_j^\rmm-\tau}{\varpi}\bigr\rfloor
 -\overline t_{r_j}^\rmm}{\overline t_{r_j+1}^\rmm-\overline t_{r_j}^\rmm}\quad\mbox{and}\quad r_j\in\bbN_+\quad\mbox{such that}\quad\label{discu-ricc-theta}\\
 &\hspace{0em}r_j\coloneqq\min\left\{r_{j0}\mid1\le r_{j0}\le\#\overline t^\rmm-1\mbox{ and } \overline t_{r_{j0}}^\rmm\le  t_{j0}^\rmm-\varpi\bigl\lfloor\tfrac{ t_{j0}^\rmm-\tau}
 {\varpi}\bigr\rfloor\le \overline t_{r_{j0}+1}^\rmm\right\}.\label{discu-ricc-rj}
\end{align}
 \end{subequations}
 
Note that, by $\varpi$-periodicity we have that~$\bfK^e(t_j^\rmm)=\bfK^e( t_j^\rmm-\varpi\bigl\lfloor\tfrac{t_j^\rmm+\tau} {\varpi}\bigr\rfloor)$ and from~$t_j^\rmm-\varpi\bigl\lfloor\tfrac{t_j^\rmm-\tau}
 {\varpi}\bigr\rfloor\in[\tau,\tau+\varpi]$ it follows that
there exists one, and only one, $\overline t_{r_j}^\rmm$ in the Riccati temporal discretization
satisfying~\eqref{discu-ricc-rj}. In particular,~$\theta\in[0,1]$.

\subsection{Computation of the oblique projection input control}\label{sS:comput-u-obli}
Here we address the discrete version of the explicit feedback in~\eqref{FeedKunRod}.
Essentially, what remains is the construction of the oblique projection~$P_{\clU_M}^{\clE_M^\perp}$, which is analogous to that
of an orthogonal projection as in~\eqref{orthProj-disc}, and reads
\[
 P_{\clU_M}^{\clE_M^\perp}y\approx\bfU \widetilde\bfV^{-1}\bfE^\top\bfM\overline y,
\]
where, together with the matrix~$\bfU\in\bbR^{N\times M_0}$, whose columns contain our~$M_0=M_\sigma$ (vector)
actuators~$\overline \Phi_{j}=\overline\indf_{\omega_j}\in\bbR^{N\times 1}$, we consider also the matrix~$\bfE\in\bbR^{N\times M_0}$,
whose columns contain our~$M_0$ (vector)
auxiliary eigenfunctions~$\overline{e}_{M,i}\in\bbR^{N\times 1}$.
Now,~$\widetilde\bfV\in\bbR^{M_0\times M_0}$ stands for the matrix
whose entries are~${\widetilde\bfV}_{(i,j)}=\overline{e}_{M,i}^\top\bfM\overline\Phi_{j}$. For more details,
see~\cite[sect.~8]{RodSturm20}.
The discretized
feedback control input reads,
at time~$t_j^\rmm$, 
\begin{align}\label{discu-obli}
 &\bfu^{\rm obli}(t_j^\rmm)= \widetilde\bfV^{-1}\bfE^\top
 \Bigl(\bfS_\nu +\bfL^0(t_j^\rmm)+\bfL^1(t_j^\rmm)-\lambda\bfM\Bigr)\overline y(t_j^\rmm)\in\bbR^{ M_0\times  1},\\
 &\mbox{for~\eqref{uy-obli}}.\notag
 \end{align}

\subsection{Short comparison}\label{sS:Ric-vs-OP}
The computation of the feedback input as in~\eqref{discu-obli}  only requires
the computation of the inverse of the  matrix~$\widetilde\bfV\in\bbR^{M_0\times M_0}$,
whose size depends only the number of actuators, thus the numerical time needed to compute such inversion is independent of number~$N$ of spatial
mesh points. This is a computational advantage when compared to Riccati feedbacks as~\eqref{discu-ricc},
which require more time as~$N$ increases.
Further, we do not need to compute and save, offline prior to solve the dynamical system (parabolic equation),  the array feedback in~\eqref{F_Ric-disc}, $\bfK\in\bbR^{M_0\times \widehat N\times m_0}$,
containing the input feedback operator for each time~$\overline t_j$ in the discrete temporal mesh.  Indeed,~\eqref{discu-obli} can be simply computed, online, at each time~$t_j$, while solving the dynamical system. 
\black

On the other side the stabilization property of the explicit feedback in~\eqref{FeedKunRod}
is more sensitive to the number and placement
of the actuators in concrete examples. For suitable actuator placements, the Riccati feedback may succeed to stabilize the system, when the explicit feedback fails to.

\section{Stabilizing performance and set of actuators}\label{S:LocAct}
The placement of the actuators is a crucial point, in particular, for the explicit 
feedback~\eqref{FeedKunRod}. In order to make a comparison with the
Riccati feedbacks, we place the  indicator functions actuators~$\indf_{\omega_j}$ as illustrated in
Figure~\ref{fig:Mesh_coarse} for the cases of~$M_0=M_\sigma\in\{1,4,9\}$ actuators. Namely, we take actuators with rectangular supports~${\omega_j}$ as Cartesian products of the 1D actuators supports in~\eqref{loc-actOP1D} (cf.~\cite[sect.~4.8.1]{KunRod19-cocv}),
where we also show the coarsest
mesh used for the computations.
\begin{figure}[ht]
\centering
\subfigure
{\includegraphics[width=0.32\textwidth]{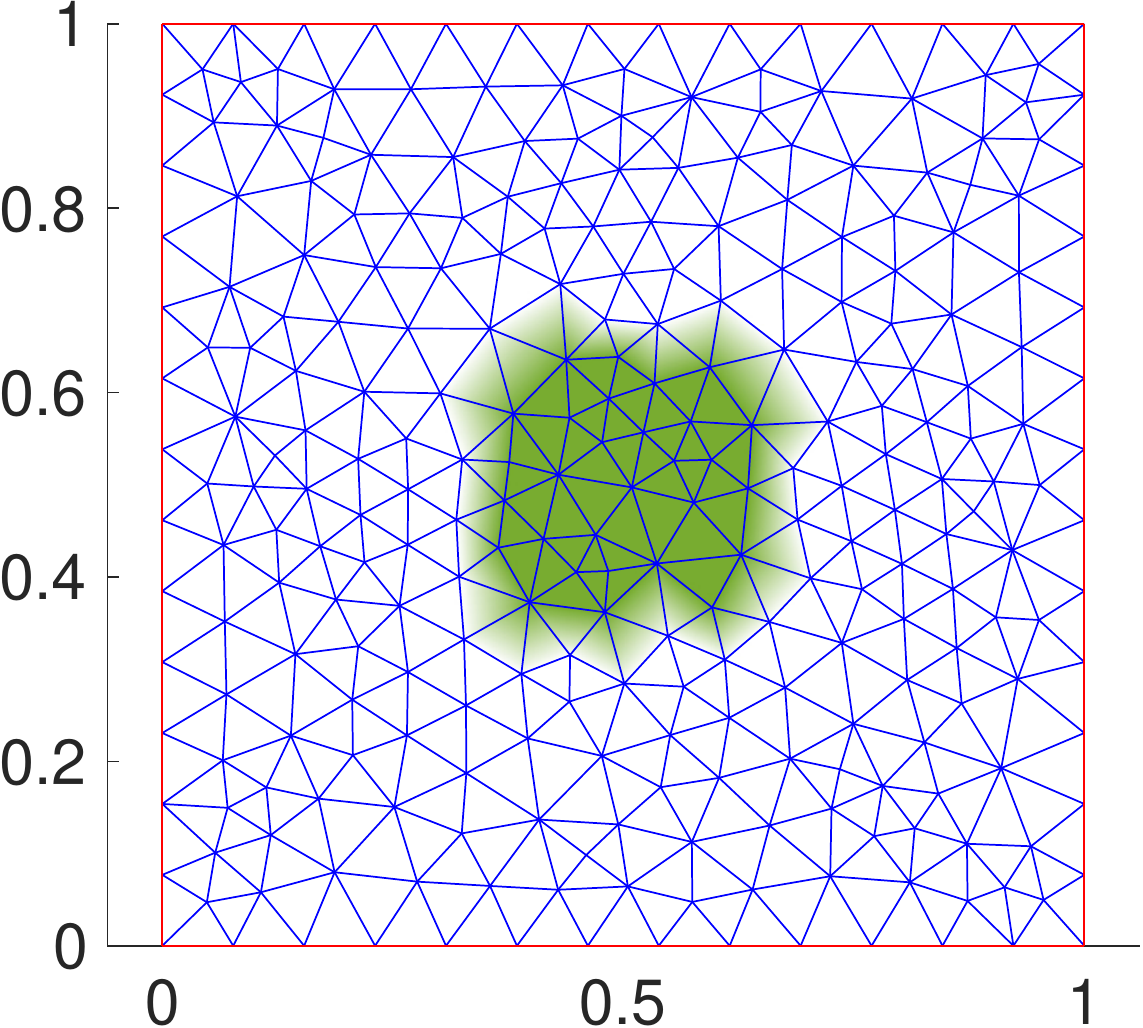}}
\,
\subfigure
{\includegraphics[width=0.32\textwidth]{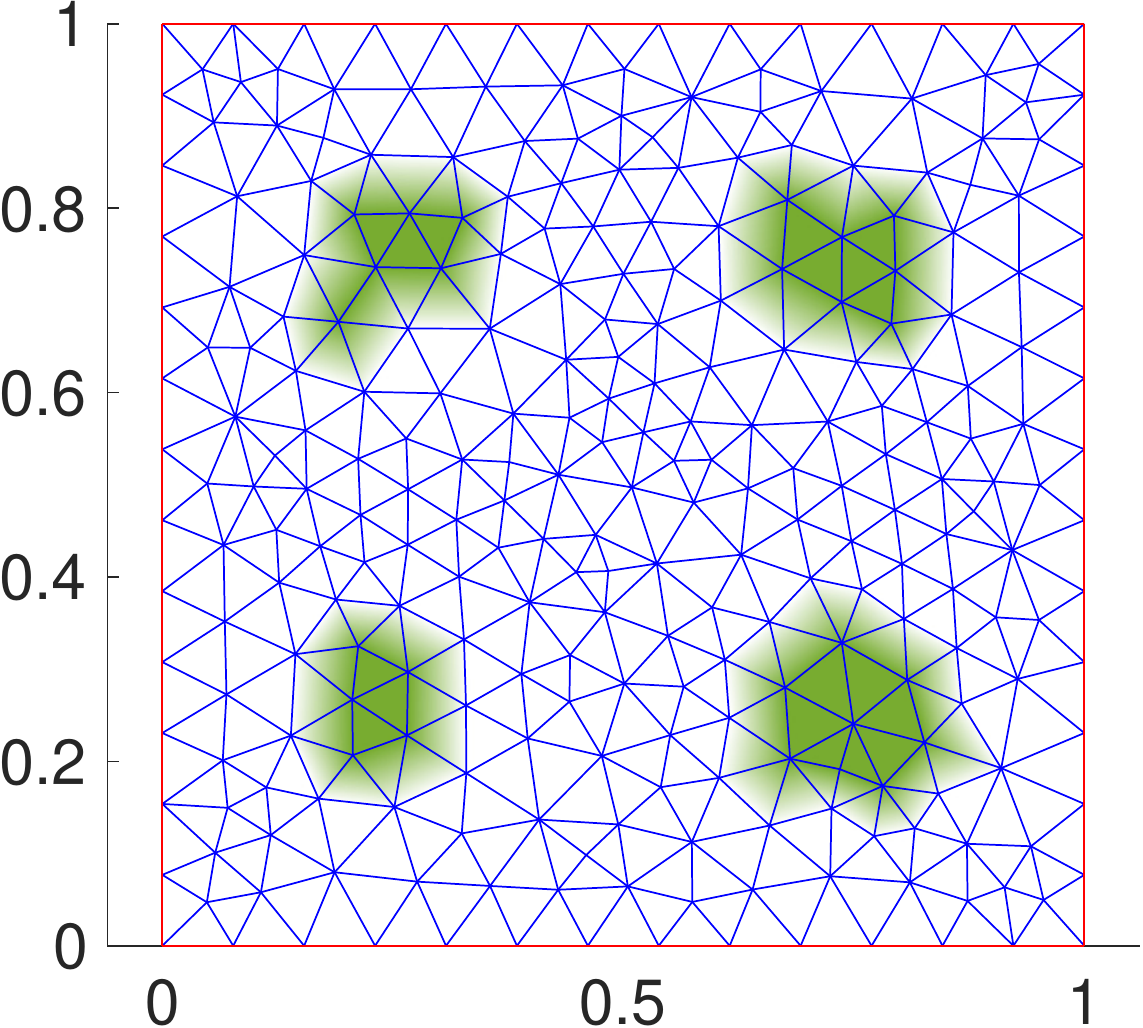}}
\,
\subfigure
{\includegraphics[width=0.32\textwidth]{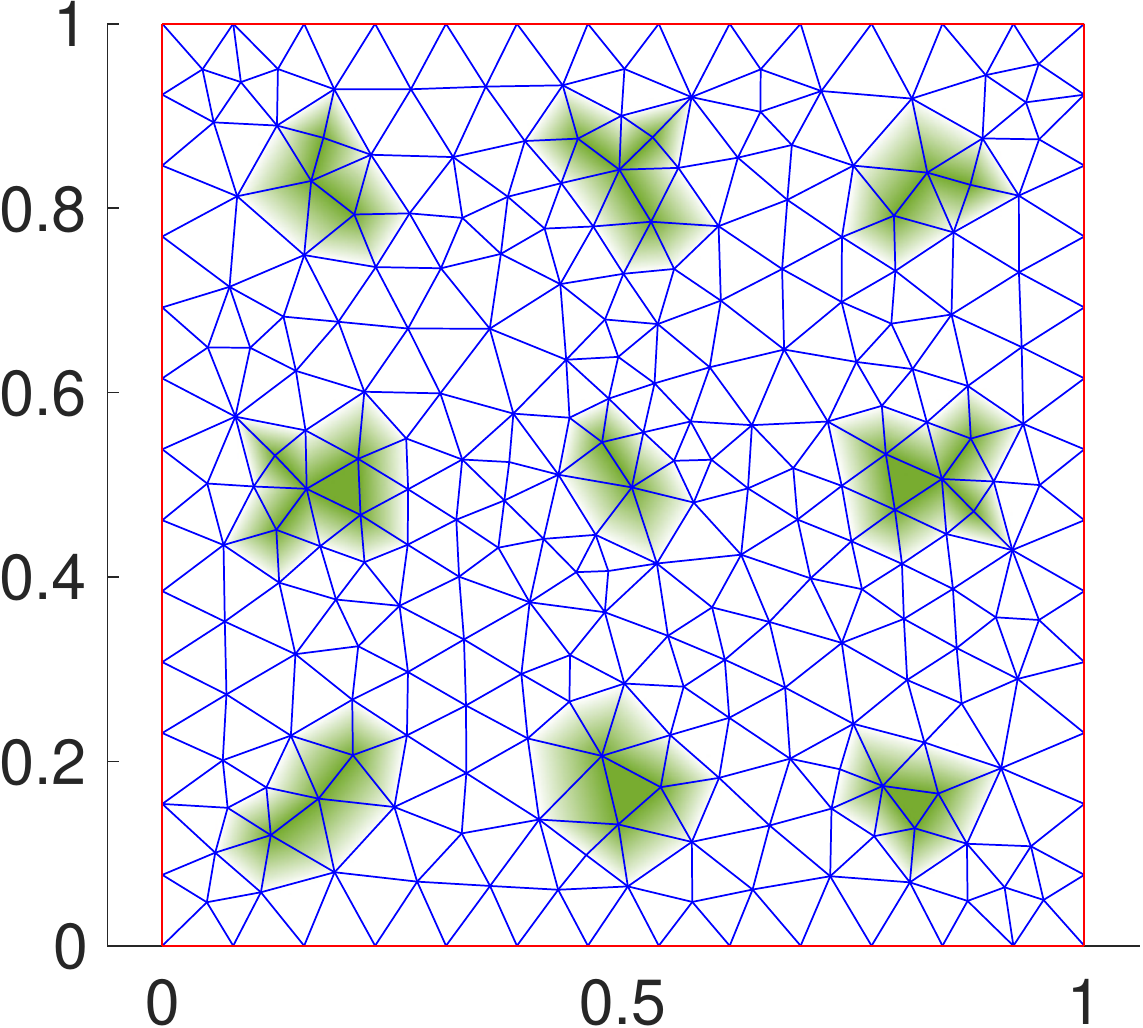}}
\caption{Initial spatial meshes ($\rho=0$) and supports~$\omega_i$ of actuators~$\indf_{\omega_i}$.}
\label{fig:Mesh_coarse}
\end{figure}
Since the mesh is unstructured and  coarse the supports do not look like the rectangular supports we have the continuous level. We shall perform simulations is refinements of such mesh where the supports look more like those rectangular subdomains as we increase the number~$\rho$ of refinements, as we can see in Fig.~\ref{fig:Mesh_refs}. The coarsest triangulation~$\clT^0$ corresponds to~$\rho=0$, the refined triangulation~$\clT^\rho$ for~$\rho\in\{1,2,3\}$ is obtained by dividing each triangle~$\clT^{\rho-1}_k$ of the triangulation~$\clT^{\rho-1}$ into~$4$ congruent triangles by connecting the middle points of the edges of~$\clT^{\rho-1}_k$ (regular refinement).

\begin{figure}[ht]
\centering
\subfigure
{\includegraphics[width=0.32\textwidth]{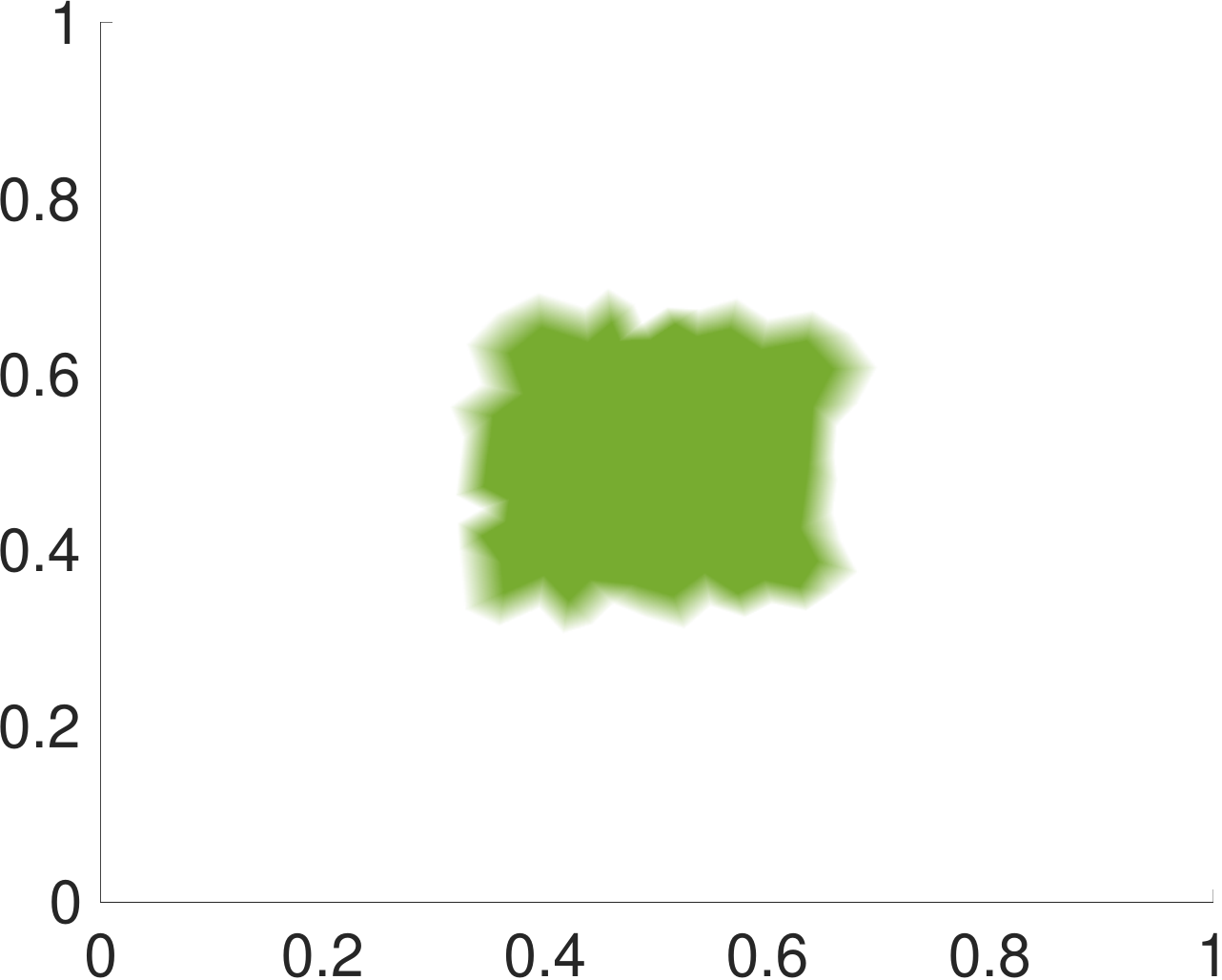}}
\,
\subfigure
{\includegraphics[width=0.32\textwidth]{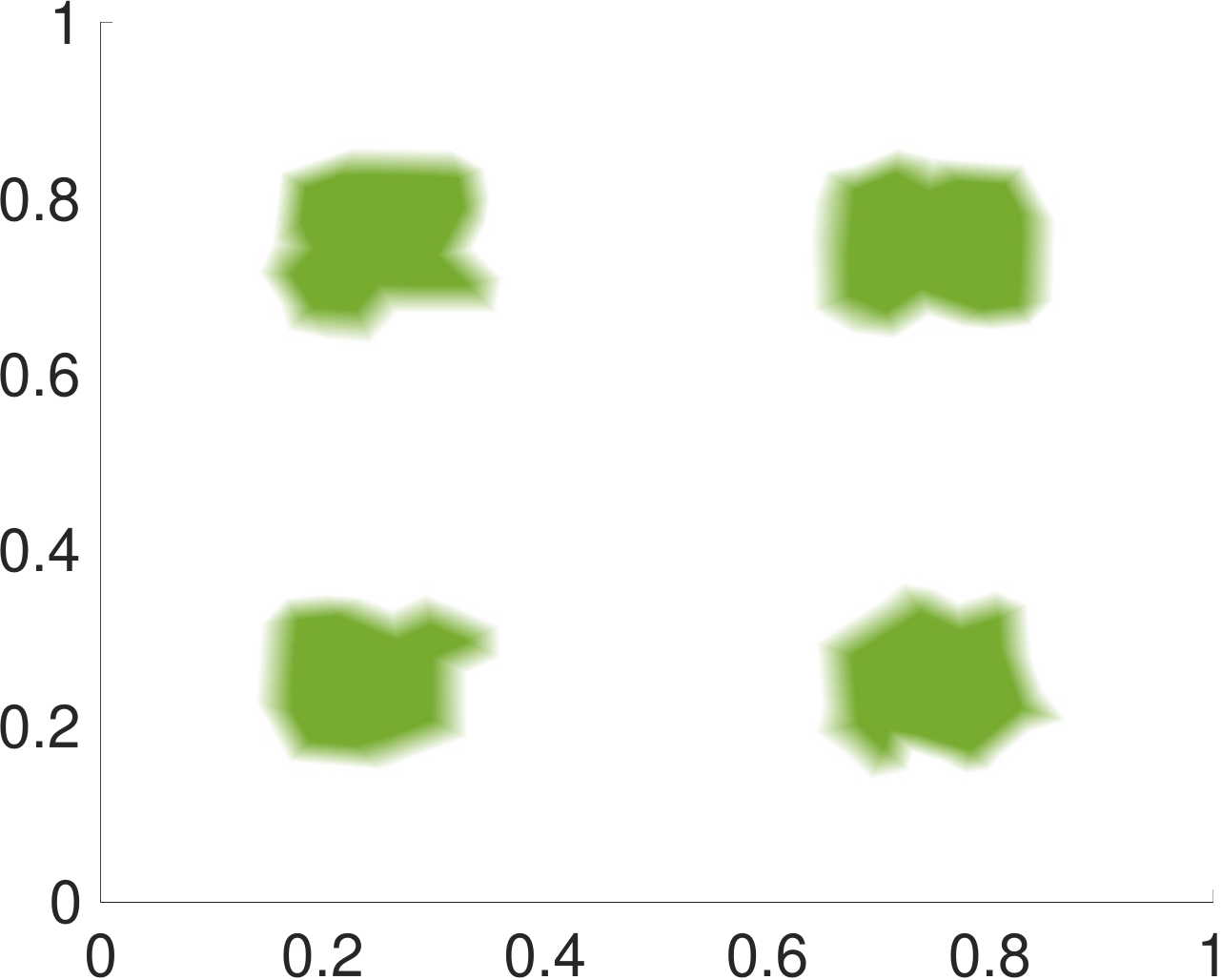}}
\,
\subfigure
{\includegraphics[width=0.32\textwidth]{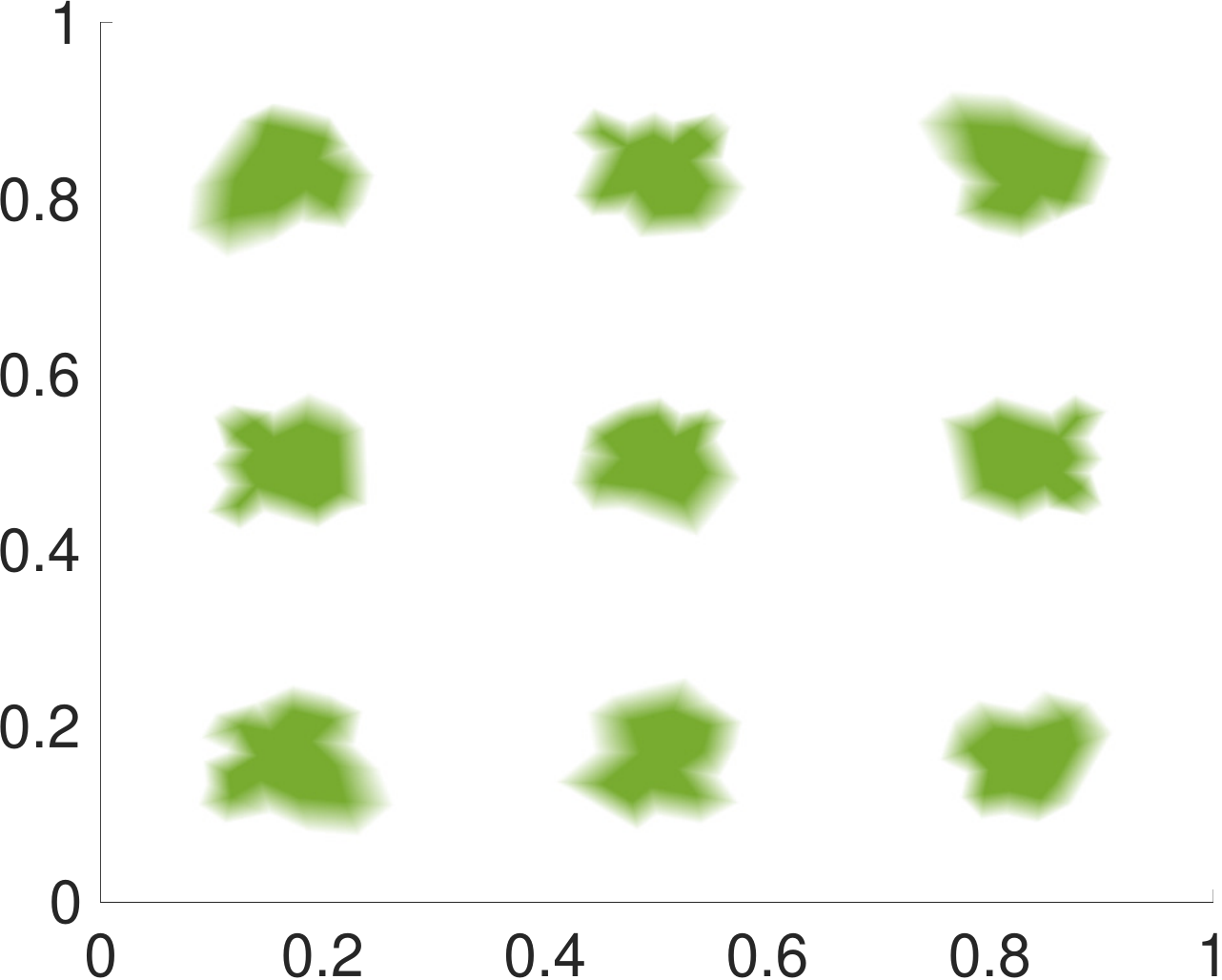}}
\\
\subfigure
{\includegraphics[width=0.32\textwidth]{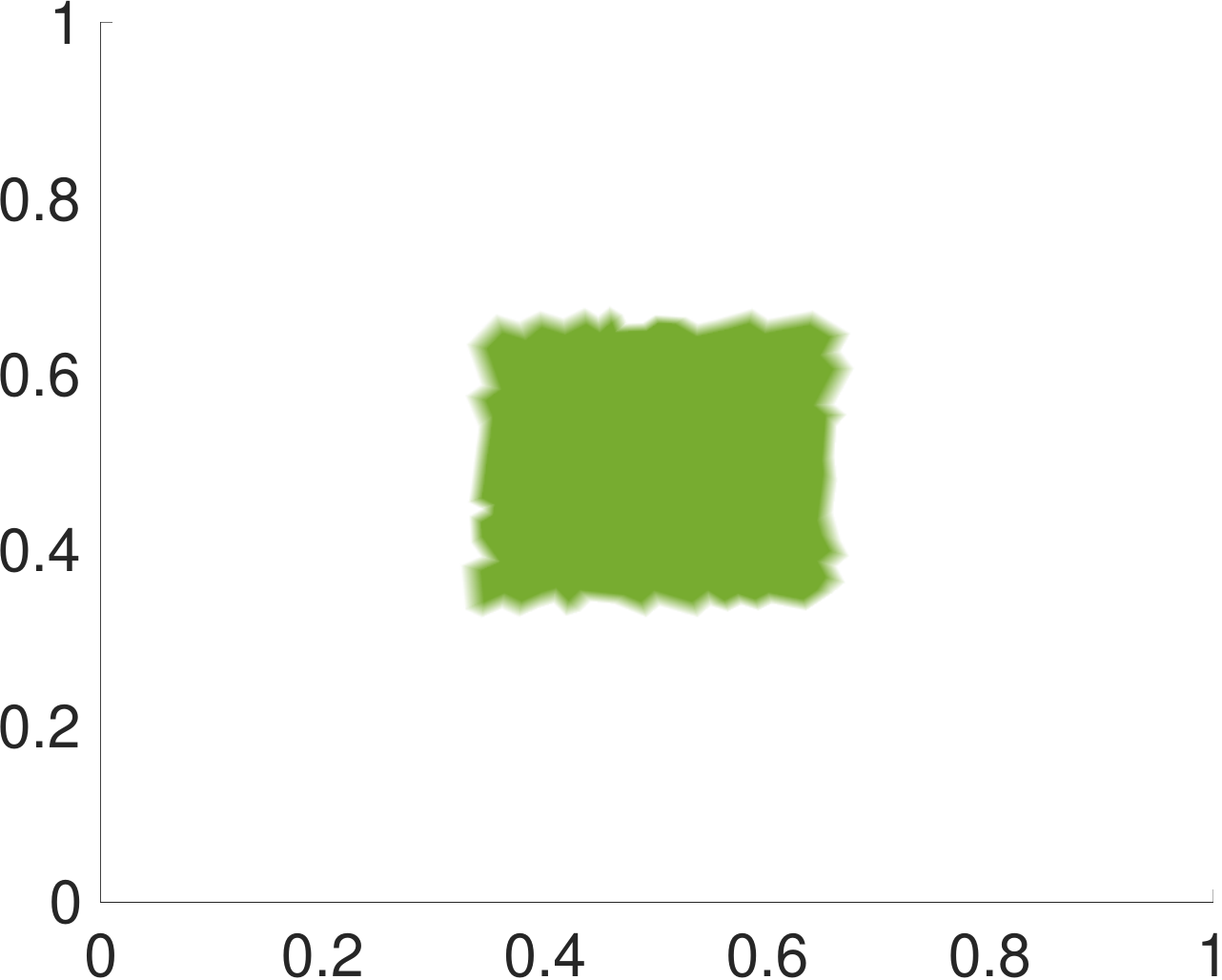}}
\,
\subfigure
{\includegraphics[width=0.32\textwidth]{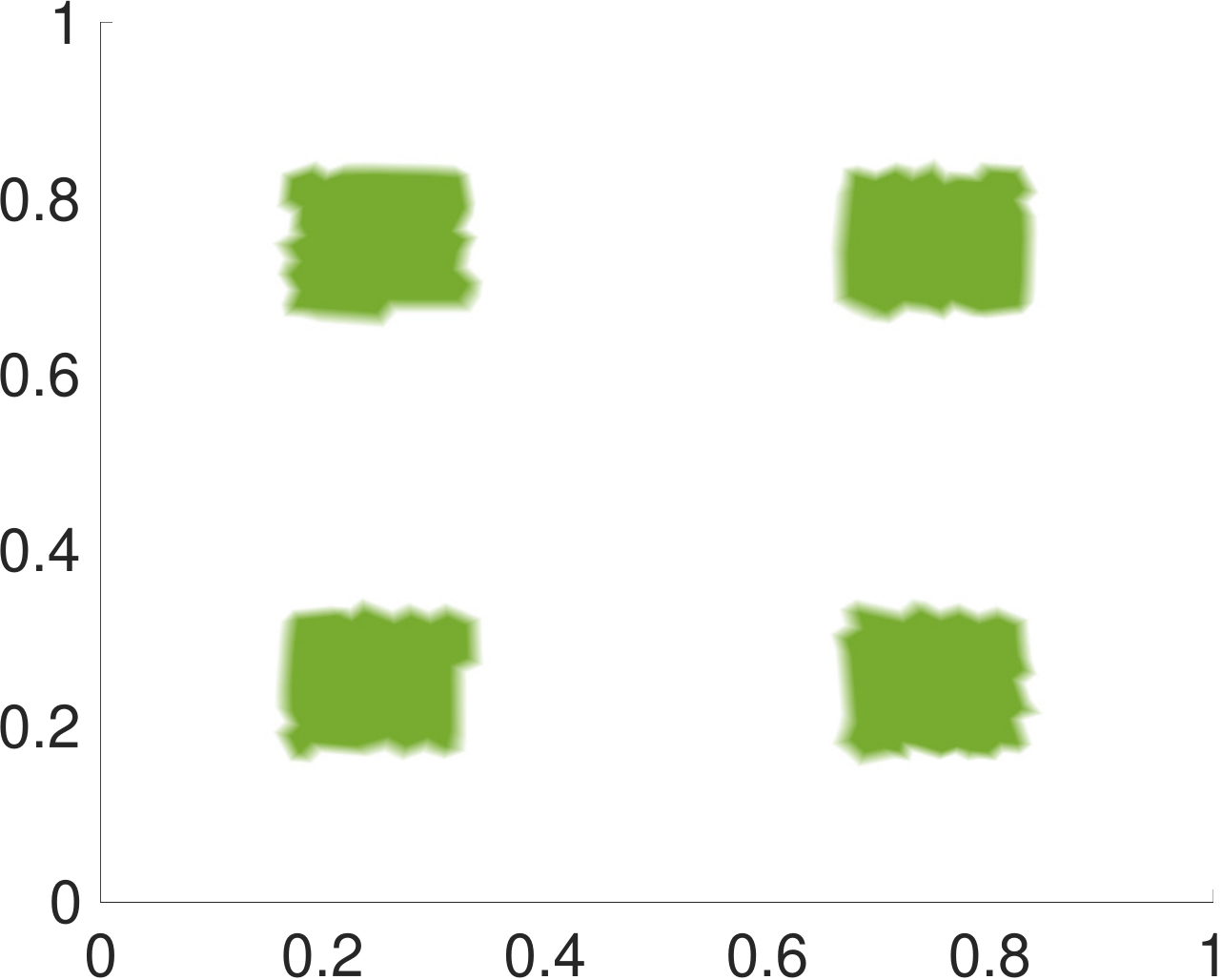}}
\,
\subfigure
{\includegraphics[width=0.32\textwidth]{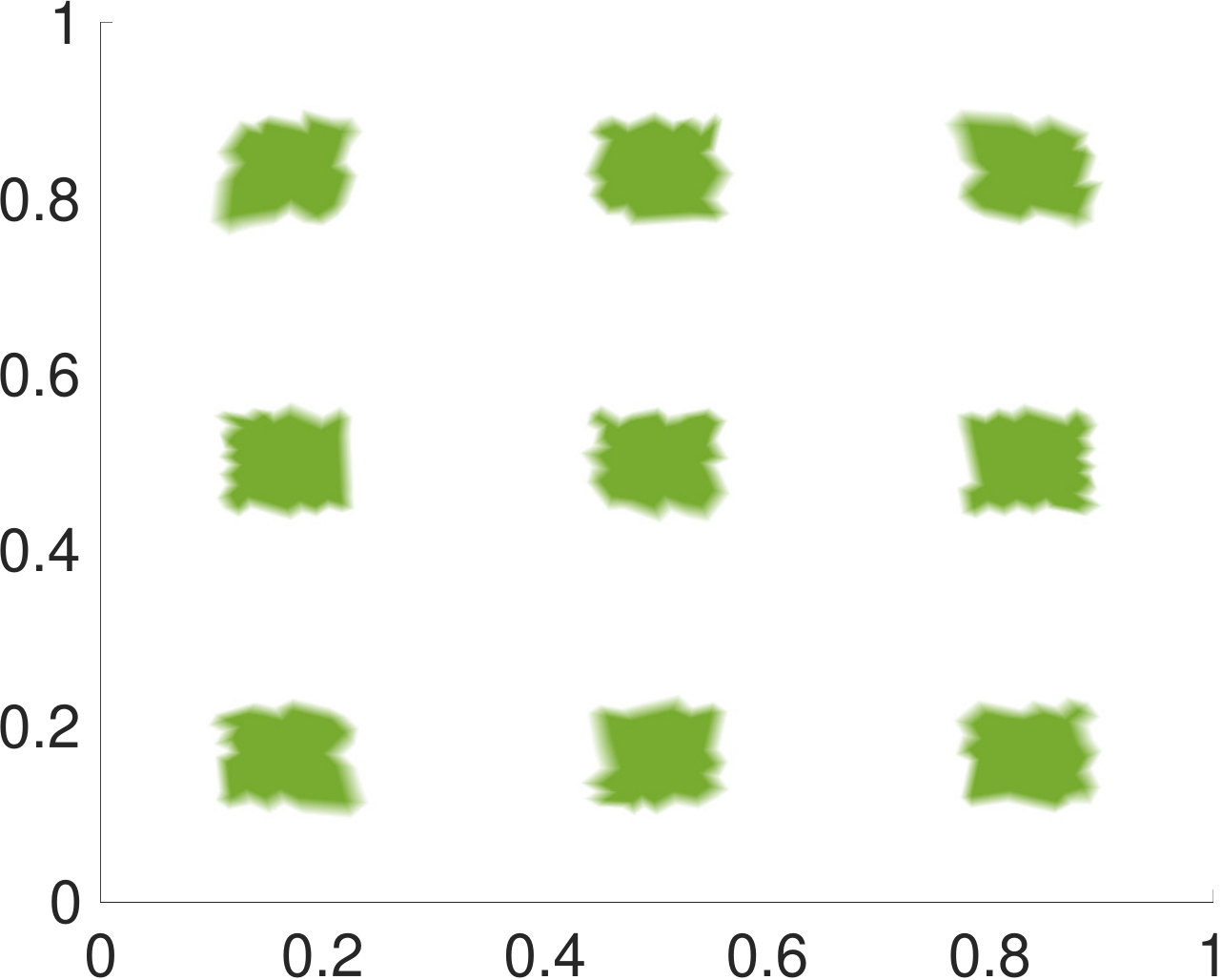}}
\\
\subfigure
{\includegraphics[width=0.32\textwidth]{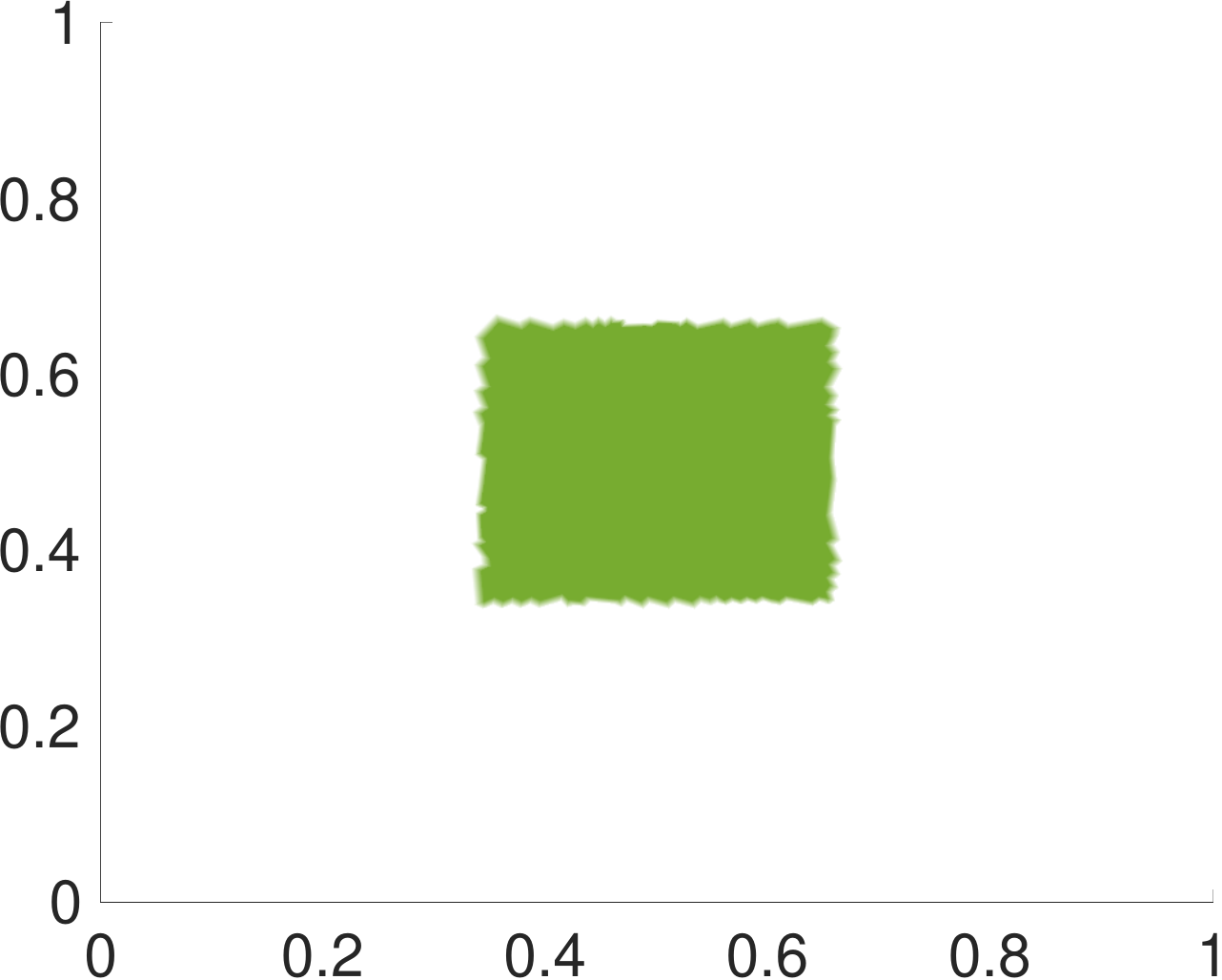}}
\,
\subfigure
{\includegraphics[width=0.32\textwidth]{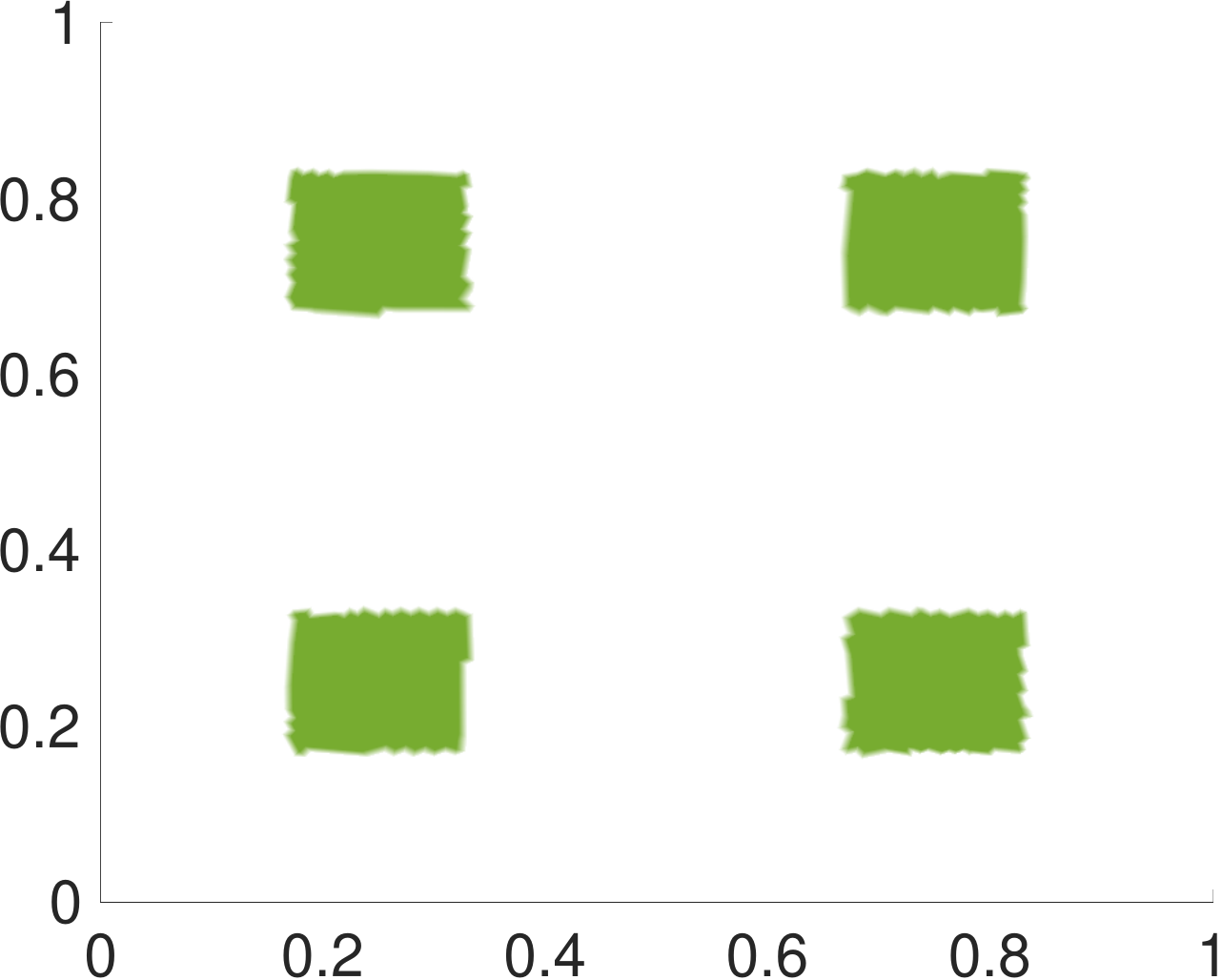}}
\,
\subfigure
{\includegraphics[width=0.32\textwidth]{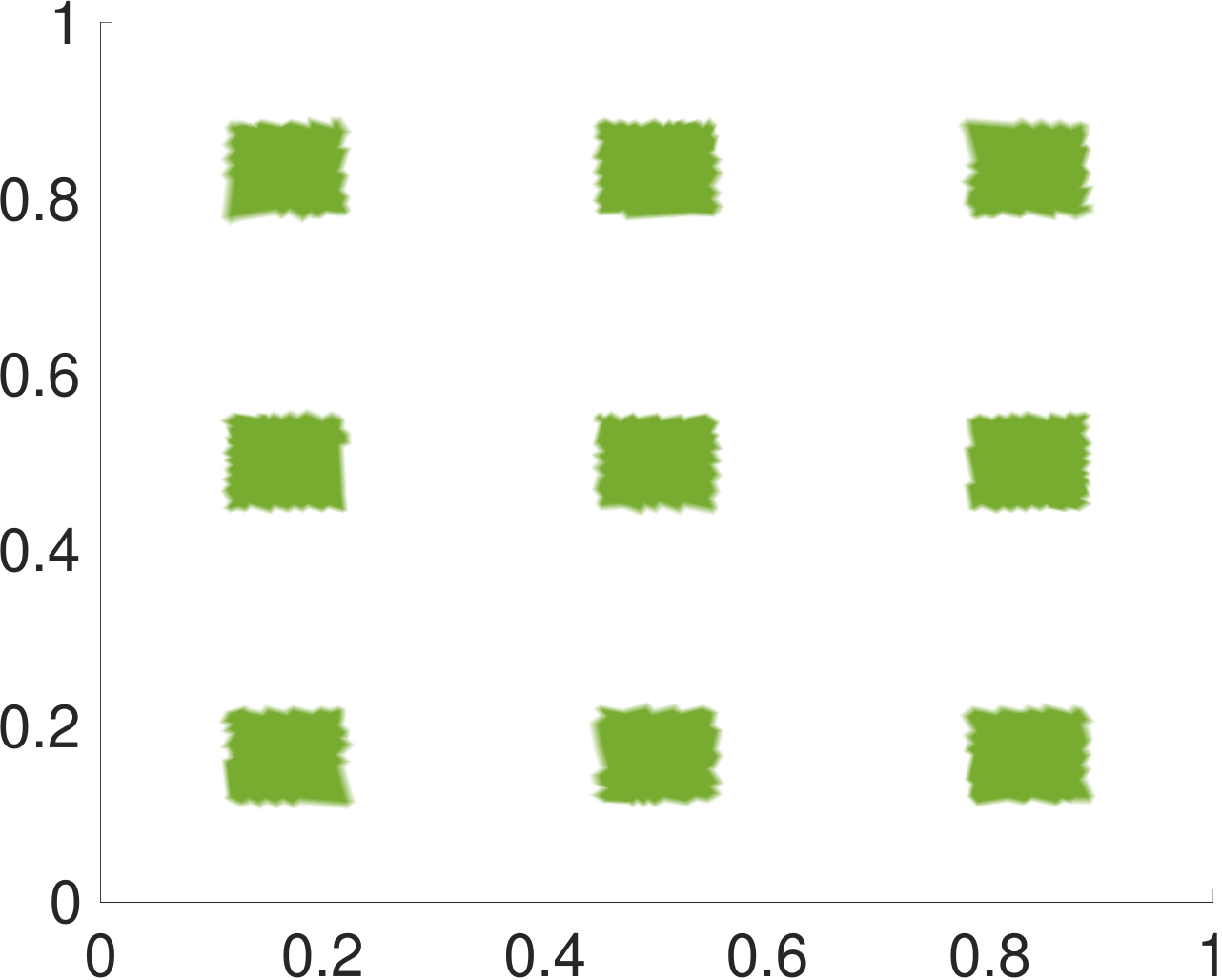}}
\caption{The $\rho$-th row  shows the supports~$\omega_i$ of actuators~$\indf_{\omega_i}$ after~$\rho$ regular refinement, $\rho\in\{1,2,3\}$. }
\label{fig:Mesh_refs}
\end{figure}

We consider~\eqref{sys-y-parab-clK} under Neumann boundary conditions, and
\begin{subequations}\label{dataSystem}
\begin{align}
\nu&=0.1,&\qquad a&=-\tfrac52+x_1 -\norm{\sin(6t+x_1)}{\bbR},\\
 y_0&\coloneqq1-2x_1x_2,&\qquad
b&=\bigl(x_1+x_2,\;\norm{\cos(6t)x_1x_2}{\bbR}\bigr).
\end{align} 
\end{subequations}

The instability of the free dynamics is shown in Figure~\ref{fig:free-dyn}.
Here, we used the coarsest spatial mesh with time step~$k=0.01$ and the mesh obtained after~$3$
regular refinements with time step~$k=0.001$.
\begin{figure}[ht]
\centering
\subfigure
{\includegraphics[width=0.45\textwidth]{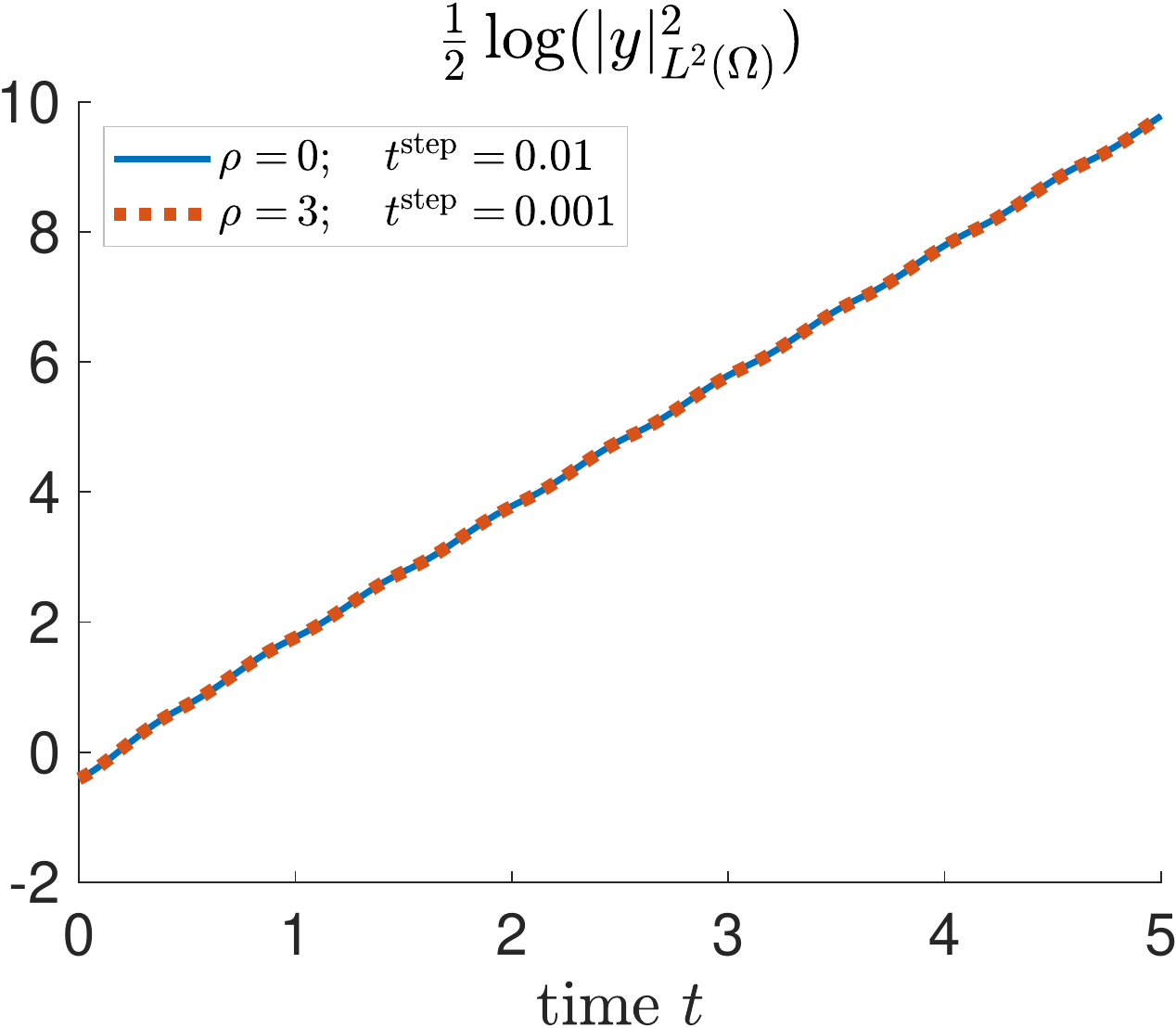}}
\caption{Free dynamics. Evolution of the norm of the solution. }
\label{fig:free-dyn}
\end{figure}

We shall compare the time-truncated cost
\begin{align}\label{costT}
\bfJ=\bfJ^{\rm feed}(y_0,T)\coloneqq\tfrac12\norm{y^{\rm feed}}{L^2((0,T),H)}^2+\beta\tfrac12\norm{u^{\rm feed}}{L^2((0,T),\bbR^{M_0})}^2,
\end{align}
associated to Riccati and oblique projection feedbacks, ${\rm feed}\in\{{\rm ricc},{\rm obli}\}$.

Note that~$(a,b)$ in~\eqref{dataSystem} is time-periodic with period~$\tfrac\pi6$. 
We compute offline (prior to solve the parabolic equations),  the input periodic Riccati 
feedback~$\widehat\bfK=\bfK$ as in~\eqref{F_Ric-disc},
for corresponding actuators and coarse spatial mesh in Figure~\ref{fig:Mesh_coarse},
for
\[
t\in[\tau,\tau+\varpi],\quad\mbox{with}\quad \tau=0.1,\quad\varpi=\tfrac\pi6,
\]
(here, we could have chosen any~$\tau\ge0$) and with the parameters
\begin{align}\label{num-paramRic}
\beta=1,\quad\overline\mu=\mu_{\rm ric}=1,\quad\mbox{and}\quad k_{\rm ric}=0.005.
\end{align}
Such feedback is then used as in~\eqref{discu-ricc} for refined meshes.

The
explicit feedback~\eqref{discu-obli} is computed online (while solving the equations), with the parameter
\[
 \lambda=1.
\]

Since in the time interval~$(0,+\infty)$, the cost~\eqref{LQcost} is minimized by the Riccati feedback,
we may expect to have that, for large~$T$, 
\begin{equation}\label{Ric.le.obli}
\bfJ^{\rm ricc}(y_0,T)\le\bfJ^{\rm obli}(y_0,T). 
\end{equation}

To construct the feedback in~\eqref{discu-obli}, as auxiliary eigenfunctions we have chosen
the Cartesian products of the first $M_0^\frac12$ one-dimensional (Neumann) eigenfunctions (as proposed in~\cite[sect.~4.8.1]{KunRod19-cocv})
\begin{subequations}\label{auxil.eig-x}
\begin{align}
\clE_{M_0}&=\linspan\left\{e_{\bfj}\mid\bfj=(\bfj_1,\bfj_2)\in \{1,2,\dots,M_0^\frac12\}^2 \right\},\qquad M_0\in\{1,4,9\},\\
 e_{\bfj}&\coloneqq\cos\left((\bfj_1-1)\pi x_1\right)
 \cos\left((\bfj_2-1)\pi x_2\right).
\end{align} 
\end{subequations}

\subsection{Using one actuator}\label{sS:num.M=1}
In Figure~\ref{fig:M1expl} we see that the explicit
oblique projection feedback~\eqref{discu-obli} for the case of~$1$ actuator is not able to stabilize the system,
while in Figure~\ref{fig:M1ricc}
we see that the Riccati feedback~\eqref{discu-ricc} is (for the given initial  state~$y_0$).
The later was computed
for the coarsest mesh, and we can also see that it
is still able to stabilize the system exponentially for refined meshes (again, for the given initial state).
Note that we cannot conclude, from the present simulation result corresponding to a single initial condition, that the Riccati feedback will stabilize the solutions corresponding to an arbitrary initial state.
\begin{figure}[ht]
\centering
\subfigure
{\includegraphics[width=0.45\textwidth]{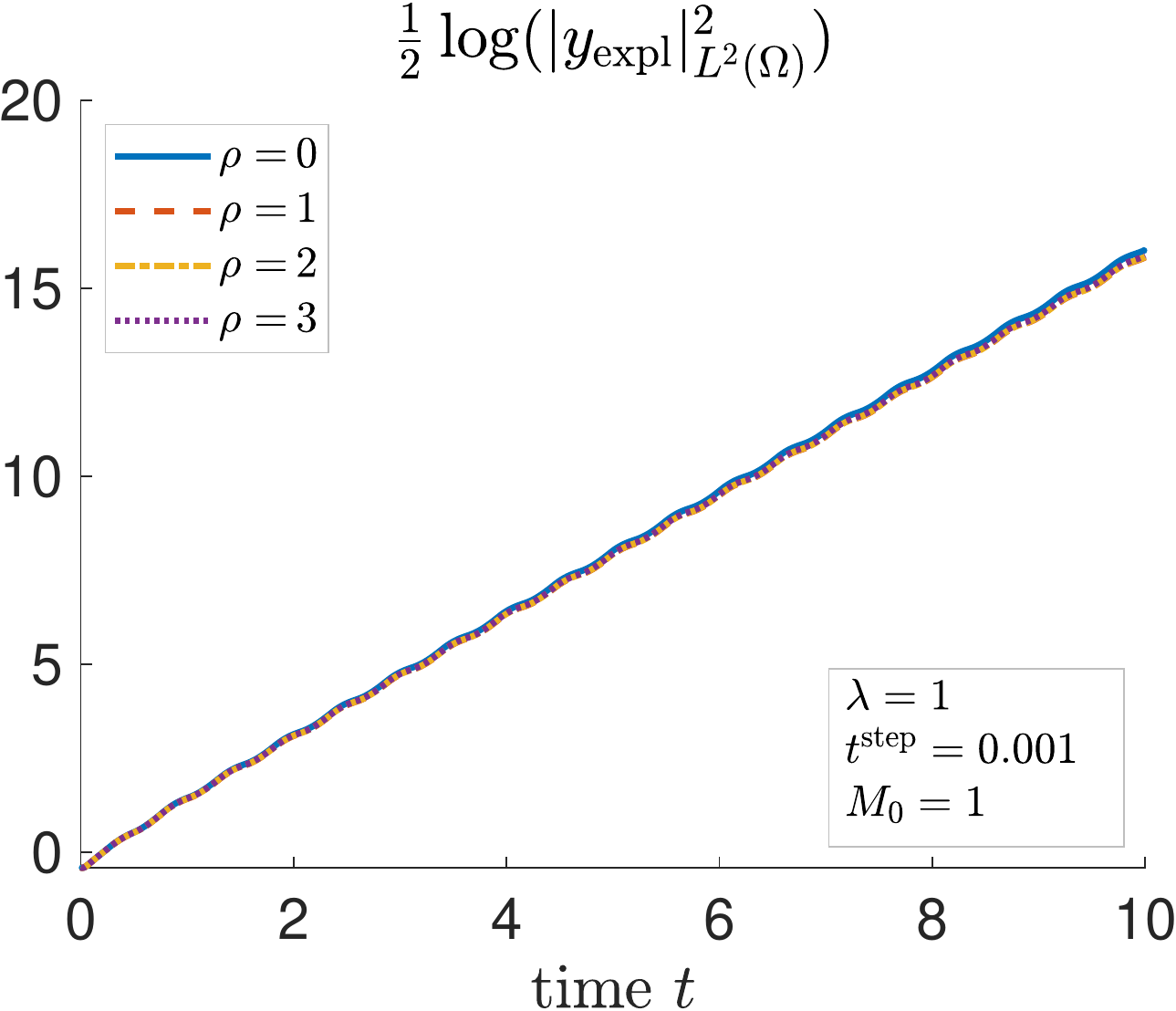}}
\quad
\subfigure
{\includegraphics[width=0.45\textwidth]{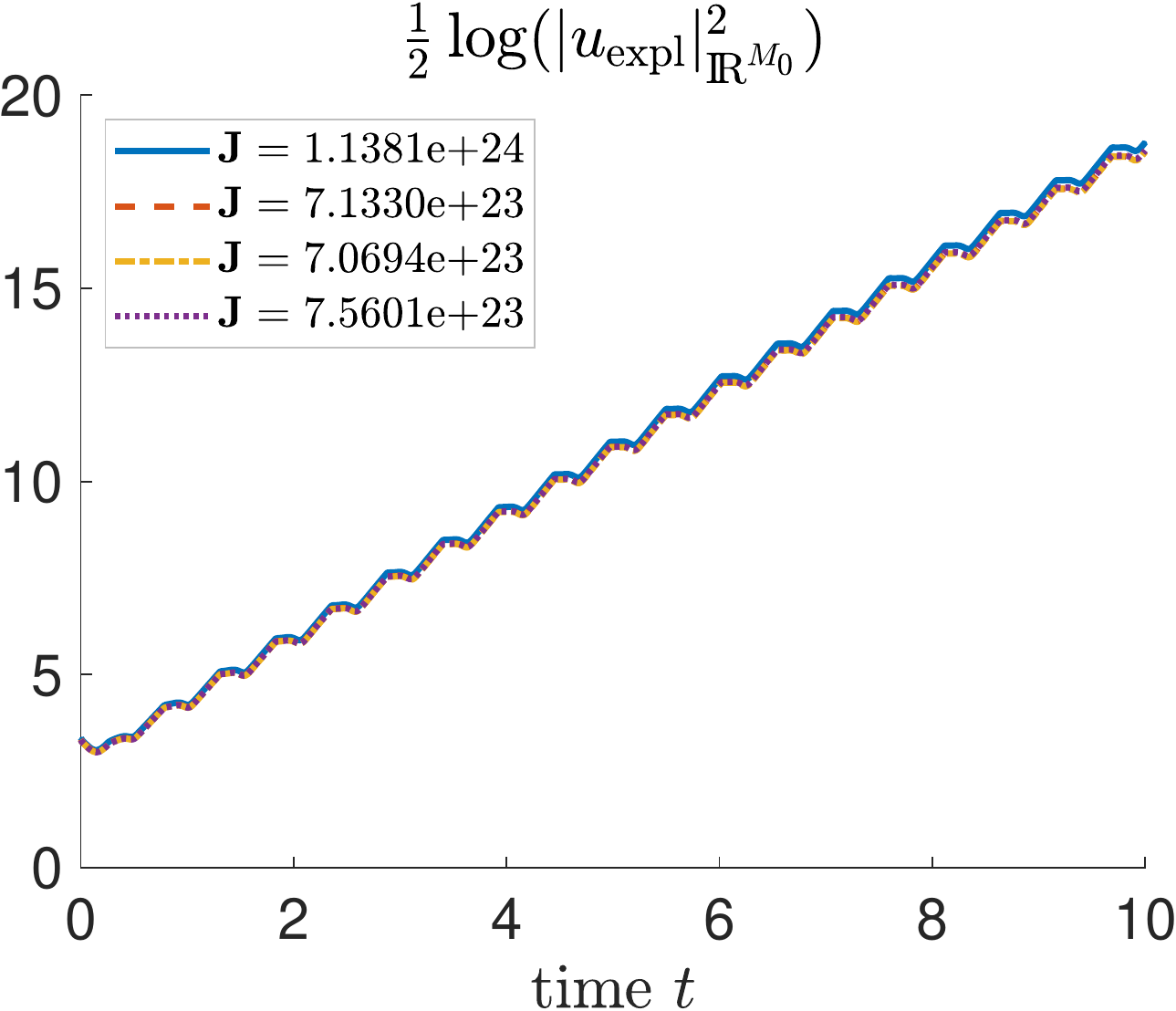}}
\caption{$M_0=1$. Oblique projection input feedback~\eqref{discu-obli}.\newline}
\label{fig:M1expl}
\subfigure
{\includegraphics[width=0.45\textwidth]{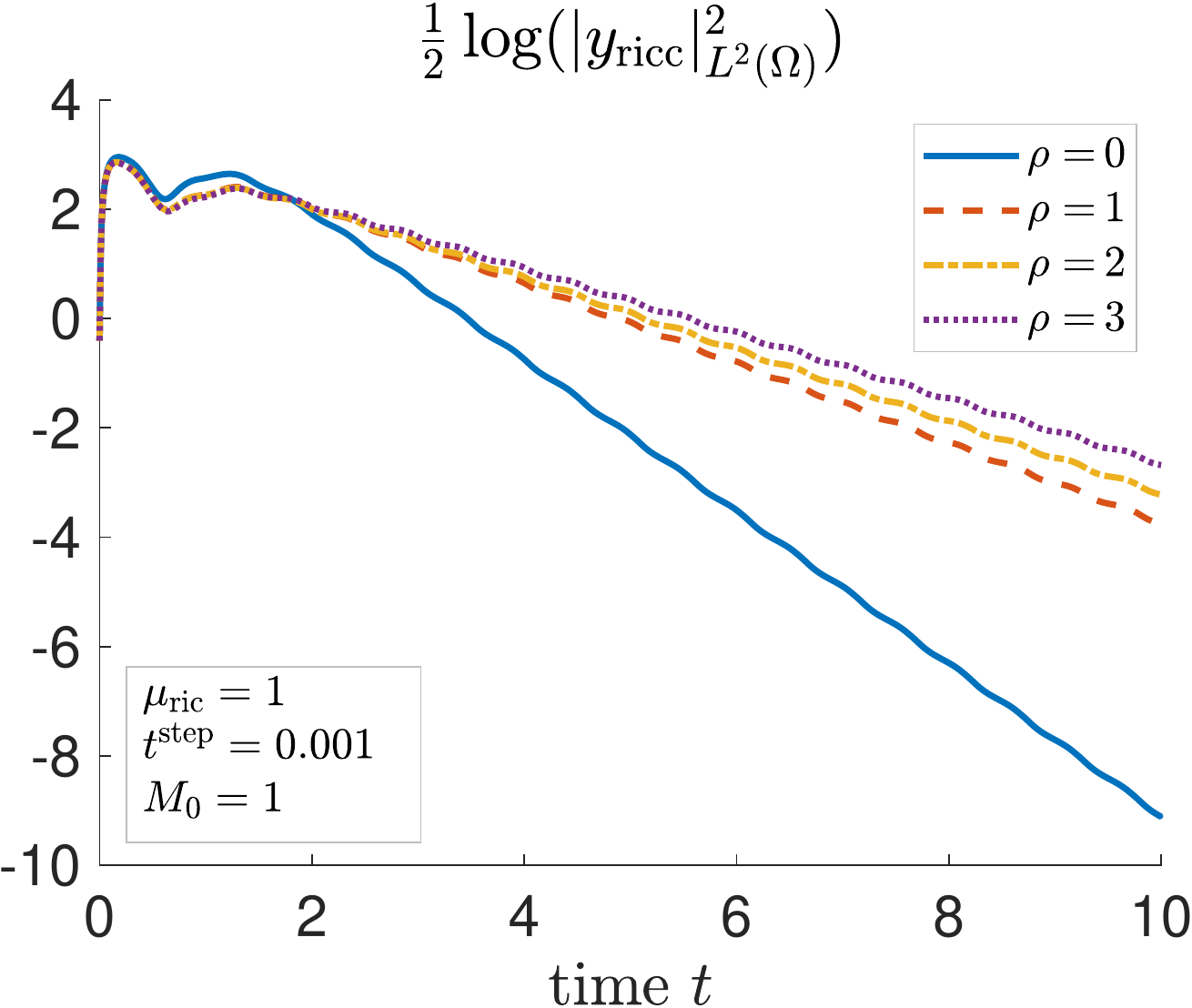}}
\quad
\subfigure
{\includegraphics[width=0.45\textwidth]{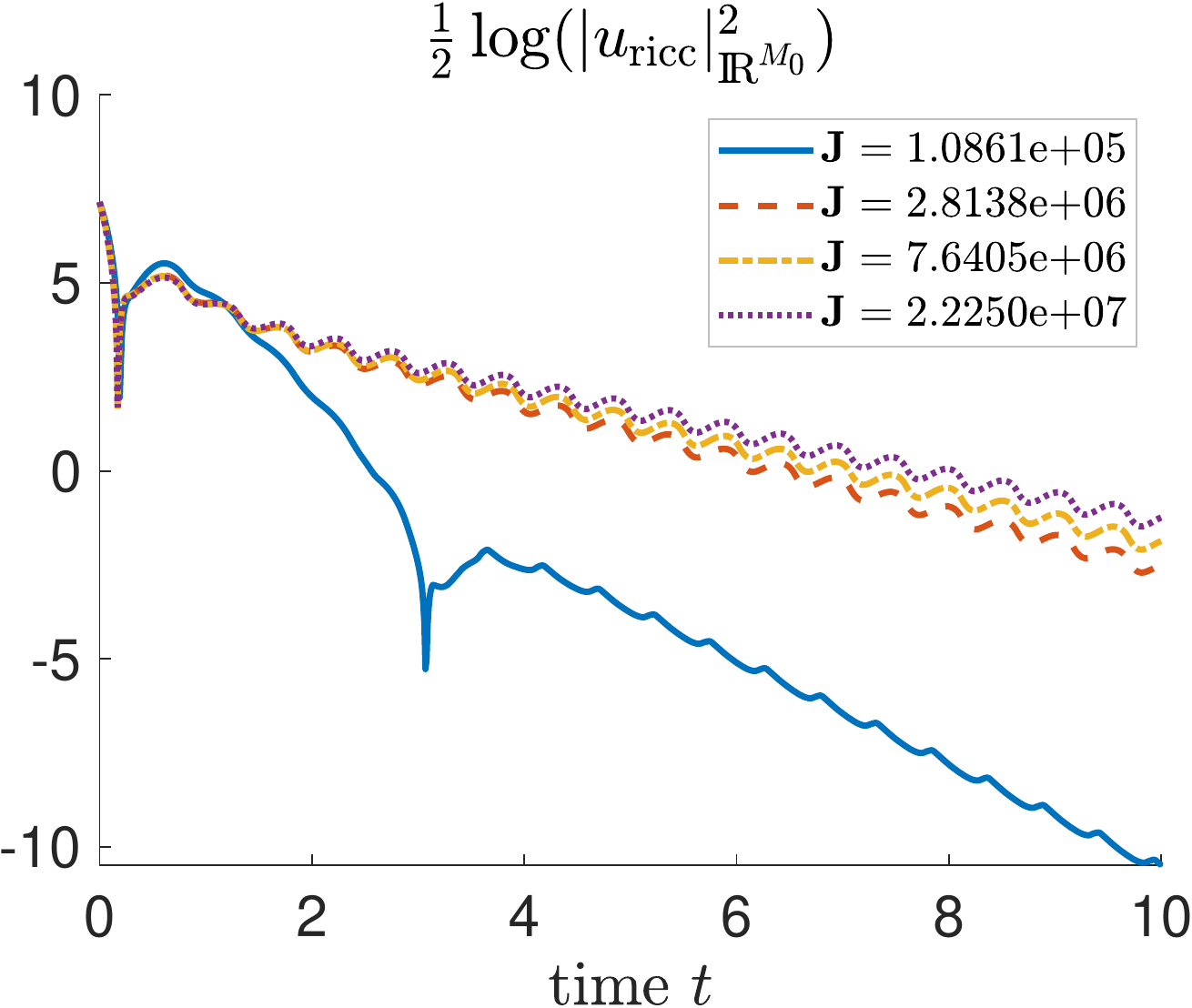}}
\caption{$M_0=1$. Riccati input feedback~\eqref{discu-ricc}.}
\label{fig:M1ricc}
\end{figure}
We also  observe that the asked stability rate~$\overline\mu=\mu_{\rm ric}=1$ is guaranteed for the coarsest mesh but not for the refined meshes. This shows that the Riccati feedback computed for the coarsest mesh does not lead to a good approximation of the Riccati operator solution and to the  Riccati matrix solution for refined triangulations.
This could also be a sign that one single actuator is not able to stabilize the system with exponential rate~$\mu_{\rm ric}=1$ (for all initial states). Indeed for a rectangular with support~$\omega=(\tfrac12-r,\tfrac12+r)$ centered at the center of our spatial square~$\Omega$, this can be seen for the autonomous system corresponding to the reaction-convection pair~$(a,b)=(c,(0,0))$ with small enough constant~$c<0$, because for the solution of the system
\[
\tfrac{\p}{\p t} z+(-\nu\Delta+\Id) z+cz=u\indf_\omega,\qquad z(0)=z_0\coloneqq\cos(\pi x_1)\cos(\pi x_2),
\]
under Neumann boundary conditions, since~$(z_0,\indf_\omega)_{L^2(\Omega)}=0$ and
$z_0$ is an eigenfunction of~$-\nu\Delta+\Id+c\Id$, we find that,  for arbitrary control input~$u\in L^2(\bbR_+,\bbR)$ we will have
\[
\tfrac{\p}{\p t} (P_{\bbR z_0}z)=(-2\pi^2\nu-1-c)P_{\bbR z_0}z,\quad P_{\bbR z_0}z(0)=z_0,
\]
where~$P_{\bbR z_0}$ is the orthogonal projection onto the linear span~${\bbR z_0}$ of~$\{z_0\}$. Hence,
\[
P_{\bbR z_0}z(t)=\rme^{(-2\pi^2\nu-1-c)t}z_0,
\]
which diverges to~$+\infty$ if~$c<-2\pi^2\nu-1$.

\subsection{Using four actuators}
We increase the number of actuators (preserving the total volume covered by then).
Figures~\ref{fig:M4expl} and~\ref{fig:M4ricc} show that, with~$4$ actuators,
both oblique projection and Riccati based feedbacks  are able to stabilize the system.
\begin{figure}[ht]
\centering
\subfigure
{\includegraphics[width=0.45\textwidth]{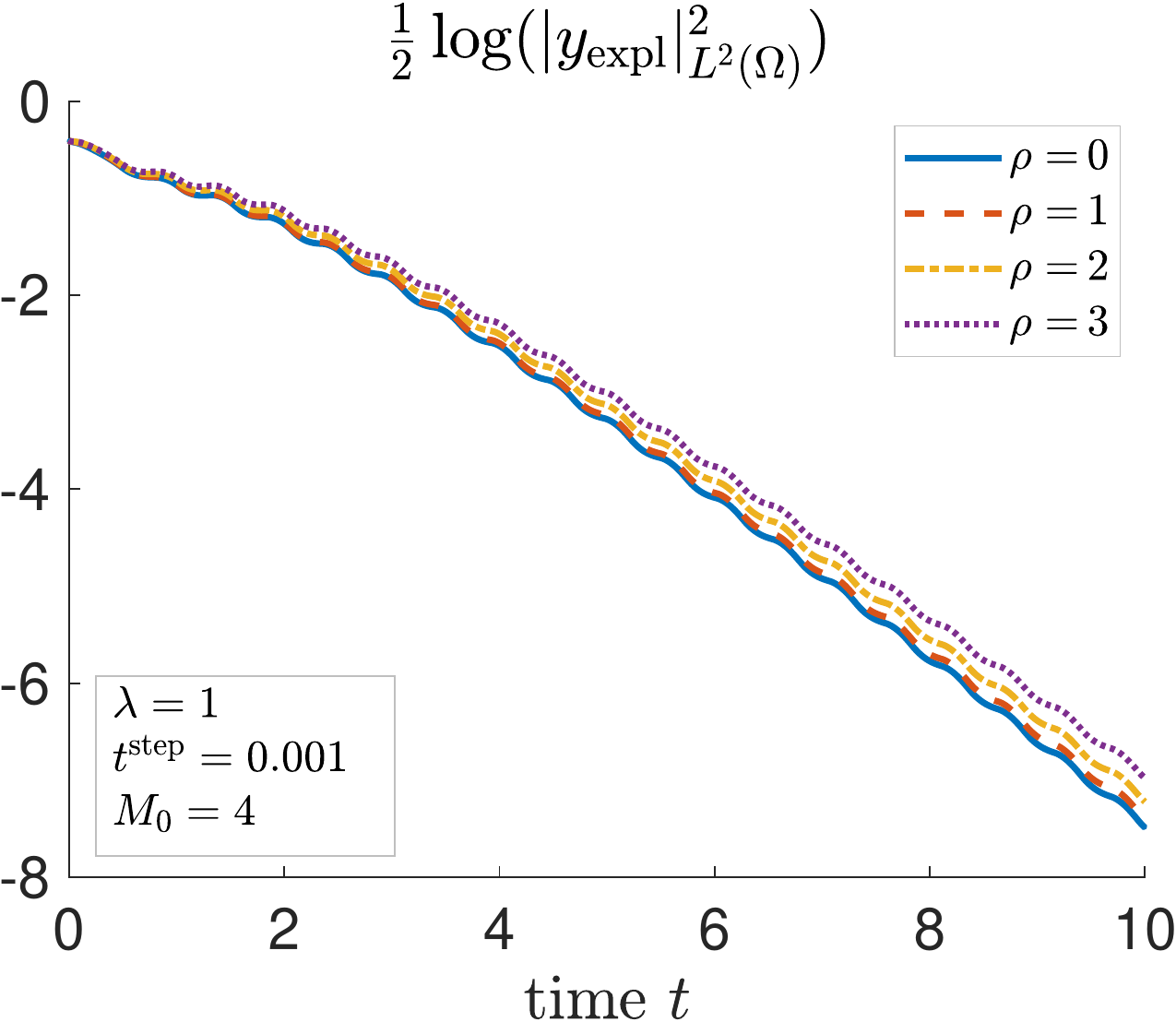}}
\quad
\subfigure
{\includegraphics[width=0.45\textwidth]{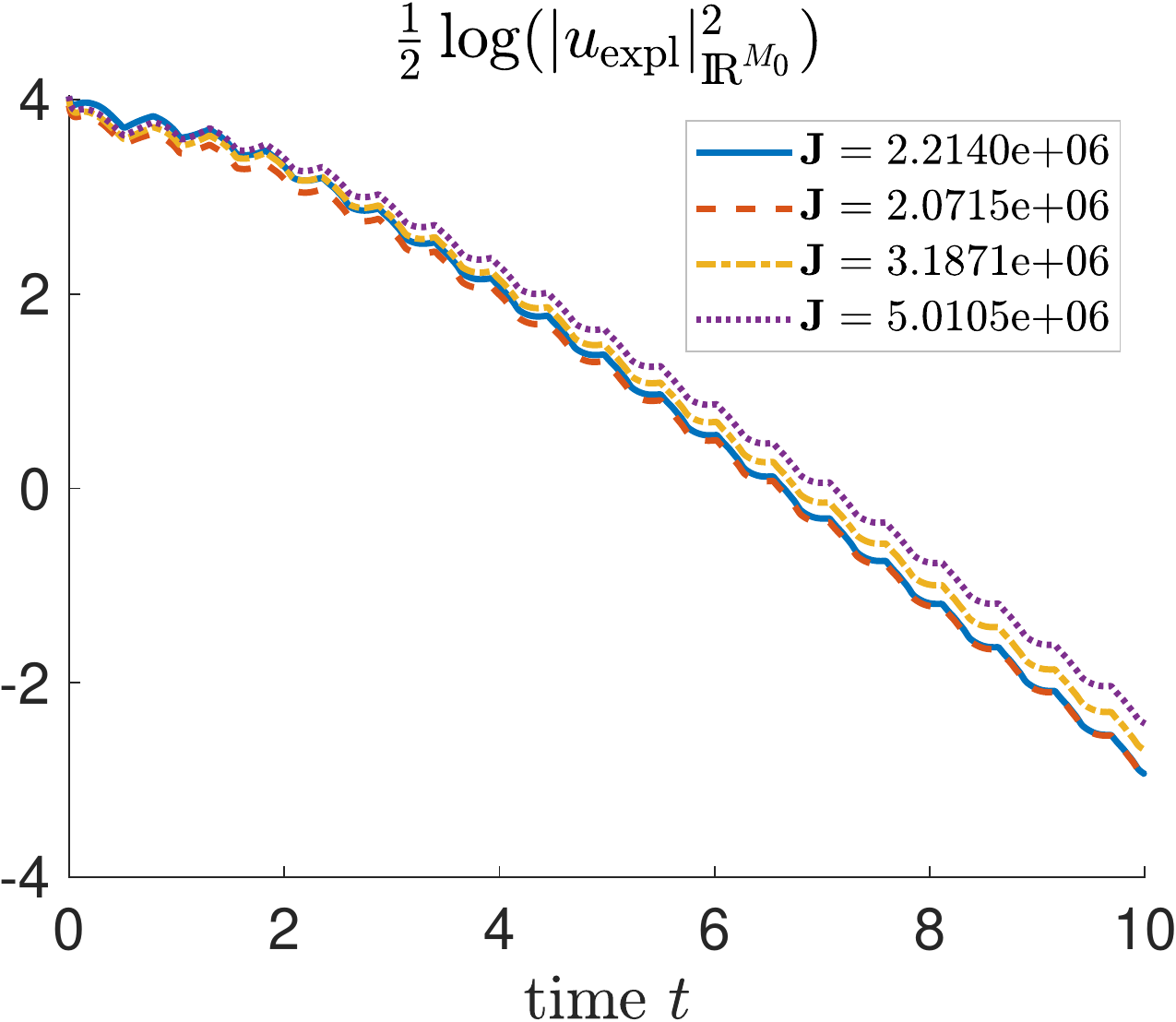}}
\caption{$M_0=4$. Oblique projection input feedback~\eqref{discu-obli}.\newline}
\label{fig:M4expl}
\subfigure
{\includegraphics[width=0.45\textwidth]{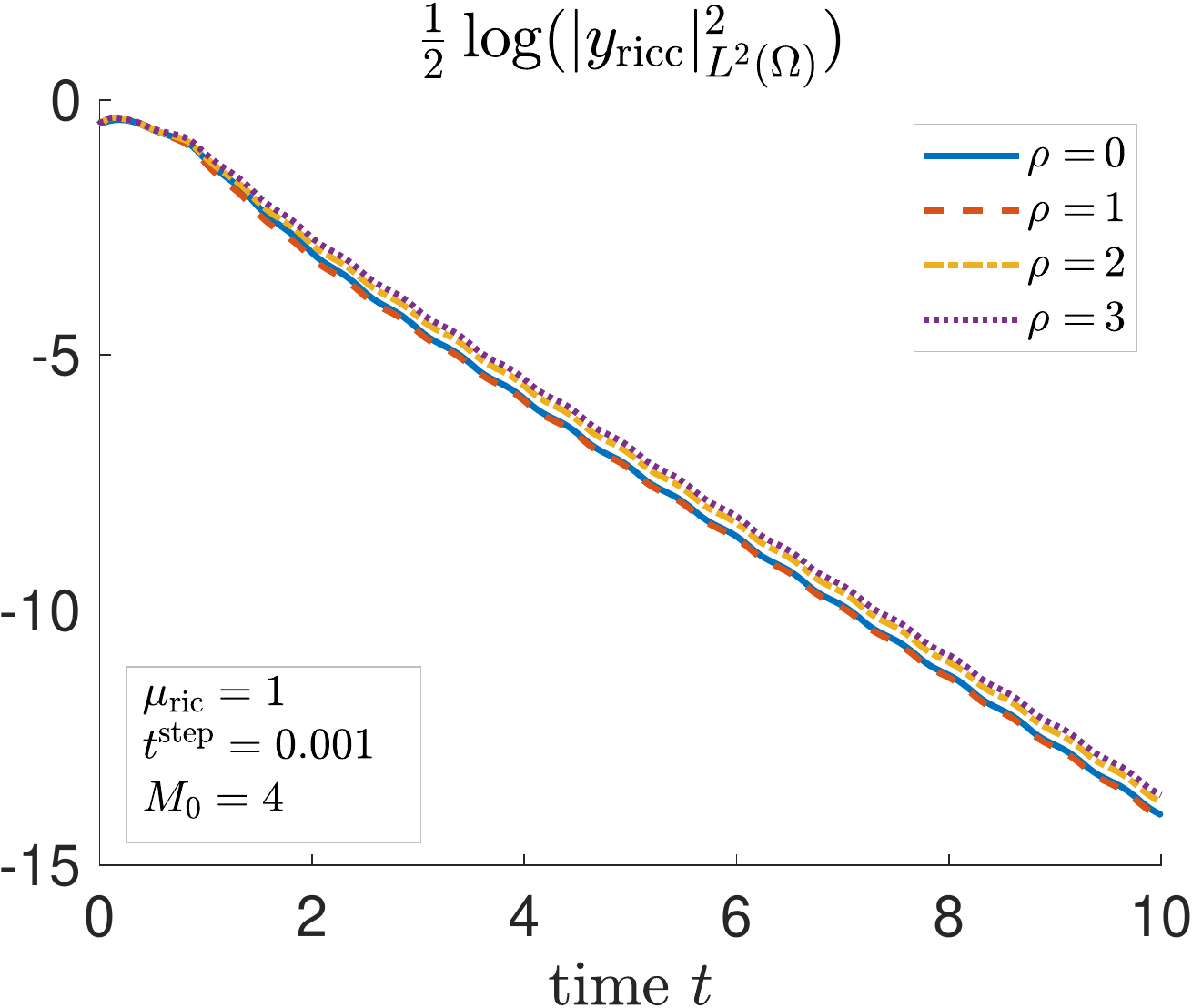}}
\quad
\subfigure
{\includegraphics[width=0.45\textwidth]{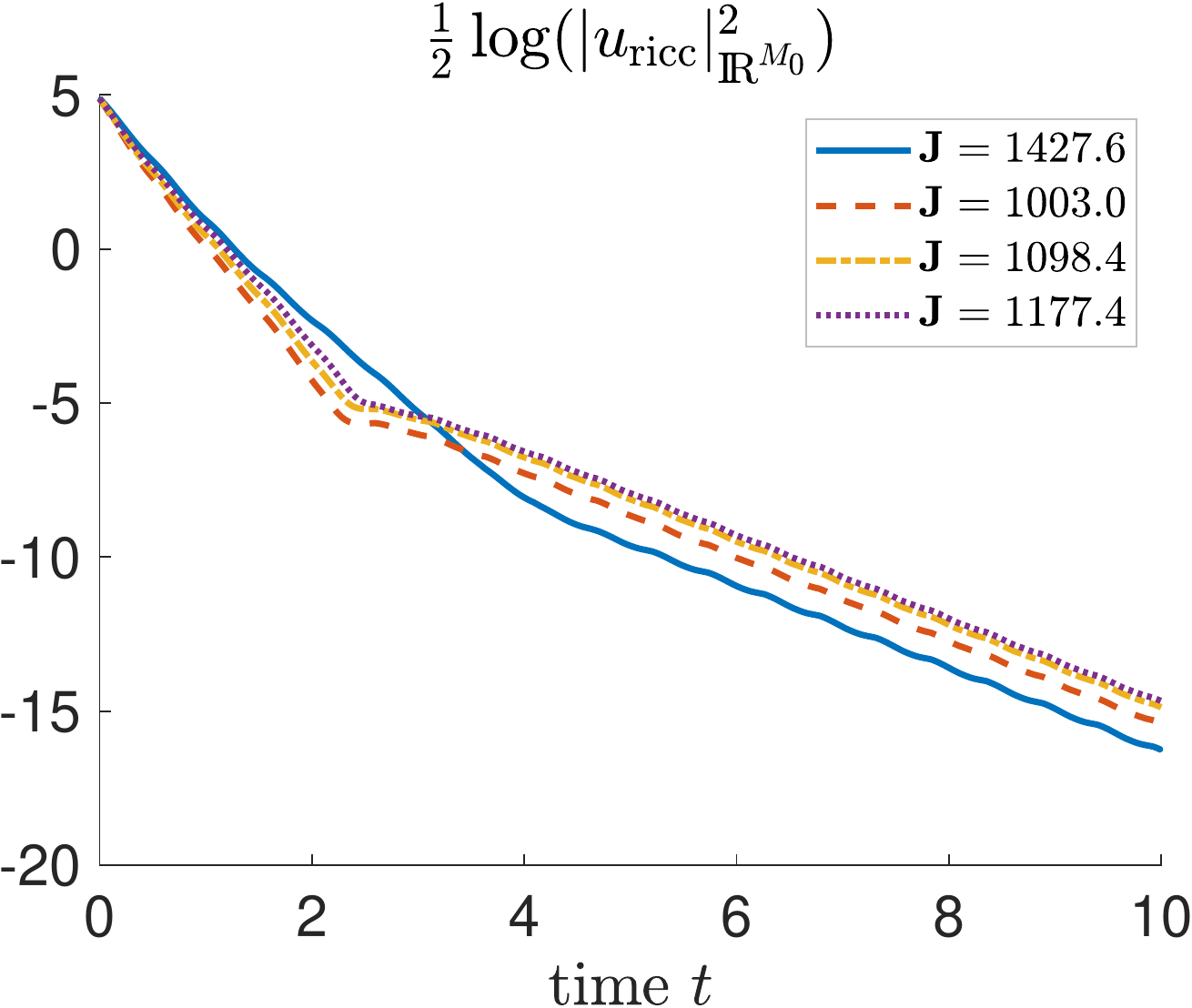}}
\caption{$M_0=4$. Riccati input feedback~\eqref{discu-ricc}.}
\label{fig:M4ricc}
\end{figure}
As expected, we also see that the quadratic cost functional~\eqref{costT} is smaller for Riccati.

Further, we see that an exponential stability rate~$\mu\le\lambda=1$ is provided by the  oblique projection
feedback and that an exponential stability rate~$\mu>\mu_{\rm ric}=1$ is provided by the Riccati feedback.
This confirms the theoretical results.

As we see in Fig.~\ref{fig:Mesh_coarse}, we have a rough  approximation of the~$4$ actuators in the coarse mesh at least when compared with the approximation after~$3$ refinements as in Fig~\ref{fig:Mesh_refs}. In spite of this fact, we still observe in Fig.~\ref{fig:M4ricc} that the Riccati feedback computed for the coarsest mesh also provides a stability rate~$\mu>\mu_{\rm ric}$ for the refined meshes. This is a first sign towards the validation the approach we propose of computing the Riccati input feedback  operator in coarse meshes and using it in refined meshes.

\subsection{Using nine actuators}
To strengthen the validation of the proposed approach, we consider next the case of~$9$ actuators, whose approximation in Fig.~\ref{fig:Mesh_coarse} is again rough when compared to the one obtained in the mesh after~$3$ refinements shown in Fig~\ref{fig:Mesh_refs}.
With~$9$ actuators we see, in Figures~\ref{fig:M9expl} and~\ref{fig:M9ricc}, that
both oblique projection and Riccati  feedbacks are able to stabilize the system. Again we see that an exponential stability rate~$\mu\le\lambda=1$ is provided by the  oblique projection
feedback and that an exponential stability rate~$\mu>\mu_{\rm ric}=1$ for both the coarsest and the refined meshes. 
\begin{figure}[ht]
\centering
\subfigure
{\includegraphics[width=0.45\textwidth]{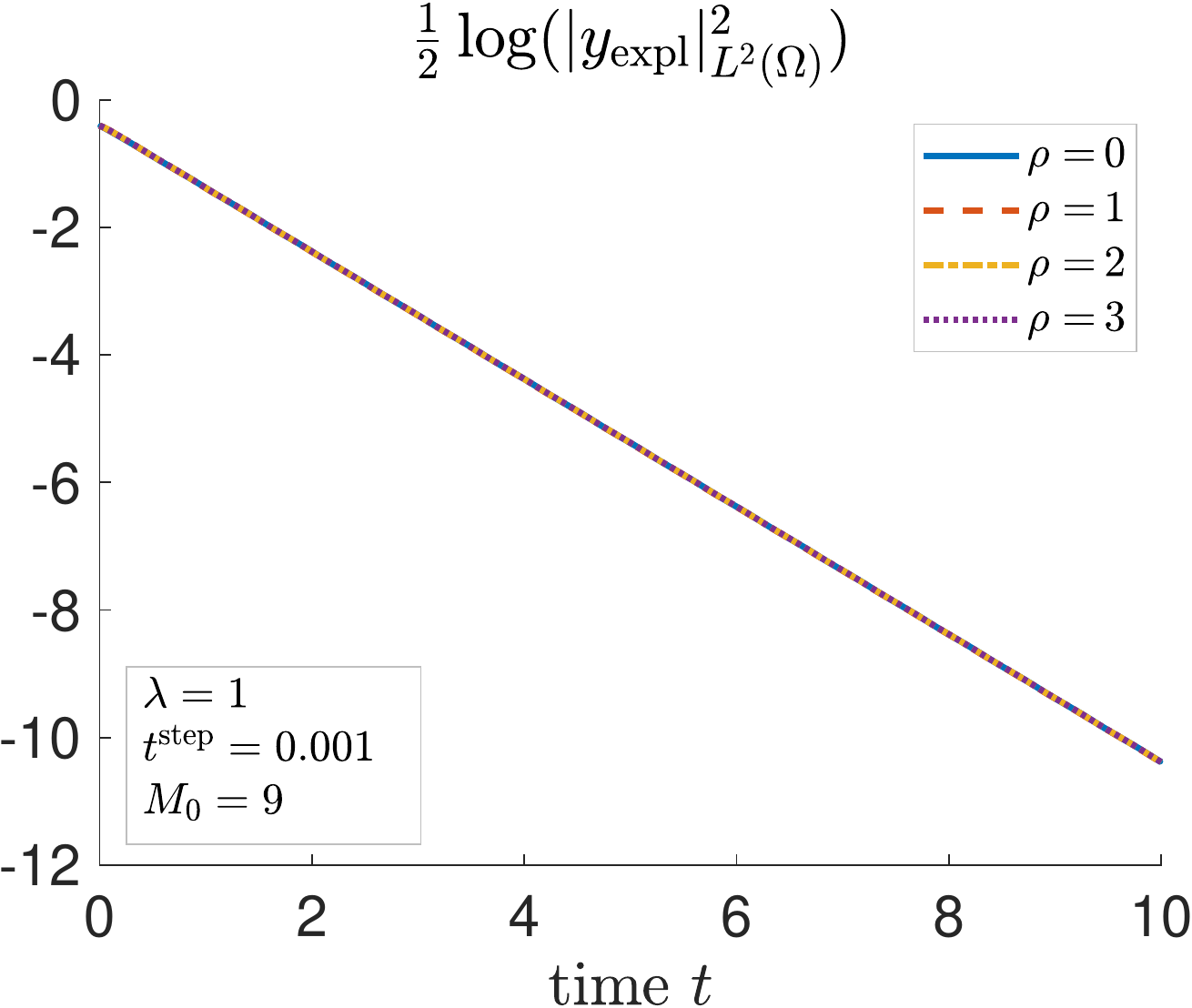}}
\quad
\subfigure
{\includegraphics[width=0.45\textwidth]{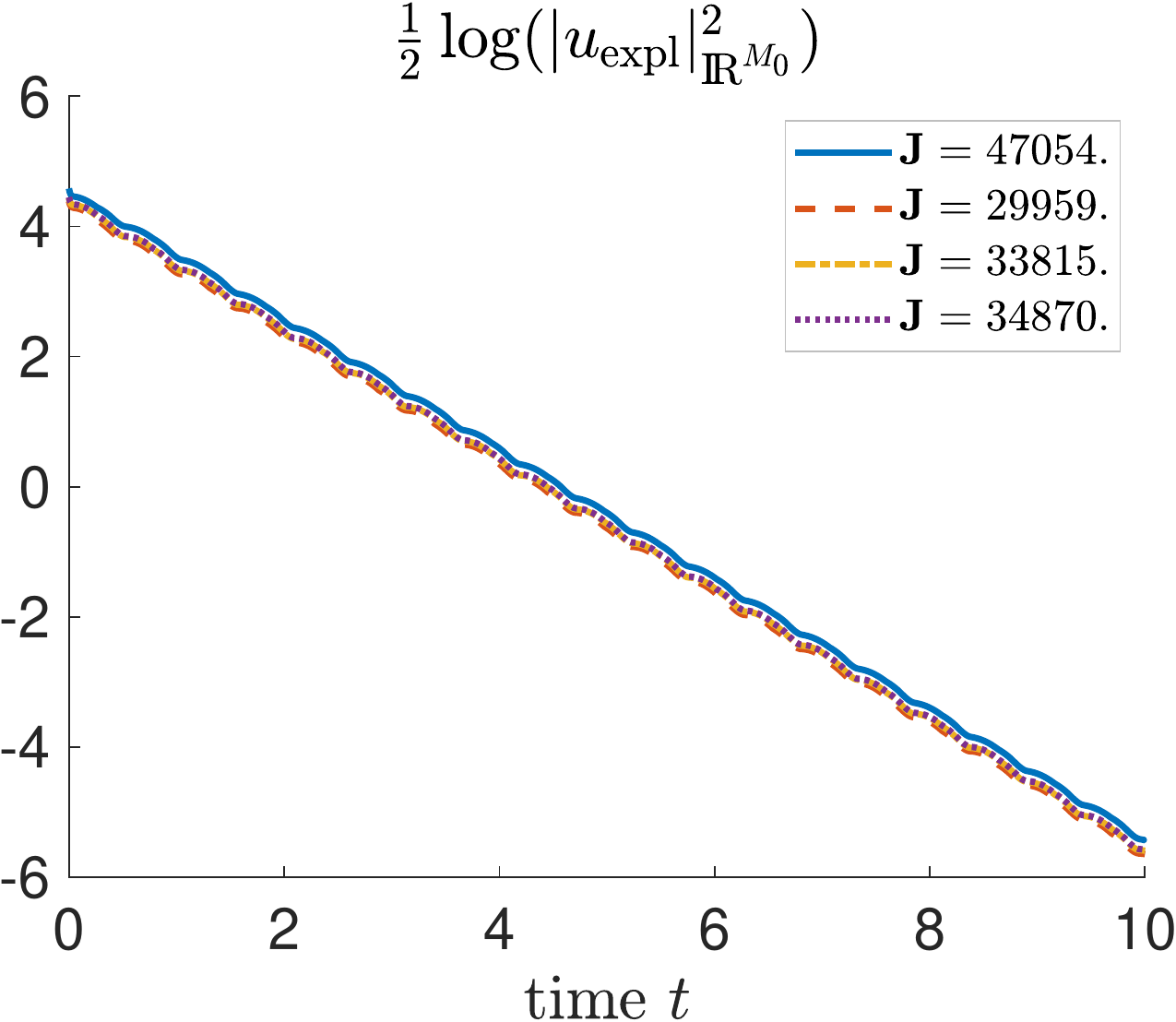}}
\caption{$M_0=9$. Oblique projection input feedback~\eqref{discu-obli}.\newline}
\label{fig:M9expl}
\subfigure
{\includegraphics[width=0.45\textwidth]{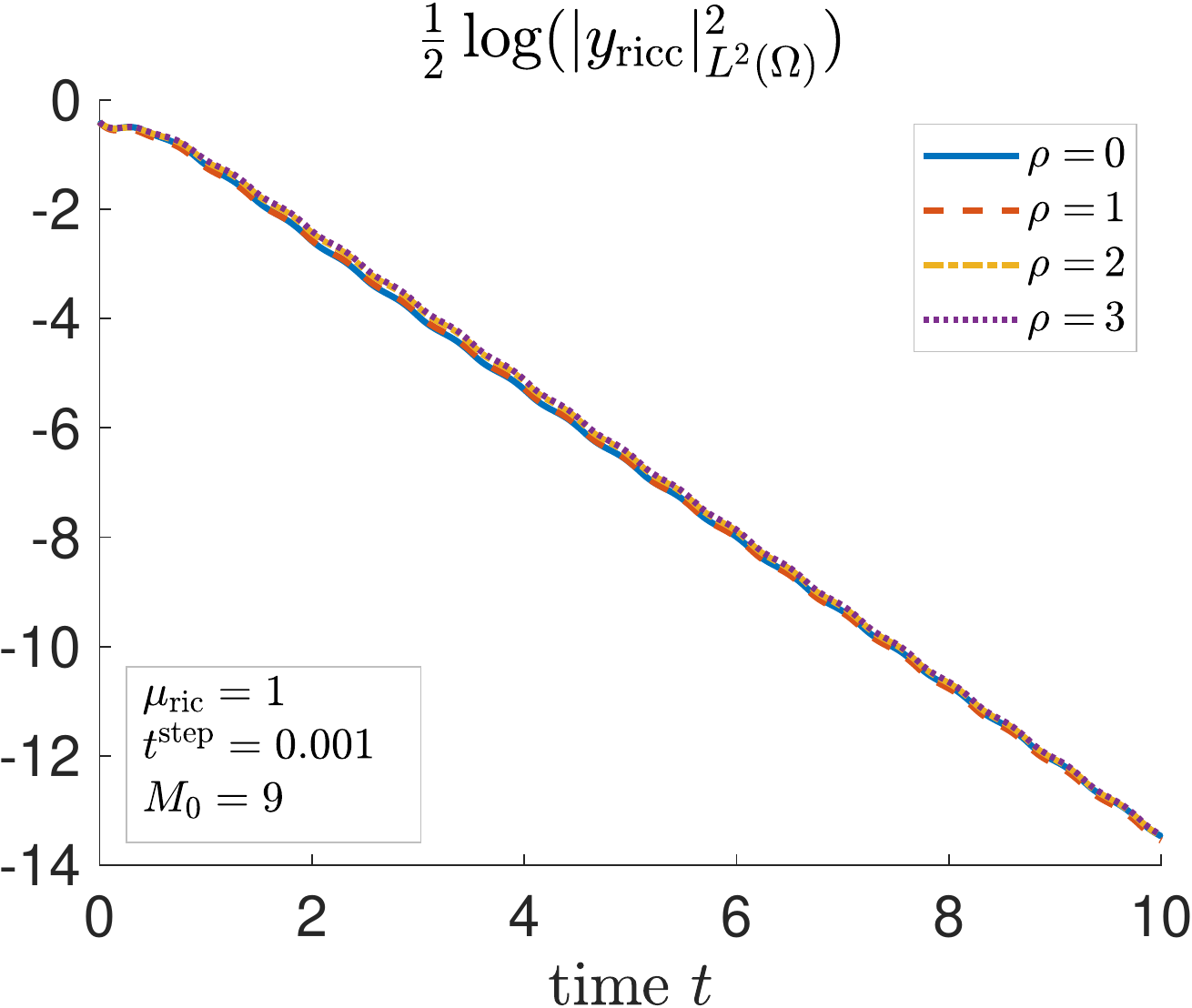}}
\quad
\subfigure
{\includegraphics[width=0.45\textwidth]{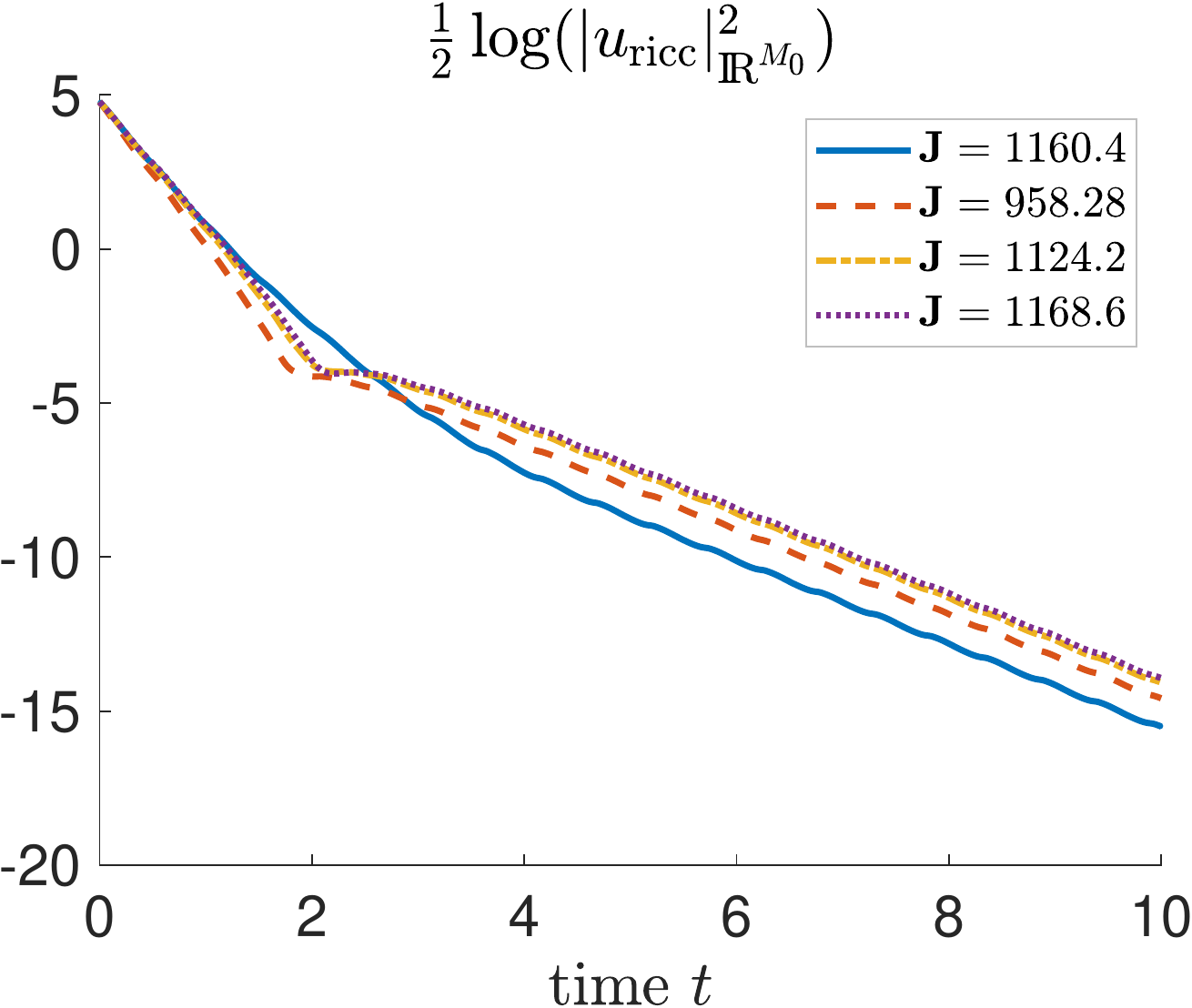}}
\caption{$M_0=9$. Riccati input feedback~\eqref{discu-ricc}.}
\label{fig:M9ricc}
\end{figure}
Furthermore, it is interesting to observe that, with the naked eye, we cannot see a difference on
the behavior of the norm of the state in Figure~\ref{fig:M9expl}.
This shows that with the coarsest mesh we obtain already an accurate behavior of the controlled dynamics. 
This could be partially explained from the fact that the dynamics of the projection~$z=P_{\clE_{M_0}}y$ is explicitly
imposed, $\dot z=-\lambda z$; see~\eqref{dyn-fin}.
Finally, we see that by taking a larger number~$M_0$ of actuators the quadratic
 cost decreases for both feedbacks, note that this is a
 nontrivial observation because, in particular, by construction (following~\cite[sect.~4.8.1]{KunRod19-cocv}), see~Fig.~\ref{fig:Mesh_refs}, last row)
 the total volume (area) covered by the actuators
is independent of the number of actuators.
Our spatial domain
is partitioned into~$M_0$ rescaled copies
of itself, and a rescaled actuator-subdomain is placed in each copy.

\subsection{Performance of algorithm solving the periodic Riccati equation}
In Fig.~\ref{fig:iterRicPer} we show the performance of iterative Algorithm~\ref{Alg:PRE}, by showing the evolution of the error until it reaches a value smaller than~${\rm tol}=(N{\tt eps})^\frac12\approx2.4845\times10^{-07}$.
The error converges exponentially to zero, which confirms the result in Theorem~\ref{T:convRicPer}.
It is interesting to see that, after a suitable number~$\underline n$ of iterations, the exponential rates for the cases $M_0\in\{1,4,9\}$ are close to each other, even likely the same with the naked eye as
\begin{align}
{\rm error}(n+1)&\approx\rme^{-0.4}{\rm error}(n)\approx\rme^{-0.4(n+1-\underline n)}{\rm error}(\underline n),\quad\mbox{for}\quad n\ge\underline n,\notag
\intertext{where we have denoted}
{\rm error}(n)&\coloneqq\norm{\Pi^{n}(\tau)-\Pi^{n}(\tau+\varpi)}{\clL(L^2(\Omega)}.\notag
\end{align}
\begin{figure}[ht]
\centering
\subfigure
{\includegraphics[width=0.7\textwidth]{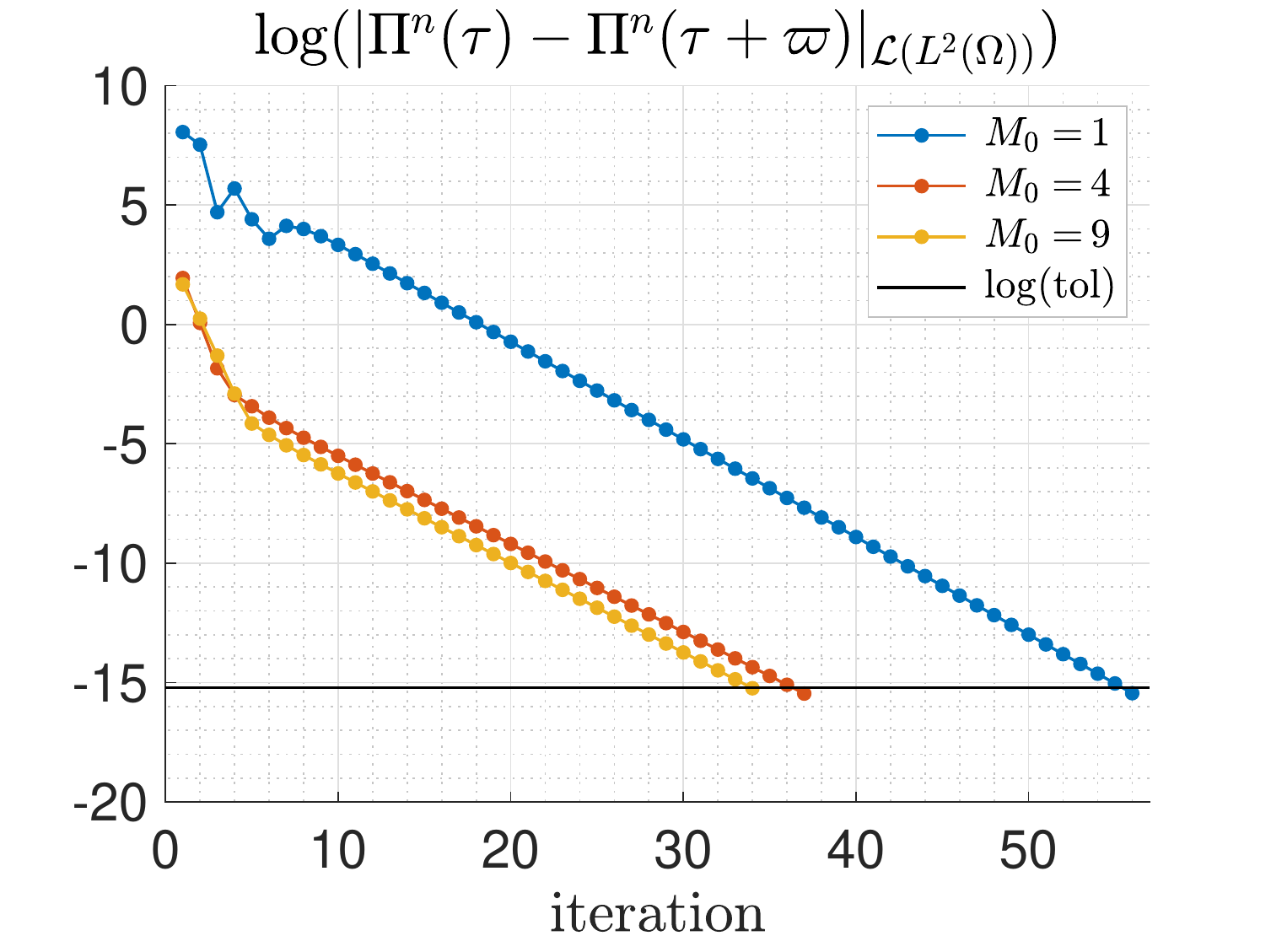}}
\caption{Error evolution for iterations of Algorithm~\ref{Alg:PRE}.}
\label{fig:iterRicPer}
\end{figure}

We would like to report that, in the case of~$4$ and~$9$ actuators the Riccati time-step~$k_{\rm ric}$ as in~\eqref{num-paramRic} turned out to be small enough so that step~\ref{Alg:DRE:squeze-kric} in Algorithm~\ref{Alg:DRE} was never activated. Instead, in the case of~$1$ actuators such step was often activated and  more than once for some time instants~$T$. That is, computing the optimal feedback for a single actuator took more time for each iteration of Algorithm~\ref{Alg:PRE}. Roughly speaking this may suggest that stabilization with a single actuator is (at least; cf.~section~\ref{sS:num.M=1}) more difficult (which somehow agrees with common sense).

\section{Further remarks}\label{S:remarks}
We give additional details and comments on the followed procedure.

\subsection{On the proposed strategy for solving the Riccati equations}\label{sS:tolARE}
For solving the algebraic Riccati equations
within Algorithms~\ref{Alg:ARE} and~\ref{Alg:DRE}, we have set the stopping criteria as
$\frac{\dnorm{\fkT_{(\bfY,\bfB,\bfC)}(\mathbf\Pi)}{}}{ \max\left\{1,\dnorm{\mathbf\Pi}{}\right\} }
<N^{\frac{1}{2}}{\tt eps}^{\frac{1}{2}}$, where~${\tt eps}\approx 10^{-16}$ is the Matlab epsilon/accuracy and~$\dnorm{\Bigcdot}{}\coloneqq\norm{\Bigcdot}{\clL(\bbR^N)}$. This tolerance value is the minimal one used/proposed in the
software~\cite{BennerRicSolver}, which we use hereafter.

Within Algorithm~\ref{Alg:PRE}, for solving the periodic Riccati equation, we have set again the tolerance~$\varepsilon=N^\frac12 {\tt eps}^\frac12$ and replaced the norm in~$\clL(H)$ by the norm in~$\clL(\bbR^N)$ for the discretized equations. 

Following Algorithm~\ref{Alg:DRE}, we solve an algebraic Riccati equation at each time step. 
In order to speed the computations up, we could naturally think of taking a further linear approximation for the nonlinear unknown term in~\eqref{Ric-discj-sch}, namely, we could take~$\bfZ_1\coloneqq\frac12(
\mathbf\Pi^{r+1}\bfB \bfB^\top \mathbf\Pi^{r}+\mathbf\Pi^{r}\bfB \bfB^\top \mathbf\Pi^{r+1})$ instead of~$\bfN\coloneqq\mathbf\Pi^{r}\bfB \bfB^\top \mathbf\Pi^{r}$. By taking such~$\bfZ_1$ (linear on~$\mathbf\Pi^{r}$) we would need to solve a single Lyapunov equation at each time step which would likely be cheaper/faster. An alternative could be to take an Adams--Bashforth linear extrapolation~$\bfZ_0\coloneqq 
2\mathbf\Pi^{r+1}\bfB \bfB^\top \mathbf\Pi^{r+1}-\mathbf\Pi^{r+2}\bfB \bfB^\top \mathbf\Pi^{r+2}$ (independent of~$\mathbf\Pi^{r}$) instead of~$\bfN$.
However, by using an extra approximation for~$\bfN$ we will induce an extra error which will back-propagate over the time interval~$[\tau,\tau+\varpi]$, which we want to avoid (or minimize) when looking for the time $\varpi$-periodic solution of the time $\varpi$-periodic Riccati equation; see Algorithm~\ref{Alg:PRE}.

\subsection{On the Lyapunov equations}
 It is not our goal to discuss details on the numerical solution of the Lyapunov equation~\eqref{Lyap-iter}, but since it plays a crucial role in the solution of the algebraic Riccati equation, we refer the interested reader to the survey~\cite[sect.~5]{BennerSaak13-gamm} where an ADI (alternating direction implicit) based iterative method is proposed for finding low-rank representations/approximations for the solution of Lyapunov matrix equations, and also to the related work~\cite{OpmeerReisWollner13} concerning Lyapunov operator equations in infinite-dimensional spaces and references therein. See also~\cite{LuWachspress91} for discussions on other methods. In this manuscript the Lyapunov equations were solved in factorized form with the matrix sign function~\cite[sect.~IV]{Benner06}.

\subsection{On generalized Riccati equations}
Note that in~\eqref{DiscRicc} we require the computation of the matrix~$\bfX$ involving the inverse of the mass matrix, this is not a problem for the coarse discretizations we  used to compute the solutions~$\mathbf\Pi$ of the Riccati equations. In case we want or need to compute such solutions~$\mathbf\Pi$ in fine discretizations, where we will have larger matrices, the inverse of the mass matrix can be an issue. In that case, an option to overcome this issue could be computing first
the product~$\mathbf\Xi=\bfM^{-1}\mathbf\Pi\bfM^{-1}$, and then recover~$\mathbf\Pi=\bfM\mathbf\Xi\bfM$. Note that from~\eqref{DiscRicc} we find that~$\mathbf\Xi$ solves a more general equation as follows
\[
\bfM\dot{\mathbf\Xi}\bfM+\widehat\bfX^\top \mathbf\Xi\bfM+\bfM\mathbf\Xi \widehat\bfX
-\bfM\mathbf\Xi\widehat\bfB \widehat\bfB^\top \mathbf\Xi\bfM
+\bfC^\top\bfC=0,
\]
where~$\widehat\bfX=\bfM\bfX=-(\bfS_\nu +\bfL^0+\bfL^1-\overline\mu\bfM)$ and~$\widehat\bfB=\bfM\bfB$ (cf.~\cite[Equ.~(15)]{MalqvistPerStill18}, \cite[below Equ.~(8)]{BreitenDolgovStoll-na20}, \cite[Equ.~(4.8)]{Heiland16}). Thus, we can likely avoid the issues associated with the inverse of~$\bfM$,  with the expense of computing the solution of a more general equation.
In any case, the solution of the Riccati equation~$\mathbf\Pi$ is a full matrix at each instant of time time, for positive definite~$\bfC^\top\bfC$, so we have anyway a constraint in the size of~$\mathbf\Pi$. For the algebraic Riccati equation, in the case~$\bfC$ has small rank (when compared to the size of~$\bfM$)  we can expect that~$\mathbf\Pi$ is well approximated by products as~$\mathbf\Pi_\rmf^\top\mathbf\Pi_\rmf$ where~$\mathbf\Pi_\rmf$ has a small rank as well. For the solution of the differential Riccati equations this low-rank phenomenon is less clear. In any case, if~$\mathbf\Pi$ is not necessarily positive definite, then we may need a different approach in Algorithm~\ref{Alg:DRE}, where we have exploited the fact that~$\mathbf\Pi$ is positive definite to guarantee the positive definiteness of~$\bfQ$ for small time-step; see~\eqref{Ric-discj-Q}.

\section{Conclusions}\label{S:conclusion}
At the theoretical level, we have shown that the explicit oblique projections based feedback operator introduced in~\cite{KunRod19-cocv} is able to stabilize parabolic equations with reaction-convection terms~$A_{\rm rc}$ taking values in~$\clL(H,V')+\clL(V,H)$. We have also shown that the solution of the time-periodic Riccati equation can be found by an iterative process. 
At the numerical level, we have discussed general aspects from the finite-elements numerical implementation of stabilizing feedbacks, as
the classical Riccati based feedbacks and  oblique
projections based feedbacks. The stabilizing performance of such feedbacks has been illustrated by results of simulations.

\subsection{Oblique projections as an alternative to Riccati} Oblique projections
feedbacks are an interesting alternative to the Riccati feedbacks, since they are easier to
compute and implement numerically, and because we do not need to save
the solution of the differential Riccati equation, prior to simulations. In particular, the
oblique projections feedback input can be computed online in real time.
Another disadvantage of Riccati based feedbacks is that its computation
is unfeasible for general nonautonomous systems
in the entire unbounded time interval. Thus we  have restricted the numerical computations
to the case of a time-periodic reaction-convection terms, where we can reduce the computations to a finite time interval with length equal to the time-period.  

\subsection{Oblique projections as a nonalternative to Riccati}
An  advantage of the Riccati based feedbacks is that it is less sensitive
with respect to the number and location of the actuators,
and it further minimizes the total spent energy (classical quadratic cost).
So, if the minimization of the spent energy  is important/asked in a given
application and/or if we do not have an enough number of actuators at our disposal, then  the oblique projections
based feedbacks may  be not an alternative to the Riccati based ones. In such case, we have to face the fact that
for fine discretizations
it is difficult, and maybe unfeasible, to compute and save the entire
array with the solution of the input Riccati operator for each discrete instant of time, in the fixed (large)
bounded time interval (e.g., for a large time-period).
To circumvent this issue, we propose to compute the Riccati feedback for a coarse mesh,
and use it to
construct an ``extended'' feedback allowing us to perform simulations in
appropriately refined meshes,
in both spatial and temporal domains. We presented simulations showing that such strategy provides us 
with a stabilizing feedback.

\subsection{Open questions. Possible future works}
Concerning Riccati feedbacks, it is clear that the coarsest
spatial and temporal meshes must be ``fine enough''.
It would be interesting to investigate this point in order to quantify ``how fine'' such meshes must be taken.
This is expected to depend on the given system (free) dynamics.    
We have used Algorithm~\ref{Alg:PRE} to compute the solution of the time-periodic operator differential
Riccati equation. This approach will be expensive (time consuming)
if the time-period~$\varpi$ is large. 
Though we expect the error in Algorithm~\ref{Alg:PRE} to convergence exponentially to zero (cf. Thm.~\ref{T:convRicPer} and Fig.~\ref{fig:iterRicPer}), the  exponential rate can be relatively small and we will need a large number of iterations. Thus it would be interesting to know more about such rate. In~\cite{GusevJohaKagsShirVarga10} the authors are able to compute the solution for large time-periods in a relatively short time. Unfortunately, from the results reported in~\cite[sect.~5]{GusevJohaKagsShirVarga10} the methods evaluated/compared in the same reference are likely not appropriate for computing the periodic matrix Riccati solution for matrices as large as those coming from finite-element discretizations of partial differential equations; the sizes of the matrices considered~\cite[sect.~5.1.1, Table~2, and sect.~5.1.2]{GusevJohaKagsShirVarga10} are far from the number of nodes of the mesh in Fig.~\ref{fig:Mesh_coarse}.

 \medskip\noindent
 {\bf Acknowledgments.} 
The author acknowledges partial support from the Upper Austria Government and 
the Austrian Science
Fund (FWF): P 33432-NBL.

\renewcommand*{\bibfont}{\normalfont\footnotesize}
\bibliographystyle{plainurl}
\bibliography{Expl-Ricc}

\end{document}